\documentclass[graybox]{svmult}

\usepackage{amsfonts,amsmath,amstext,latexsym,amssymb,physics}
\usepackage{mathptmx}       % selects Times Roman as basic font
\usepackage{helvet}         % selects Helvetica as sans-serif font
\usepackage{courier}        % selects Courier as typewriter font
\usepackage{type1cm}        % activate if the above 3 fonts are
\usepackage{newtxtext}       %
\usepackage{newtxmath}       % selects Times Roman as basic font
\usepackage{graphicx}        % standard LaTeX graphics tool
\usepackage{multicol}        % used for the two-column index
\usepackage[bottom]{footmisc}% places footnotes at page bottom

\usepackage{cite}
\usepackage{bm}

\usepackage{enumitem}  % http://ctan.org/pkg/enumitem
\usepackage{calc}% http://ctan.org/pkg/calc
\setlist{labelindent=1pt,itemsep=0.1cm}
\setlist[itemize]{leftmargin=0.7cm}
\setlist[enumerate]{itemindent=0em,leftmargin=0.7cm}

\usepackage{float} % for placement of tables and figures %

\usepackage{marginnote}

\usepackage[bookmarks=false,hidelinks,unicode]{hyperref}
\DeclareMathOperator{\esssup}{ess\;sup}
\newcommand*{\QEDB}{\hfill\ensuremath{\square}}%

\smartqed
\allowdisplaybreaks
\smartqed

\begin{document}
\title*{Representations of polynomial covariance type commutation relations by linear integral operators with separable kernels in $L_p$}
\titlerunning{Representations  by  Linear integral operators with separable kernels }        % if too long for running head

\author{Domingos Djinja \and Sergei Silvestrov \and Alex Behakanira Tumwesigye}
\authorrunning{D. Djinja, S. Silvestrov, A. B. Tumwesigye} % if too long for running head

\institute{Domingos Djinja \at
Department of Mathematics and Informatics, Faculty of Sciences, Eduardo Mondlane University, Box 257, Maputo, Mozambique \\
\email{domingos.djindja@uem.ac.mz}
\at
Division of Mathematics and Physics, School of Education, Culture and Communication, M{\"a}lardalen University, Box 883, 72123 V{\"a}ster{\aa}s, Sweden \\
\email{domingos.celso.djinja@mdu.se}
\and
Sergei Silvestrov  (corresponding author)
\at Division of Mathematics and Physics, School of Education, Culture and Commu\-nication, M{\"a}lardalen University, Box 883, 72123 V{\"a}ster{\aa}s, Sweden. \\
\email{sergei.silvestrov@mdu.se}
\and
Alex Behakanira Tumwesigye \at
Department of Mathematics, College of Natural Sciences, Makerere University, Box 7062, Kampala, Uganda. \email{alex.tumwesigye@mak.ac.ug}
}


%%%%%%%%%\index[aut]{Djinja, Domingos}
%%%%%%%%%\index[aut]{Silvestrov, Sergei}
%%%%%%%%%\index[aut]{Tumwesigye, Alex Behakanira}

\maketitle
%%%%%%%%%%\label{chap:OkekeAbbasSilvestrov}

\abstract*{
Representations of polynomial covariance type commutation relations by linear integral operators on $L_p$ over measures spaces  are investigated. Necessary and sufficient conditions for integral operators to satisfy polynomial covariance type commutation relations are obtained in terms of their kernels. For important classes of polynomial  covariance commutation relations associated to arbitrary monomials and to affine functions, these conditions on the kernels are specified in terms of the coefficients of the monomials and affine functions.
By applying these conditions, examples of integral operators on $L_p$ spaces, with separable kernels  representing  covariance commutation relations associated to monomials, are constructed for the kernels involving multi-parameter trigonometric functions, polynomials and Laurent polynomials on bounded intervals. Commutators of these operators are computed and exact conditions for commutativity of these operators in terms of the parameters are obtained.
}

\abstract{
Representations of polynomial covariance type commutation relations by linear integral operators on $L_p$ over measures spaces  are investigated. Necessary and sufficient conditions for integral operators to satisfy polynomial covariance type commutation relations are obtained in terms of their kernels. For important classes of polynomial  covariance commutation relations associated to arbitrary monomials and to affine functions, these conditions on the kernels are specified in terms of the coefficients of the monomials and affine functions.
By applying these conditions, examples of integral operators on $L_p$ spaces, with separable kernels  representing  covariance commutation relations associated to monomials, are constructed for the kernels involving multi-parameter trigonometric functions, polynomials and Laurent polynomials on bounded intervals. Commutators of these operators are computed and exact conditions for commutativity of these operators in terms of the parameters are obtained.
}

\section{Introduction}
Commutation relations \index{commutation relation} of the form
\begin{equation} \label{CovrelABeqBFA}
  AB=B F(A)
\end{equation}
where $A, B$ are elements of an associative algebra and  $F$ is a function of the elements of the algebra, are important in many areas of Mathematics and applications. Such commutation relations are usually called covariance relations, crossed product relations or semi-direct product relations. A pair $(A,B)$ of elements of an algebra that satisfies \eqref{CovrelABeqBFA} is called a representation of this relation \cite{Samoilenkobook}.

Representations of covariance commutation relations \eqref{CovrelABeqBFA} by linear operators are important for study of actions and induced representations of groups and semigroups, crossed product operator algebras, dynamical systems, harmonic analysis, wavelets and fractal analysis and, hence have applications in physics and engineering
\cite{BratJorgIFSAMSmemo99,BratJorgbook,JorgWavSignFracbook,JorgOpRepTh88,JorMoore84,MACbook1,MACbook2,MACbook3,OstSambook,Pedbook79,Samoilenkobook}.

A description of the structure of representations for the relations of the form \eqref{CovrelABeqBFA}
by bounded and unbounded self-adjoint operators, normal operators, unitary operators, partial isometries
and other linear operators with special involution conditions on a Hilbert space, have been considered in
\cite{BratEvansJorg2000,CarlsenSilvExpoMath07,CarlsenSilvAAM09,CarlsenSilvProcEAS10,DutkayJorg3,DJS12JFASilv,DLS09,DutSilvProcAMS,DutSilvSV,JSvT12a,JSvT12b,Mansour16,JMusondaPhdth18,JMusonda19,Musonda20,Nazaikinskii96,OstSambook,PerssonSilvestrov031,PerssonSilvestrov032,PersSilv:CommutRelDinSyst,RST16,RSST16,Samoilenkobook,SaV8894,SilPhD95,STomdynsystype1,SilWallin96,SvSJ07a,SvSJ07b,SvSJ07c,SvT09,Tomiyama87,Tomiama:SeoulLN1992,Tomiama:SeoulLN2part2000,AlexThesis2018,VaislebSa90}.
using reordering formulas for functions of the algebra elements and operators satisfying covariance commutation relation,
functional calculus and spectral representation of operators and interplay with dynamical systems or iterated function systems generated
by iteration of the maps involved in the commutation relations.

In this paper, we construct representations of the covariance commutation relations \eqref{CovrelABeqBFA} by linear integral operators on Banach spaces $L_p$ over measure spaces.
%When $B=0$, the relation \eqref{CovrelABeqBFA} is trivially satisfied for any $A$. Thus, we focus on construction and properties of nontrivial representations of \eqref{CovrelABeqBFA}.
 We consider representations by the linear integral operators defined by kernels satisfying different conditions. We derive conditions on such kernel functions so that the corresponding operators satisfy \eqref{CovrelABeqBFA} for polynomial $F$ when both operators  are of linear integral type.
 Representations of polynomial covariance type commutation relations by linear integral operators on $L_p$ over measures spaces  are constructed. In contrast to \cite{OstSambook,Samoilenkobook,SaV8894,VaislebSa90} devoted to involutive representations of covariance type relations by operators on Hilbert spaces using spectral theory of operators on Hilbert spaces,
we aim at direct construction of various classes of representations of covariance type relations in specific important classes of operators on Banach spaces more general than Hilbert spaces without imposing any involution conditions and not using classical spectral theory of operators. Conditions for such representations are described in terms of kernels of the corresponding integral operators. Necessary and sufficient conditions for integral operators to satisfy polynomial covariance type commutation relations are obtained in terms of their kernels. For important classes of polynomial  covariance commutation relations associated to arbitrary monomials and to affine functions, these conditions on the kernels are specified in terms of the coefficients of the monomials and affine functions. By applying these conditions, examples of integral operators on $L_p$ spaces, with separable kernels  representing  covariance commutation relations associated to monomials, are constructed for the kernels involving multi-parameter trigonometric functions, polynomials and Laurent polynomials on bounded intervals. Commutators of these operator are computed and exact conditions for commutativity of these operators in terms of the parameters are obtained.

%Representations of polynomial covariance type commutation relations by linear integral operators on $L_p$ over measures spaces  are investigated.
%Conditions for such representations are described in terms of kernels of the corresponding integral operators. %Representation by integral operators are studied both for general polynomial covariance commutation relations and %for important classes of polynomial covariance commutation relations associated to arbitrary monomials and to affine %functions. Examples of integral operators with separable kernels on $L_p$ spaces representing the covariance %commutation relations are constructed. Sequences of non-commuting operators which satisfy monomial covariance %commutation relation and converge to commuting operators are presented.

This paper is organized in four sections. After the introduction,  we present in Section \ref{SecPreNot} some preliminaries, notations, basic definitions and an useful lemma.
In Section \ref{SecRepreBothLI}, we present representations when both operators $A$ and $B$ are linear integral operators with separable kernels acting on $L_p$ spaces. In Section \ref{SecRepreBothLIExample} we construct examples.

\section{Preliminaries and notations}\label{SecPreNot}
In this section we provide the preliminaries, basic definitions and notations needed in this article, referring to    \cite{AdamsG,AkkLnearOperators,BrezisFASobolevSpaces,FollandRA,HutsonPym,Kantarovitch,KolmogorovFominETFFAbook,KolmogorovFominETFFAbookVol2,KrasnolskZabreyko,RudinRCA}.
for further reading.

Let $\mathbb{R}$ be the set of all real numbers, $\mathbb{Z}$ the set of all integers, $\mathbb{N}$ the set of all positive integers and $ X$ a non-empty set.
Let  $(X,\Sigma, \mu)$ be a $\sigma$-finite measure space, where $\Sigma$ is a $\sigma-$algebra with measurable subsets of $X$,
and $X$ can be covered with at most countable many disjoint sets $E_1,E_2,E_3,\ldots$ such that $ E_i\in \Sigma, \,
\mu(E_i)<\infty$, $i=1,2,\ldots$  and $\mu$ is a measure.
For $1\leq p<\infty,$ we denote by $L_p(X,\mu)$, the set of all classes of equivalent (different on a set of zero measure)
measurable functions $f:X\to \mathbb{R}$ such that
$\int\limits_{X} |f(t)|^p d\mu < \infty.$
This is a Banach space (Hilbert space when $p=2$) with norm
$\| f\|_p= \left( \int\limits_{X} |f(t)|^p dt \right)^{\frac{1}{p}}.$
We denote by $L_\infty(X,\mu)$ the set of all classes of equivalent measurable functions $f:X\to \mathbb{R}$ such that exists $\lambda >0$,
$|f(t)|\leq \lambda$
almost everywhere. This is a Banach space with norm
$\displaystyle \|f\|_{\infty}=\mathop{\esssup}_{t\in X} |f(t)|.$
The support of a function $f:\, X\to\mathbb{R}$ is  ${\rm supp }\, f = \{t\in X \colon \, f(t)\not=0\}.$

Let $(\mathbb{R},\Sigma,\mu)$ be the standard Lebesgue measure space. We will use the following notation
\begin{equation}\label{QGpairingDefinition}
  Q_{G}(u,v)=\int\limits_{G} u(t)v(t)d\mu
\end{equation}
where $G\in \Sigma$ and $u,v$ are such functions $u,v:\, G\to \mathbb{R} $ that integral on the right hand side exist and is finite.

We will consider a useful lemma for  integral operators which will be used throughout the article.
\begin{lemma}\label{LemmaAllowInfSetsEqLp}
Let $(X,\Sigma,\mu)$ be a $\sigma$-finite measure space. Let $f\in L_q(G_1,\mu)$, $g\in L_q(G_2,\mu)$ for
$1\leq q\leq\infty$ and let $G_1,G_2\in \Sigma$,  $i=1,2$. Set $G=G_1\cap G_2$. Then the following statements are equivalent\textup{:}
\begin{enumerate}[label=\textup{\arabic*.}, ref=\arabic*]
  \item \label{LemmaAllowInfSetsEqLp:cond1} For all $x\in L_p(X,\mu)$, $1\leq p \leq \infty$ such that $\displaystyle \frac{1}{p}+\frac{1}{q}=1$,
\begin{equation*} % \label{LemmaAllowInfSetsEqLpQLinearFunctinal}
 Q_{G_1}(f,x)= \int\limits_{G_1} f(t)x(t)d\mu=\int\limits_{G_2} g(t)x(t)d\mu=Q_{G_2}(g,x).
\end{equation*}
\item \label{LemmaAllowInfSetsEqLp:cond2} The following conditions hold\textup{:}
    \begin{enumerate}[label=\textup{\alph*)}]
      \item for almost every $t\in G$, $f(t)=g(t)$,
      \item for almost every $t \in G_1\setminus G,\ f(t)=0,$
      \item  for almost every $t \in G_2\setminus G,\ g(t)=0.$
    \end{enumerate}
\end{enumerate}
\end{lemma}

\begin{proof} \smartqed
\noindent \ref{LemmaAllowInfSetsEqLp:cond2}$\,\Rightarrow \,$\ref{LemmaAllowInfSetsEqLp:cond1}\ By additivity of $\mu$ we have
\begin{eqnarray*}
 \int\limits_{G_1} f(s)x(s)d\mu_s &=& \int\limits_{G_1\setminus G} f(s)x(s)d\mu_s+\int\limits_{G} f(s)x(s)d\mu_s=
 \int\limits_{G} g(s)x(s)d\mu_s\\
 &=&\int\limits_{G_2\setminus G} g(s)x(s)d\mu_s+\int\limits_{G} g(s)x(s)d\mu_s=\int\limits_{G_2} g(s)x(s)d\mu_s.
\end{eqnarray*}
\noindent \ref{LemmaAllowInfSetsEqLp:cond1}$\,\Rightarrow $ \ref{LemmaAllowInfSetsEqLp:cond2}
 Suppose that \ref{LemmaAllowInfSetsEqLp:cond1} is true. Let $1\le p <\infty$. By applying H\"older inequality we conclude that the following linear functionals $H_f: L_p(X,\mu)\to \mathbb{R},$ $H_g: L_p(X,\mu)\to\mathbb{R},$ defined by
  \begin{equation*}
   H_f(x)=\int\limits_{G_1} f(s)x(s)d\mu,\ H_g(x)= \int\limits_{G_2} g(s)x(s)d\mu,
  \end{equation*}
  are continuous. In fact, by applying H\"older inequality we have for all $x\in L_p(X)$
   \begin{align*}
     |H_f(x)| =&\left|\ \int\limits_{G_1} I_{G_1}(s)f(s)x(s)d\mu\right|\, \leq \left( \int\limits_{X} |I_{G_1}(s)f(s)|^q d\mu\right)^{1/q} \left( \int\limits_{X} |x(s)|^pd\mu\right)^{1/p}\\
     =&\left(\ \int\limits_{G_1} |f(s)|^q d\mu\right)^{1/q} \left(\, \int\limits_{X} |x(s)|^pd\mu\right)^{1/p} \leq \|f\|_{L_q} \|x\|_{L_p},
   \end{align*}
   where $f(\cdot) \in L_q(G_1,\mu)$ by hypothesis, $1 < q \leq \infty$, $\frac{1}{p}+\frac{1}{q}=1$. Similarly, one proves that $H_g$ is continuous.
   Therefore, by Riesz representation theorem of linear functionals \cite[Theorem 4.11]{BrezisFASobolevSpaces} (for $1<p<\infty$) and \cite[Theorem 4.14]{BrezisFASobolevSpaces} (for $p=1$ ) there exists $h\in L_q(X,\mu)$, $1 <q\leq \infty$ such that $\frac{1}{p}+\frac{1}{q}=1$ and for all $x\in L_p(X,\mu)$
\begin{equation}\label{ProofLemmaExtThmRepresentLp}
  \int\limits_{X} h(s)x(s)d\mu=\int\limits_{G_1} f(s)x(s)d\mu=\int\limits_{G_2} g(s)x(s)d\mu.
\end{equation}
Since this representation is unique we have
\begin{align*}
 h(s)=\, &I_{G}(s)f(s)+I_{G_1\setminus G}(s)f(s)+I_{G_2\setminus G}(s)g(s)
 \\ =\, &I_{G}(s)g(s)+I_{G_1\setminus G}(s)f(s)+I_{G_2\setminus G}(s)g(s),\  G=G_1\cap G_2.
\end{align*}
This implies that
$I_G(s)h(s)=I_G(s)f(s)=I_G(s)g(s)$. The function $h$ belongs to $L_q(X,\mu)$ since $f\in L_q(G_1,\mu)$, $g\in L_q(G_2,\mu)$ and $L_q(G_i,\mu)$, $i=1,2$ is a vector space. From \eqref{ProofLemmaExtThmRepresentLp} we conclude that
$I_{X\setminus G_2}(s)h(s)=I_{X\setminus G_1}(s)h(s)=0$ almost everywhere.
Hence
\begin{eqnarray*}
  I_{G_1\setminus G}(s)f(s)&=&I_{G_1\setminus G}(s)h(s)=0 \mbox{ almost everywhere since } G_1\setminus G \subseteq X\setminus G_2,\\
  I_{G_2\setminus G}(s)g(s)&=&I_{G_2\setminus G}(s)h(s)=0 \mbox{ almost everywhere since } G_2\setminus G \subseteq X\setminus G_1.
\end{eqnarray*}

Let $p=\infty$. By taking
     the indicator function $x(t)=I_{H_1}(t)$ of the set $H_1=G_1\cup G_2$, we have
\[ \textstyle \int\limits_{G_1} f(t)x(t)d\mu=\int\limits_{G_2} g(t)x(t)d\mu=
\int\limits_{G_1} f (t)d\mu=\int\limits_{G_2} g (t)d\mu =\eta,
\]
where $\eta$ is a constant. Now by taking $x(t)=I_{G_1\setminus G}(t)$ we get
\[ \textstyle
\int\limits_{G_1} f(t)x(t)d\mu=\int\limits_{G_2} g(t)x(t)d\mu= \int\limits_{G_1\setminus G} f (t)d\mu=\int\limits_{G_2} g (t)\cdot 0d\mu =0.
\]
Then $\int\limits_{G_1\setminus G} f (t)d\mu =0$. Analogously by taking $x(t)=I_{G_2\setminus G}(t)$ we get
    $      \int\limits_{G_2\setminus G} g (t)d\mu=0$.
    We claim that $f(t)=0$ for almost every $t\in G_1\setminus G$ and
    $g(t)=0$ for almost every $t\in G_2\setminus G$.
    We take a partition $\bigcup S_i$ of the set  $[\alpha_1,\beta_1]\setminus G$
    such that $S_i\cap S_j\neq \emptyset$ for $i\neq j$ and each set $S_i$ has positive measure.
    For each $x(t)=I_{S_i}(t)$ we have
    $$
    \int\limits_{G_1} f(t)x(t)d\mu=\int\limits_{G_2} g(t)x(t)d\mu=
    \int\limits_{S_i} f (t)d\mu=\int\limits_{G_2} g (t)\cdot 0d\mu =0.
    $$
    Thus,
    $
          \int\limits_{S_i} f (t)d\mu=0.
    $
    Since we can choose arbitrary partition with positive measure on each of its elements,
    $
      f(t)=0 \ \mbox{ for almost every } t\in G_1\setminus G.
    $
    Analogously,
    $
      g(t)=0 \  \mbox{ for almost every } t\in G_2\setminus G.
    $
    Therefore,
    $$
      \eta = \int\limits_{G_1} f (t)d\mu=\int\limits_{G_2} g (t)d\mu =\int\limits_G f (t)d\mu=\int\limits_G g (t)d\mu.
    $$
    Then, for all function $x\in L_p(X,\mu)$ we have
    $
      \int\limits_G f (t)x(t)d\mu=\int\limits_G g (t)x(t)d\mu \Leftrightarrow   \int\limits_G [f(t)-g(t)]x(t)d\mu=0.
    $
    With $x(t)=\left\{\begin{array}{cc} 1, & \mbox{if } f(t)-g(t)>0, \\ -1,  & \mbox{ if } f(t)-g(t)<0,\,   \end{array}\right.$
    for almost every $t\in G$ and $x(t)=0$ for almost every $t\in X\setminus G$, we get
    $\int\limits_{G} |f(t)-g(t)|d\mu =0.$ This implies that $f(t)=g(t)$ for almost every $t\in G$.
 \QEDB
\end{proof}

\section{Representations by linear integral operators}\label{SecRepreBothLI}
Let $(X,\Sigma,\mu)$ be a  $\sigma$-finite measure space. In this section we consider  representations
of the covariance type commutation relation \eqref{CovrelABeqBFA} when both $A$ and $B$
are linear integral operators \index{integral operator} on Banach space $L_p(X,\mu)$ with $\ 1\le p\le\infty$, and separable kernels:
\begin{equation}\label{OpeqnsSKthmGenSeparedKnlsAB}
  (Ax)(t)= \int\limits_{G_A} \sum _{i=1}^{l_A} a_i(t)c_i(s)x(s)d\mu_s,\quad (Bx)(t)= \int\limits_{G_B}\sum_{j=1}^{l_B} b_j(t)e_j(s)x(s)d\mu_s,
\end{equation}
for almost every $t$, where the index in $\mu_s$ indicates the variable of integration, $G_A\in \Sigma$ and $G_B\in \Sigma$.
When $B=0$, the relation \eqref{CovrelABeqBFA} is trivially satisfied for any $A$. If $A=0$, then the relation \eqref{CovrelABeqBFA} reduces to $F(0)B=0$. This implies either ($F(0)=0$ and $B$ can be any well defined operator) or $B=0$. Thus, we focus on construction and properties of non-zero representations of \eqref{CovrelABeqBFA}.
We follow a similar argument to Folland \cite{FollandRA} to prove sufficient conditions for these operators to be well defined and bounded. If $a_i, b_j\in L_p(X,\mu)$, $c_i\in L_q(G_A,\mu)$, $e_j\in L_q(G_B,\mu)$, $i,j, l_A,l_B$ are positive integers,  $i=1,\ldots, l_A$, $j=1,\ldots, l_B$,  $1\leq q\leq\infty$, $\frac{1}{p}+\frac{1}{q}=1$,  then operators $A$ and $B$ are well defined and bounded. In fact, if we let
$(A_ix)(t)=\int\limits_{G_A} a_i(t)c_i(s)x(s)d\mu_s$ for almost every $t$ and $(Ax)(t)=\sum\limits_{i=1}^{l_A} (A_ix)(t)$, for almost every $t\in X$, and if $1<p<\infty$, then
by H\"older inequality
we have for all $x\in L_p(X,\mu)$,  $c_i(\cdot)x(\cdot)\in L_1(G_A,\mu)$ and
\begin{align*}
    \|A_ix\|^p_{L_p(X,\mu)} = & %\resizebox{0.81\hsize}{!}{$\displaystyle
    \int\limits_{X} \left|\ \int\limits_{G_A} a_i(t)c_i(s)x(s)d\mu_s \right|^p d\mu_t \\
    \leq & \int\limits_{X} |a_i(t)|^p d\mu_t
  \left(\ \int\limits_{X} |I_{G_A}(s)c_i(s)|^q d\mu_s \right)^{\frac{p}{q}}
  \int\limits_{X} |x(s)|^p d\mu_s
   \\
   = & \int\limits_{X} |a_i(t)|^p d\mu_t \left(\ \int\limits_{G_A} |c_i(s)|^q d\mu_s \right)^{\frac{p}{q}} \int\limits_{X} |x(s)|^p d\mu_s \\
     = &\|a_i\|^p_{L_p(X,\mu)} \|c_i\|_{L_q(G_A,\mu)}^{p}\|x\|_{L_p(X,\mu)}^p,
\end{align*}
for $1< p <\infty$, $i=1,\ldots, l_A$. For $p=\infty$,  by H\"older inequality, for all $x\in L_\infty(X,\mu)$, $c_i(\cdot)x(\cdot)\in L_1(G_A,\mu)$ and
\begin{align*}
 \|A_ix\|_{L_\infty(X,\mu)} & =  \mathop{\esssup}_{t\in X} \left|\int\limits_{G_A} a_i(t)c_i(s)x(s)d\mu_s \right|
 \\
 & \leq
 \left(\mathop{\esssup}_{t\in X} |a_i(t)|\right) \left(\int\limits_{X} |I_{G_A}(s)c_i(s)|d\mu_s\right) \left(\ \mathop{\esssup}_{s\in X} |x(s)|\right)
\\
 &= \left(\ \mathop{\esssup}_{t\in X} |a_i(t)| \right)\left(\ \int\limits_{G_A} |c_i(s)|d\mu_s\right) \left(\ \mathop{\esssup}_{s\in X} |x(s)|\right)
 \\
 &
 =\|a_i\|_{L_\infty(X,\mu)}\|c_i\|_{L_1(G_A,\mu)}\|x\|_{L_\infty(X,\mu)},
\end{align*}
for $i=1,\ldots, l_A$. Similarly, one argues on the case $p=1$.   Since $L_p(X,\mu)$, $1\leq p \leq \infty$ is a linear space, we conclude that operators $A$, $B$ are well defined on $L_p(X,\mu)$ and bounded.

In the following theorem we will consider another equivalent way of writing the operators
in \eqref{OpeqnsSKthmGenSeparedKnlsAB} in order to simplify notation on special cases that we will see along the article.

\begin{theorem}\label{thmBothIntOPSupGenSeptedKernels}
Let $(X,\Sigma,\mu)$ be $\sigma$-finite measure space. Let $A:L_p(X,\mu)\to L_p(X,\mu)$,  $B:L_p(X,\mu)\to L_p(X,\mu)$, $1\le p\le\infty$ be nonzero operators defined as follows
\begin{gather}\label{OpeqnsSKthmSupGenSeparedKnlsA}
  (Ax)(t)= \int\limits_{G_A} \sum_{k=1}^{l_{A,1}}\sum_{i=1}^{l_{A,2}}\sum _{l=1}^{l_{A,3}} \gamma_A^{k,i,l}a_{k,l}(t)c_{i,l}(s)x(s)d\mu_s,\\
\label{OpeqnsSKthmSupGenSeparedKnlsB}
  (Bx)(t)= \int\limits_{G_B}\sum_{u=1}^{l_{B,1}}\sum_{w=1}^{l_{B,2}}\sum _{r=1}^{l_{B,3}} \gamma_B^{u,w,r} b_{u,r}(t)e_{w,r}(s)x(s)d\mu_s,
\end{gather}
for almost every $t$, where the index in $\mu_s$ indicates the variable of integration, $G_A\in \Sigma$ and $G_B\in \Sigma$, $a_{k,l},b_{u,r}\in L_p(X,\mu)$, $c_{i,l}\in L_q(G_A,\mu)$, $e_{w,r}\in L_q(G_B,\mu)$, $k,i,l, l_{A,1},l_{A,2}, l_{A,3}$, $u,w,r, l_{B,1},l_{B,2}, l_{B,3}$ are positive integers such that $1\le k\le  l_{A,1}$, $1\le i\le  l_{A,2}$, $1\le l\le  l_{A,3}$, $1\le u\le  l_{B,1}$, $1\le w\le  l_{B,2}$,  $1\le r\le  l_{B,3}$ and $1\le q\le\infty$ such that $\frac{1}{p}+\frac{1}{q}=1$, $\gamma_A^{k,i,l}, \gamma_B^{u,w,r}\in \mathbb{R}$. Consider a polynomial defined by
$F(z)=\sum\limits_{j=0}^{n} \delta_j z^j$, where $\delta_j \in\mathbb{R}$ $j=0,\ldots,n$. Let $ G=G_A\cap G_B,$ and
\[
 \tilde{\gamma}_{1}=1,\quad \tilde{\gamma}_{m}=\prod_{v=1}^{m-1} Q_{G_A} (a_{k_{v+1},l_{v+1}},c_{i_v,l_v}),\ m\ge 2.
\]
where $Q_{\Lambda}(u,v)$, $\Lambda\in\Sigma$, is defined by \eqref{QGpairingDefinition}. Then
$
  AB=BF(A)
$
if and only if  the following conditions are fulfilled:
\begin{enumerate}[label=\textup{\arabic*.}, ref=\arabic*, leftmargin=*]
  \item \label{thmBothIntOPSupGenSeptedKernels:cond1}
        for almost every $(t,s)\in  X\times G$,
           \begin{gather*}
             -\sum_{u,w,r=1}^{l_{B,1},l_{B,2},l_{B,3}} \delta_0 \gamma_B^{u,w,r} b_{u,r}(t)e_{w,r}(s)+\sum_{(k,i,l=1)}^{(l_{A,1},l_{A,2}, l_{A_3})}\sum_{(u,w,r=1)}^{(l_{B,1}, l_{B_2}, l_{B_3})} \gamma_A^{k,i,l} \gamma_B^{u,w,r} \\
             \cdot a_{k,l}(t)Q_{G_A}(c_{i,l},  b_{u,r} ) e_{w,r}(s)\\
     =\sum_{(u,w,r=1)}^{(l_{B,1},l_{B,2},l_{B,3})}\sum_{j=1}^n  \sum_{(k_1,i_1, l_1=1)}^{(l_{A,1}, l_{A,2}, l_{A,3})}
     \ldots \sum_{(k_{j},i_{j}, l_{j}=1)}^{(l_{A,1}, l_{A,2}, l_{A,3})} \delta_j \tilde{\gamma}_{j}\gamma_B^{u,w,r} \\ \cdot\prod_{v=1}^{j} \gamma_A^{k_v,i_v,l_v}  Q_{G_B}(e_{w,r},a_{k_1,l_1})b_{u,r}(t) c_{i_j,l_j} (s)
            \end{gather*}
\item  \label{thmBothIntOPSupGenSeptedKernels:cond2}  for almost every $(t,s)\in  X\times (G_A\setminus G)$,
 \begin{gather*}
 \sum_{(u,w,r=1)}^{(l_{B,1},l_{B,2},l_{B,3})}  \sum_{j=1}^n  \sum_{(k_1,i_1, l_1=1)}^{(l_{A,1}, l_{A,2}, l_{A,3})}
 \ldots \sum_{(k_{j},i_{j}, l_{j}=1)}^{(l_{A,1}, l_{A,2}, l_{A,3})} \delta_j \tilde{\gamma}_{j}\gamma_B^{u,w,r}\\
 \hspace{4cm}  \cdot\prod_{v=1}^{j} \gamma_A^{k_v,i_v,l_v}
  Q_{G_B}(e_{w,r},a_{k_1,l_1})b_{u,r}(t)  c_{i_{j},l_{j}}(s)=0.
 \end{gather*}
\item \label{thmBothIntOPSupGenSeptedKernels:cond3} for almost every $(t,s)\in  X\times (G_B\setminus G)$,
\begin{align*}
& \sum_{(u,w,r=1)}^{(l_{B,1},l_{B,2},l_{B,3})} \delta_0 \gamma_B^{u,w,r} b_{u,r}(t)e_{w,r}(s)=\hspace{-2mm}\sum_{(k,i,l=1)}^{(l_{A,1},l_{A,2},l_{A,3})}\sum_{(u,w,r=1)}^{(l_{B,1},l_{B,2},l_{B,3})} \gamma_A^{k,i,l} \gamma_B^{u,w,r}  a_{k,l}(t)\\
  & \hspace{7cm} \cdot Q_{G_A}(c_{i,l},  b_{u,r}) e_{w,r}(s).
\end{align*}
\end{enumerate}
\end{theorem}
\begin{proof}
We observe that if for all $k,i,l$, $1\leq k \leq l_{A,1}$, $1\leq i\leq l_{A,2}$, $1\leq l\leq l_{A,3}$, $1\leq u\leq l_{B,1},$ $1\leq w\leq l_{B,2}$, $1\leq r\leq l_{B,3}$, $a_{k,l},b_{u,r}\in L_p(X,\mu),\ 1\leq p\leq\infty$, $c_{i,l}\in L_q(G_A,\mu)$, $e_{w,r}\in L_q(G_B,\mu)$, where $1\leq q\leq \infty$, with $\frac{1}{p}+\frac{1}{q}=1,$ $L_p(X,\mu)$, $L_q(X,\mu)$ are linear spaces, then  operators $A$ and $B$ are well-defined. By direct calculation, we have
  \begin{eqnarray*}
  (A^2x)(t)&=&\int\limits_{G_A} \sum_{k_1=1}^{l_{A,1}}\sum_{i_1=1}^{l_{A,2}}\sum_{l_1=1}^{l_{A,3}} \gamma_A^{k,i,l} a_{k_1,l_1}(t)c_{i_1,l_1}(s)(Ax)(s)d\mu_s\\
  &=&\int\limits_{G_A} \sum_{(k_1,i_1, l_1=1)}^{(l_{A,1}, l_{A,2}, l_{A,3})}\sum_{(k_2,i_2, l_2=1)}^{(l_{A,1}, l_{A,2}, l_{A,3})} \gamma_A^{k_1,i_1,l_1}\gamma_A^{k_2,i_2,l_2} a_{k_1,l_1}(t)c_{i_1,l_1}(s) a_{k_2,l_2}(s)d\mu_s \\
  & &\cdot \int\limits_{G_A} c_{i_2,l_2}(\tau)x(\tau)d\mu_{\tau}\\
  &=& \sum_{(k_1,i_1, l_1=1)}^{(l_{A,1}, l_{A,2}, l_{A,3})}\sum_{(k_2,i_2, l_2=1)}^{(l_{A,1}, l_{A,2}, l_{A,3})} \gamma_A^{k_1,i_1,l_1}\gamma_A^{k_2,i_2,l_2} a_{k_1,l_1}(t)Q_{G_A}(a_{k_2,l_2},c_{i_1,l_1})\\
  & & \cdot \int\limits_{G_A} c_{i_2,l_2}(\tau)x(\tau)d\mu_\tau
  \\
  &=&
  \int\limits_{G_A} \sum_{(k_1,i_1, l_1=1)}^{(l_{A,1}, l_{A,2}, l_{A,3})}\sum_{(k_2,i_2, l_2=1)}^{(l_{A,1}, l_{A,2}, l_{A,3})} \prod_{v=1}^2 \gamma_A^{k_v,i_v,l_v} \tilde{\gamma}_{2} a_{k_1,l_1}(t)c_{i_2,l_2}(\tau)x(\tau)d\mu_\tau,
\end{eqnarray*}
for  almost every $t$.  We suppose that
\begin{equation}\label{AonCompositionIntegralOpSupGenSeparetedKernels}
\begin{array}{l}
(A^{m}x)(t)= \displaystyle
\int\limits_{G_A} \sum_{(k_1,i_1, l_1=1)}^{(l_{A,1}, l_{A,2}, l_{A,3})}\hspace{0mm}\ldots \hspace{0mm} \sum_{(k_m,i_m, l_m=1)}^{(l_{A,1}, l_{A,2}, l_{A,3})} \\
\hspace{4cm} \displaystyle \prod_{v=1}^m \gamma_A^{k_v,i_v,l_v} \tilde{\gamma}_m a_{k_1,l_1}(t) c_{i_m,l_m}(\tau)d\mu_\tau, \  m\ge 1
\end{array}
  \end{equation}
 for almost every $t$. Then
  \begin{align*}
& (A^{m+1}x)(t) = \int\limits_{G_A}\hspace{-0.5mm} \sum_{(k_1,i_1, l_1=1)}^{(l_{A,1}, l_{A,2}, l_{A,3})}
   \ldots \sum_{k_{m},i_{m}, l_{m}=1}^{l_{A,1}, l_{A,2}, l_{A,3}} \prod_{v=1}^{m} \gamma_A^{k_v,i_v,l_v} \tilde{\gamma}_m a_{k_1,l_1}(t) c_{i_m,l_m}(\tau)
   \\
    &\hspace{0,5cm} \cdot\Big(\ \int\limits_{G_A}  \sum_{k_{m+1},i_{m+1}, l_{m+1}=1}^{l_{A,1}, l_{A,2}, l_{A,3}} \gamma_A^{k_{m+1},i_{m+1},l_{m+1}} a_{k_{m+1},l_{m+1}}(\tau)c_{i_{m+1},l_{m+1}}(s)x(s)d\mu_s \Big)d\mu_\tau
    \\
    &\ = \int\limits_{G_A} \sum_{(k_1,i_1, l_1=1)}^{(l_{A,1}, l_{A,2}, l_{A,3})}\ldots \sum_{(k_{m+1},i_{m+1}, l_{m+1}=1)}^{(l_{A,1}, l_{A,2}, l_{A,3})} \prod_{v=1}^{m+1} \gamma_A^{k_v,i_v,l_v} \tilde{\gamma}_m a_{k_1,l_1}(t) \\
     &\hspace{4cm}\cdot Q_{G_A}(c_{i_m,l_m},a_{k_{m+1},l_{m+1}})
   c_{i_{m+1},l_{m+1}}(s)x(s)d\mu_s \\
    &\ = \int\limits_{G_A} \sum_{(k_1,i_1, l_1=1)}^{(l_{A,1}, l_{A,2}, l_{A,3})}\ldots \sum_{(k_{m+1},i_{m+1}, l_{m+1}=1)}^{(l_{A,1}, l_{A,2}, l_{A,3})} \prod_{v=1}^{m+1} \gamma_A^{k_v,i_v,l_v} \gamma_{m+1} a_{k_1,l_1}(t)  \\
     &\hspace{7cm} \cdot c_{i_{m+1},l_{m+1}}(s)x(s)d\mu_s
    %\int\limits_{\alpha_1}^{\beta_1} a(t)a(s)b(s)ds\int\limits_{\alpha_1}^{\beta_1} a(\tau_n)b(\tau_n)d\tau_n \cdot \hdots \cdot\int\limits_{\alpha_1}^{\beta_1} b(\tau_1)x(\tau_1)d\tau_1=\\ \label{AonCompositionIntegralOperatorSplittedKernel}
  %&=&a(t)\left( \int\limits_{\alpha_1}^{\beta_1} a(s)b(s)ds\right)^{n-1}\int\limits_{\alpha_1}^{\beta_1} b(\tau_1)x(\tau_1)d\tau_1.
\end{align*}
for almost every $t$. Then, by induction \eqref{AonCompositionIntegralOpSupGenSeparetedKernels}  holds true
for all positive integers $m$. Therefore, we have
\begin{align*}
(F(A)x)(t)=&\resizebox{0.82\hsize}{!}{$\displaystyle  \delta_0 x(t)+\int\limits_{G_A} \sum_{j=1}^n   \delta_j  \sum_{(k_1,i_1, l_1=1)}^{(l_{A,1}, l_{A,2}, l_{A,3})}\ldots \sum_{(k_{j},i_{j}, l_{j}=1)}^{(l_{A,1}, l_{A,2}, l_{A,3})} \prod_{v=1}^{j} \gamma_A^{k_v,i_v,l_v} \tilde{\gamma}_{j}
  $}
\\
& \hspace{6cm} \cdot a_{k_1,l_1}(t)  c_{i_{j},l_{j}}(s)x(s)d\mu_s
\end{align*}
for almost every $t$. Then we compute,
  \begin{align*} %\label{CompABProofThmBothIntOpSupGenSeparatedKnDI}
& (ABx)(t)= \int\limits_{G_A} \sum\limits_{(k,i,l=1)}^{(l_{A,1},l_{A,2},l_{A,3})} \gamma_A^{k,i,l} a_{k,l}(t)c_{i,l}(s)
   \Big(\ \int\limits_{G_B}\sum\limits_{(u,w,r=1)}^{(l_{B,1},l_{B,2},l_{B,3})} \gamma_B^{u,w,r} b_{u,r}(s) \\
  & \hspace{8cm}  \cdot e_{w,r}(\tau)x(\tau)d\mu_\tau \Big)d\mu_s
   \\
&\int\limits_{G_B} \sum\limits_{(k,i,l=1)}^{(l_{A,1},l_{A,2},l_{A,3})}\sum\limits_{(u,w,r=1)}^{(l_{B,1},l_{B,2},l_{B,3})}\gamma_A^{k,i,l} \gamma_B^{u,w,r}  a_{k,l}(t)Q_{G_A}(c_{i,l}, b_{u,r}) e_{w,r}(\tau)x(\tau)d\mu_\tau, \\
  &
    (BF(A)x)(t)=\delta_0\int\limits_{G_B} \sum\limits_{(u,w,r=1)}^{(l_{B,1},l_{B,2},l_{B,3})}\gamma_B^{u,w,r} b_{u,r}(t)e_{w,r}(\tau)x(\tau)d\mu_\tau\\
   & + \int\limits_{G_B}\hspace{0mm} \sum\limits_{(u,w,r=1)}^{(l_{B,1},l_{B,2},l_{B,3})}\gamma_B^{u,w,r} b_{u,r}(t)e_{w,r}(s) \Big(\ \int\limits_{G_A} \sum\limits_{j=1}^n   \delta_j  \sum\limits_{(k_1,i_1, l_1=1)}^{(l_{A,1}, l_{A,2}, l_{A,3})}\hspace{0mm}\ldots \hspace{0mm}\sum\limits_{k_{j},i_{j}, l_{j}=1}^{l_{A,1}, l_{A,2}, l_{A,3}}
   \\
&  \Big(\prod_{v=1}^{j} \gamma_A^{k_v,i_v,l_v} \tilde{\gamma}_{j} a_{k_1,l_1}(s) c_{i_{j},l_{j}}(\tau)x(\tau)d\mu_\tau\Big)\Big)d\mu_s
\\
&= \int\limits_{G_B} \sum\limits_{(u,w,r=1)}^{(l_{B,1},l_{B,2},l_{B,3})} \delta_0 \gamma_B^{u,w,r} b_{u,r}(t)e_{w,r}(\tau)x(\tau)d\mu_\tau \\
& + \int\limits_{G_A}\sum\limits_{(u,w,r=1)}^{(l_{B,1},l_{B,2},l_{B,3})} \sum\limits_{j=1}^n  \sum\limits_{(k_1,i_1, l_1=1)}^{(l_{A,1}, l_{A,2}, l_{A,3})}\ldots \sum\limits_{(k_{j},i_{j}, l_{j}=1)}^{(l_{A,1}, l_{A,2}, l_{A,3})} \delta_j \tilde{\gamma}_{j}\gamma_B^{u,w,r} \prod_{v=1}^{j} \gamma_A^{k_v,i_v,l_v}  \\
%\label{EqBFAProofThmBothIntOpSupGenSepratKnDI}
& \hspace{5cm} \cdot Q_{G_B}(e_{w,r},a_{k_1,l_1})b_{u,r}(t) c_{i_{j},l_{j}}(\tau)x(\tau)d\mu_\tau.
  \end{align*}
Thus, $(ABx)(t)=(BF(A)x)(t)$ for all $x\in L_p( X,\mu)$ if and only if
 \begin{gather*}
  \int\limits_{G_B} \sum_{(k,i,l=1)}^{(l_{A,1},l_{A,2},l_{A,3})}\sum_{(u,w,r=1)}^{(l_{B,1},l_{B,2},l_{B,3})}\gamma_A^{k,i,l} \gamma_B^{u,w,r}  a_{k,l}(t)Q_{G_A}(c_{i,l},  b_{u,r}) e_{w,r}(\tau)x(\tau)d\mu_\tau\\
    =\int\limits_{G_B} \sum_{u,w,r=1}^{l_{B,1},l_{B,2},l_{B,3}} \delta_0 \gamma_B^{u,w,r} b_{u,r}(t)e_{w,r}(\tau)x(\tau)d\mu_\tau\\
 \resizebox{0.96\hsize}{!}{$\displaystyle
+ \int\limits_{G_A}\sum_{(u,w,r=1)}^{(l_{B,1},l_{B,2},l_{B,3})}  \sum_{j=1}^n  \sum_{(k_1,i_1, l_1=1)}^{(l_{A,1}, l_{A,2}, l_{A,3})}
\ldots \sum_{(k_{j},i_{j}, l_{j}=1)}^{(l_{A,1}, l_{A,2}, l_{A,3})} \delta_j  \tilde{\gamma}_{j}\gamma_B^{u,w,r} \prod_{v=1}^{j} \gamma_A^{k_v,i_v,l_v}
  $}
\\
\hspace{5cm} \cdot Q_{G_B}(e_{w,r},a_{k_1,l_1})b_{u,r}(t) c_{i_{j},l_{j}}(\tau)x(\tau)d\mu_\tau.
 \end{gather*}
Then by applying Lemma \ref{LemmaAllowInfSetsEqLp}, we conclude that $AB=BF(A)$ if and only if
    \begin{enumerate}[label=\textup{\arabic*.}, ref=\arabic*,leftmargin=*]
  \item for almost every $(t,s)\in  X\times G$,
            \begin{align*}
           &  -\sum_{u,w,r=1}^{l_{B,1},l_{B,2},l_{B,3}} \delta_0 \gamma_B^{u,w,r} b_{u,r}(t)e_{w,r}(s)
           \\
           &
           +\sum_{(k,i,l=1)}^{(l_{A,1},l_{A,2}, l_{A_3})}\sum_{(u,w,r=1)}^{(l_{B,1}, l_{B_2}, l_{B_3})} \gamma_A^{k,i,l} \gamma_B^{u,w,r}
            a_{k,l}(t)Q_{G_A}(c_{i,l},  b_{u,r} ) e_{w,r}(s) \\
    &
    =\sum_{(u,w,r=1)}^{(l_{B,1},l_{B,2},l_{B,3})}\sum_{j=1}^n  \sum_{(k_1,i_1, l_1=1)}^{(l_{A,1}, l_{A,2}, l_{A,3})}
    \ldots \sum_{(k_{j},i_{j}, l_{j}=1)}^{(l_{A,1}, l_{A,2}, l_{A,3})} \delta_j \tilde{\gamma}_{j}\gamma_B^{u,w,r} \prod_{v=1}^{j} \gamma_A^{k_v,i_v,l_v}
    \\
& \hspace{5cm}  \cdot Q_{G_B}(e_{w,r},a_{k_1,l_1})b_{u,r}(t) c_{i_{j},l_{j}}(s)
            \end{align*}

\item  for almost every $(t,s)\in  X\times (G_A\setminus G)$,
 \begin{align*}
&\sum_{(u,w,r=1)}^{(l_{B,1},l_{B,2},l_{B,3})}
\sum_{j=1}^n  \sum_{(k_1,i_1, l_1=1)}^{(l_{A,1}, l_{A,2}, l_{A,3})}
\ldots \sum_{(k_{j},i_{j}, l_{j}=1)}^{(l_{A,1}, l_{A,2}, l_{A,3})} \delta_j \tilde{\gamma}_{j}\gamma_B^{u,w,r} \prod_{v=1}^{j} \gamma_A^{k_v,i_v,l_v}
\\
& \hspace{3cm} \cdot Q_{G_B}(e_{w,r},a_{k_1,l_1})b_{u,r}(t) c_{i_{j},l_{j}}(s)=0;
 \end{align*}
\item for almost every $(t,s)\in  X\times (G_B\setminus G)$,
\begin{align*}
&
 \sum_{u,w,r=1}^{l_{B,1},l_{B,2},l_{B,3}} \delta_0 \gamma_B^{u,w,r} b_{u,r}(t)e_{w,r}(s)
 \\
 &
 =\sum_{(k,i,l=1)}^{(l_{A,1},l_{A,2},l_{A,3})}\sum_{(u,w,r=1)}^{(l_{B,1},l_{B,2},l_{B,3})} \gamma_A^{k,i,l} \gamma_B^{u,w,r}  a_{k,l}(t)
    Q_{G_A}(c_{i,l},  b_{u,r}) e_{w,r}(s).
\end{align*}
\end{enumerate} \QEDB
\end{proof}
\begin{remark}
Forms of operator $A:L_p(X,\mu)\to L_p(X,\mu)$, $1\le p\le \infty$ given in \eqref{OpeqnsSKthmGenSeparedKnlsAB} and \eqref{OpeqnsSKthmSupGenSeparedKnlsA} are equivalent. In fact,
if in \eqref{OpeqnsSKthmGenSeparedKnlsAB} we take \\ $l_{A,1}=l_{A,2}=l_A$, $l_{A,3}=1$, $\tilde{a}_{i,1}(\cdot)=a_i(\cdot)$, $1\le i\le l_A$, $\tilde{c}_{j,1}(\cdot)=c_{j}(\cdot)$, $1\le j\le l_A$,
$
  \gamma_A^{(i,j,1)}=\left\{\begin{array}{cc}
                   1, & i=j   \\
                   0, & i\not=j
                 \end{array}, \right.
$
we get \eqref{OpeqnsSKthmSupGenSeparedKnlsA}. On the other hand, from expression \eqref{OpeqnsSKthmSupGenSeparedKnlsA}, we assign for each triple $(i,j,k)\in \mathbb{N}^3$, $1\le i\le l_{A,1}$, $1\le j\le l_{A,2}$, $1\le k\le l_{A,3}$, the number
$l=(i-1)l_{A,2}l_{A,3}+(j-1)l_{A,3}+k$ and functions $\tilde{a}_l(\cdot)=\gamma_A^{(i,j,k)}a_{i,k}(\cdot)$, $\tilde{c}_l(\cdot)=c_{j,k}(\cdot)$. Define the partial order relation $\preceq$
 in $\mathbb{N}^3\times\mathbb{N}^3$ as follows
 \begin{equation*}
  \resizebox{0.98\hsize}{!}{$\displaystyle (i_1,j_1,k_1)\preceq (i_2,j_2,k_2) \mbox{ if } i_1 < i_2 \mbox{ or } (i_1=i_2 \mbox{ and } j_1<j_2) \mbox{ or }
   (i_1=i_2,\, j_1=j_2,\, k_1\le k_2). $}
 \end{equation*}
It is reflexive, anti-symmetric and transitive, the reader can see the corresponding definitions in \cite{FollandRA,KolmogorovFominETFFAbookVol2}. It is reflexive because for each triple $(i,j,k)\in \mathbb{N}^3$ we have
$i=i,\, j=j,\, k=k $ hence $(i,j,k)\preceq (i,j,k)$. If two triples $(i_1,j_1,k_1),(i_2,j_2,k_2)\in \mathbb{N}^3$ satisfy $(i_1,j_1,k_1)\preceq (i_2,j_2,k_2)$ and $(i_2,j_2,k_2)\preceq (i_1,j_1,k_1)$ we have
\begin{eqnarray*}
  ( i_1< i_2 \mbox{ or } i_1=i_2 ) &\mbox{ and }&  (i_2< i_1 \mbox{ or } i_2=i_1)\\
  ( j_1< j_2 \mbox{ or } j_1=j_2 ) &\mbox{ and }&  (j_2< j_1 \mbox{ or } j_2=j_1)\\
  ( k_1\leq  k_2  ) &\mbox{ and }&  (k_2\leq  k_1 ).
\end{eqnarray*}
These conditions imply that $i_1=i_2$, $j_1=j_2$ and $k_1=k_2$, and consequently $(i_1,j_1,k_1)=(i_2,j_2,k_2)$. Therefore, the relation $\preceq$ is anti-symmetric. Suppose now that $(i_1,j_1,k_1)$, $(i_2,j_2,k_2)$, $(i_3,j_3,k_3)\in \mathbb{N}^3$, satisfy $(i_1,j_1,k_1)\preceq (i_2,j_2,k_2)$ and
$(i_2,j_2,k_2)\preceq (i_3,j_3,k_3)$. We have
\begin{eqnarray*}
 && i_1 < i_2 \mbox{ or } (i_1=i_2 \mbox{ and } j_1<j_2) \mbox{ or }
   (i_1=i_2,\, j_1=j_2,\, k_1\le k_2),\\
&&  i_2 < i_3 \mbox{ or } (i_2=i_3 \mbox{ and } j_2<j_3) \mbox{ or }
   (i_2=i_3,\, j_2=j_3,\, k_2\le k_3),
\end{eqnarray*}
which imply
\begin{equation*}
i_1 < i_3 \mbox{ or } (i_1=i_3 \mbox{ and } j_1<j_3) \mbox{ or }
   (i_1=i_3,\, j_1=j_3,\, k_1\le k_3).
\end{equation*}
Hence $(i_1,j_1,k_1)\preceq (i_3,j_3,k_3)$, so the relation $\preceq$ is transitive and thus it is partial order relation. Moreover relation $\preceq$ is total order since for any two triples $(i_1,j_1,k_1)\in \mathbb{N}^3$ and $(i_2,j_2,k_2)\in \mathbb{N}^3$ we have either $(i_1,j_1,k_1)\preceq (i_2,j_2,k_2)$ or $(i_2,j_2,k_2) \preceq (i_1,j_1,k_1)$. Therefore, when $ (1,1,1)\preceq (i,j,k) \preceq (l_{A,1}, l_{A,2}, l_{A,3}),$
 $l=(i-1)l_{A,2}l_{A,3}+(j-1)l_{A,3}+k$ is uniquely defined and $1\le l \le l_{A,1}\cdot l_{A,2}\cdot l_{A,3}$. Hence, $l_A=l_{A,1}\cdot l_{A,2}\cdot l_{A,3}$ in expression \eqref{OpeqnsSKthmGenSeparedKnlsAB}.
\end{remark}

From now on, we will consider in statements that operators $A$ and $B$ are given by \eqref{OpeqnsSKthmGenSeparedKnlsAB} which is a simpler form of linear integral operator with general separated kernel.
\begin{theorem}\label{thmBothIntOPGenSeptedKernels}
Let $(X,\Sigma,\mu)$ be $\sigma$-finite measure space. Let $A:L_p(X,\mu)\to L_p(X,\mu)$,  $B:L_p(X,\mu)\to L_p(X,\mu)$, $1\le p\le\infty$ be nonzero operators defined as follows
\begin{equation*} %\label{OpeqnsSKthmGenSeparedKnls}
  (Ax)(t)= \int\limits_{G_A} \sum _{i=1}^{l_A} a_i(t)c_i(s)x(s)d\mu_s,\quad (Bx)(t)= \int\limits_{G_B}\sum_{j=1}^{l_B} b_j(t)e_j(s)x(s)d\mu_s,
\end{equation*}
for almost every $t$, where the index in $\mu_s$ indicates the variable of integration, $G_A\in \Sigma$ and $G_B\in \Sigma$, $a_i,b_j\in L_p(X,\mu)$, $c_i\in L_q(G_A,\mu)$, $e_j\in L_q(G_B,\mu)$, $i,j, l_A,l_B$ are positive integers such that $1\leq i\leq  l_A$,  $1\leq j\leq  l_B$ and  $1\leq q\leq\infty$ such that $\frac{1}{p}+\frac{1}{q}=1$. Consider a polynomial defined by
$F(z)=\sum\limits_{j=0}^{n} \delta_j z^j$, where $\delta_j \in\mathbb{R}$ $j=0,1,2,\ldots,n$. Let $ G=G_A\cap G_B,$ and
\[
 \gamma_{i_1}=1,\quad \gamma_{i_1,\ldots,i_m}=\prod_{l=1}^{m-1} Q_{G_A} (a_{i_{l+1}},c_{i_l}),\ m\ge 2.
\]
where $Q_{\Lambda}(u,v)$, $\Lambda\in\Sigma$, is defined by \eqref{QGpairingDefinition}. Then
$
  AB=BF(A)
$
if and only if  the following conditions are fulfilled:
\begin{enumerate}[label=\textup{\arabic*.}, ref=\arabic*]
  \item \label{thmBothIntOPGenSeptedKernels:cond1}
        for almost every $(t,s)\in  X\times G$,
            \begin{gather*}
            - \sum_{k=1}^{l_B} \delta_0 b_k(t)e_k(s)+\sum_{k=1}^{l_B}\sum_{m=1}^{l_A}a_m(t)Q_{G_A}(b_k,c_m)e_k(s)\\
             =\sum_{k=1}^{l_B}\sum_{j=1}^{n}\sum_{i_1,\ldots,i_j=1}^{l_A} \delta_j b_k(t)Q_{G_B}(e_k,a_{i_1})\gamma_{i_1,\ldots,i_j}c_{i_j}(s);
            \end{gather*}
\item  \label{thmBothIntOPGenSeptedKernels:cond2}  for almost every $(t,s)\in  X\times (G_A\setminus G)$,
 \begin{equation*}
 \sum_{k=1}^{l_B}\sum_{j=1}^n \sum_{i_1,\ldots, i_j=1}^{l_A} \delta_j b_k(t)Q_{G_B}(e_k,a_{i_1}) \gamma_{i_1,\ldots,i_j} c_{i_j}(s)=0;
 \end{equation*}
\item \label{thmBothIntOPGenSeptedKernels:cond3} for almost every $(t,s)\in  X\times (G_B\setminus G)$,
\begin{equation*}
 \sum_{k=1}^{l_B} \delta_0 b_k(t)e_k(s)=\sum_{k=1}^{l_B} \sum_{m=1}^{l_A} Q_{G_A}(b_k,c_m)a_m(t)e_k(s).
\end{equation*}
\end{enumerate}
\end{theorem}
\begin{proof}
We observe that since for each $i,j$, $1\le i \le l_A$, $1\le j\le l_B$,  $a_i,b_j\in L_p(X,\mu),\ 1\le p\le\infty$, $c_i\in L_q(G_A,\mu)$, $e_j\in L_q(G_B,\mu)$, where $1\le q\le \infty$, with $\frac{1}{p}+\frac{1}{q}=1,$ $L_p(X,\mu)$, $L_q(X,\mu)$ are linear spaces, then  operators $A$ and $B$ are well-defined. By direct calculation, we have
  \begin{eqnarray*}
  (A^2x)(t)&=&\int\limits_{G_A} \sum_{i=1}^{l_A} a_i(t)c_i(s)(Ax)(s)d\mu_s\\
  &=&\int\limits_{G_A} \sum_{i_1, i_2=1}^{l_A}a_{i_1}(t)c_{i_1}(s) a_{i_2}(s)d\mu_s\int\limits_{G_A} c_{i_2}(\tau)x(\tau)d\mu_{\tau}\\
  &=& \sum_{i_1, i_2=1}^{l_A} a_{i_1}(t)Q_{G_A}(a_{i_2},c_{i_1})\int\limits_{G_A} c_{i_2}(\tau)x(\tau)d\mu_\tau
  \\
  &=&
  \int\limits_{G_A} \sum_{i_1,i_2=1}^{l_A}\gamma_{i_1,i_2}a_{i_1}(t)c_{i_2}(\tau)x(\tau)d\mu_\tau,
\\
   (A^3x)(t)&=&\int\limits_{G_A} \sum_{i_1=1}^{l_A} a_{i_1}(t)c_{i_1}(s)\left(\,\int\limits_{G_A}\sum_{i_2,i_3=1}^{l_A} a_{i_2}(s)Q_{G_A}(a_{i_3},c_{i_3})c_{i_2}(\tau)x(\tau)d\mu_\tau \right)d\mu_s \\
 &=& \sum_{i_1,i_2,i_3=1}^{l_A} \int\limits_{G_A} a_{i_1}(t)Q_{G_A}(a_{i_3},c_{i_2})  c_{i_1}(s)a_{i_2}(s)d\mu_s   \int\limits_{G_A} c_{i_3}(\tau)x(\tau)d\mu_\tau \\
 &=& \int\limits_{G_A} \sum_{i_1,i_2,i_3=1}^{l_A}  Q_{G_A}(a_{i_3},c_{i_2})Q_{G_A}(a_{i_2},c_{i_1}) a_{i_1}(t)c_{i_3}(\tau)x(\tau)d\mu_\tau \\
 &=& \int\limits_{G_A} \sum_{i_1,i_2,i_3=1}^{l_A} \gamma_{i_1,i_2,i_3} a_{i_1}(t)c_{i_3}(\tau)x(\tau)d\mu_\tau.
\end{eqnarray*}
for  almost every $t$.  We suppose that
  \begin{equation*} %\label{AonCompositionIntegralOpGenSeparetedKernels}
  (A^{m}x)(t)=\int\limits_{G_A} \sum_{i_1,\ldots,i_m=1}^{l_A}\gamma_{i_1,\ldots,i_m}a_{i_1}(t)c_{i_m}(\tau)x(\tau)d\mu_\tau,\quad m=1,2,\ldots
  \end{equation*}
for  almost every $t$. Then
  \begin{align*}\nonumber
    &\resizebox{0.95\hsize}{!}{$\displaystyle (A^{m+1}x)(t)=\hspace{-1mm}\int\limits_{G_A} \hspace{-1mm} \sum_{i_1,\ldots,i_m=1}^{l_A} \hspace{-3mm}\gamma_{i_1,\ldots,i_m}a_{i_1}(t)c_{i_m}(\tau) \left(\int\limits_{G_A} \sum_{i_{m+1}=1}^{l_A} a_{i_{m+1}}(\tau)c_{i_{m+1}}(s)d\mu_s \right)d\mu_\tau $}\\
    &\ =  \int\limits_{G_A} \gamma_{i_1,\ldots,i_{m+1}}a_{i_1}(t)c_{i_{m+1}}(\tau)x(\tau)d\mu_\tau.
\end{align*}
for almost every $t$. Therefore, we have
\begin{equation*}
(F(A)x)(t)=\delta_0 x(t)+\int\limits_{G_A} \sum_{j=1}^n\sum_{i_1,\ldots,i_m=1}^{l_A} \gamma_{i_1,\ldots,i_j} \delta_j a_{i_1}(t)c_{i_j}(\tau)x(\tau)d\mu_\tau
\end{equation*}
 for almost every $t$. Then we compute,
  \begin{align}\label{CompABProofThmBothIntOpGenSeparatedKnDI}
  & \begin{array}{lll}
  (ABx)(t)&=&\displaystyle \int\limits_{G_A} \sum_{m=1}^{l_A} a_m(t)c_m(s)d\mu_s\int\limits_{G_B}\sum_{k=1}^{l_B} b_k(s) e_k(\tau) x(\tau)d\mu_{\tau} \\
  &=& \displaystyle \int\limits_{G_B} \sum_{m=1}^{l_A}\sum_{k=1}^{l_B} Q_{G_A}(b_k,c_m) a_m(t) e_k(\tau) x(\tau)d\mu_{\tau},
  \end{array}
  \\
  \nonumber
   & (BF(A)x)(t)=\int\limits_{G_B} \sum_{k=1}^{l_B} \delta_0 b_k(t)e_k(\tau)x(\tau)d\mu_{\tau}\\ \nonumber
   &+ \int\limits_{G_B}\sum_{k=1}^{l_B} b_k(t)e_k(s) \int\limits_{G_A} \sum_{j=1}^n \sum_{i_1,\ldots,i_j=1}^{l_A} \delta_j \gamma_{i_1,\ldots,i_j} a_{i_1}(s)c_{i_j}(\tau) x(\tau)d\mu_\tau   \\
\nonumber
&= \int\limits_{G_B} \sum_{k=1}^{l_B} \delta_0 b_k(t)e_k(\tau) x(\tau)d\mu_{\tau}\\ \label{EqBFAProofThmBothIntOpGenSepratKnDI}
& + \int\limits_{G_A}\sum_{k=1}^{l_B}\sum_{j=1}^n \sum_{i_1,\ldots,i_j=1}^{l_A} \delta_j \gamma_{i_1,\ldots,i_j} Q_{G_B}(e_k,a_{i_1})b_k(t) c_{i_j}(\tau)x(\tau)d\mu_{\tau}.
  \end{align}
Thus, $(ABx)(t)=(BF(A)x)(t)$ for all $x\in L_p( X,\mu)$ if and only if
 \begin{eqnarray*}
 &&  \int\limits_{G_B} \left(-\sum_{k=1}^{l_B} \delta_0 b_k(t)e_k(s)+\sum_{k=1}^{l_B}\sum_{m=1}^{l_A} Q_{G_A}(b_k,c_m)a_m(t)e_k(s)\right) x(s)d\mu_s\\
  &&  =\int\limits_{G_A} \left(\sum_{k=1}^{l_B}\sum_{j=1}^n \sum_{i_1,\ldots,i_j=1}^{l_A} \delta_j \gamma_{i_1,\ldots,i_j} Q_{G_B}(e_k,a_{i_1})b_k(t)c_{i_j}(s)\right)x(s)d\mu_s.
 \end{eqnarray*}
Then by applying Lemma \ref{LemmaAllowInfSetsEqLp}, we conclude that $AB=BF(A)$ if and only if
    \begin{enumerate}[label=\textup{\arabic*.}, ref=\arabic*]
  \item
        for almost every $(t,s)\in  X\times G$,
            \begin{eqnarray*}
            && -\sum_{k=1}^{l_B} \delta_0 b_k(t)e_k(s)+\sum_{k=1}^{l_B}\sum_{m=1}^{l_A}a_m(t)Q_{G_A}(b_k,c_m)e_k(s)\\
            && =\sum_{k=1}^{l_B}\sum_{j=1}^{n}\sum_{i_1,\ldots,i_j=1}^{l_A} \delta_j b_k(t)Q_{G_B}(e_k,a_{i_1})\gamma_{i_1,\ldots,i_j}c_{i_j}(s),
            \end{eqnarray*}
\item    for almost every $(t,s)\in  X\times (G_A\setminus G)$,
 \begin{equation*}
 \sum_{k=1}^{l_B}\sum_{j=1}^n \sum_{i_1,\ldots, i_j=1}^{l_A} \delta_j b_k(t)Q_{G_B}(e_k,a_{i_1}) \gamma_{i_1,\ldots,i_j} c_{i_j}(s)=0,
 \end{equation*}
\item for almost every $(t,s)\in  X\times (G_B\setminus G)$,
\begin{equation}
 \sum_{k=1}^{l_B} \delta_0 b_k(t)e_k(s)=\sum_{k=1}^{l_B} \sum_{m=1}^{l_A} Q_{G_A}(b_k,c_m)a_m(t)e_k(s).
\tag*{\qed}\end{equation}
\end{enumerate}
\end{proof}
\begin{remark}\label{RemOpDefInSameIntervalGenSepKernels}
In Theorem \ref{thmBothIntOPGenSeptedKernels} when $G_A=G_B=G$, conditions \ref{thmBothIntOPGenSeptedKernels:cond2} and \ref{thmBothIntOPGenSeptedKernels:cond3} are taken on set of measure zero, so we can ignore them. Thus, we only remain with condition \ref{thmBothIntOPGenSeptedKernels:cond1}. When $G_A\not=G_B$ we need to check also conditions \ref{thmBothIntOPGenSeptedKernels:cond2} and \ref{thmBothIntOPGenSeptedKernels:cond3} outside the intersection  $G=G_A\cap G_B$.
\end{remark}

The following corollary is a special case of Theorem \ref{thmBothIntOPGenSeptedKernels} for the important class of
covariance commutation relations associated to affine (degree $1$) polynomials $F$.

\begin{corollary}
Let $(X,\Sigma,\mu)$ be $\sigma$-finite measure space. Let $A:L_p(X,\mu)\to L_p(X,\mu)$,  $B:L_p(X,\mu)\to L_p(X,\mu)$, $1\le p\le\infty$ be nonzero operators defined as follows
\begin{equation*}
  (Ax)(t)= \int\limits_{G_A} \sum _{i=1}^{l_A} a_i(t)c_i(s)x(s)d\mu_s,\quad (Bx)(t)= \int\limits_{G_B}\sum_{j=1}^{l_B} b_j(t)e_j(s)x(s)d\mu_s,
\end{equation*}
for almost every $t$, where the index in $\mu_s$ indicates the variable of integration, $G_A\in \Sigma$ and $G_B\in \Sigma$, $a_i,b_j\in L_p(X,\mu)$, $c_i\in L_q(G_A,\mu)$, $e_j\in L_q(G_B,\mu)$, $i,j, l_A,l_B$ are positive integers such that $1\le i\le  l_A$, $1\le j\le  l_B$ and  $1\le q\le\infty$ with $\frac{1}{p}+\frac{1}{q}=1$. Consider a polynomial
$F(z)= \delta_0 +\delta_1z$, where $\delta_j \in\mathbb{R}$, $j=0,1$. Let $ G=G_A\cap G_B$ and
$Q_{\Lambda}(u,v)$, $\Lambda\in\Sigma$, is defined by \eqref{QGpairingDefinition}. Then $AB=\delta_0B +\delta_1 BA$
if and only if  the following conditions are fulfilled:
\begin{enumerate}[label=\textup{\arabic*.}, ref=\arabic*]
  \item
        for almost every $(t,s)\in  X\times G$,
            \begin{equation*}
      \resizebox{0.93\hsize}{!}{$\displaystyle    \hspace{-4mm}   -\sum_{k=1}^{l_B} \delta_0 b_k(t)e_k(s)+\sum_{k=1}^{l_B}\sum_{m=1}^{l_A}a_m(t)Q_{G_A}(b_k,c_m)e_k(s)
             =\sum_{k=1}^{l_B}\sum_{i_1=1}^{l_A} \delta_1 b_k(t)Q_{G_B}(e_k,a_{i_1})c_{i_1}(s)$}
            \end{equation*}
\item    for almost every $(t,s)\in  X\times (G_A\setminus G)$,
 \begin{equation*}
 \sum_{k=1}^{l_B}\sum_{i_1=1}^{l_A} \delta_1 b_k(t)Q_{G_B}(e_k,a_{i_1})  c_{i_1}(s)=0.
 \end{equation*}
\item for almost every $(t,s)\in  X\times (G_B\setminus G)$,
\begin{equation*}
 \sum_{k=1}^{l_B} \delta_0 b_k(t)e_k(s)=\sum_{k=1}^{l_B} \sum_{m=1}^{l_A} Q_{G_A}(b_k,c_m)a_m(t)e_k(s).
\end{equation*}
\end{enumerate}
\end{corollary}

\begin{proof}
It follows by Theorem \ref{thmBothIntOPGenSeptedKernels} when $n=1$. \QEDB
\end{proof}

\begin{corollary}\label{CorthmBothIntOpGensepKernellsCommutativityOnly}
Let $(X,\Sigma,\mu)$ be $\sigma$-finite measure space. Let $A:L_p(X,\mu)\to L_p(X,\mu)$,  $B:L_p(X,\mu)\to L_p(X,\mu)$, $1\le p\le\infty$ be nonzero operators defined as follows
\begin{equation*}
  (Ax)(t)= \int\limits_{G_A} \sum _{i=1}^{l_A} a_i(t)c_i(s)x(s)d\mu_s,\quad (Bx)(t)= \int\limits_{G_B}\sum_{j=1}^{l_B} b_j(t)e_j(s)x(s)d\mu_s,
\end{equation*}
for almost every $t$, where the index in $\mu_s$ indicates the variable of integration, $G_A\in \Sigma$ and $G_B\in \Sigma$, $a_i,b_j\in L_p(X,\mu)$, $c_i\in L_q(G_A,\mu)$, $e_j\in L_q(G_B,\mu)$, $i,j, l_A,l_B$ are positive integers such that $1\le i\le  l_A$, $1\le j\le  l_B$ and $1\le q\le\infty$ with $\frac{1}{p}+\frac{1}{q}=1$. Let $ G=G_A\cap G_B$ and
$Q_{\Lambda}(u,v)$, $\Lambda\in\Sigma$, is defined by \eqref{QGpairingDefinition}. Then operators $A$ and $B$ commute if and only if  the following conditions are fulfilled:
\begin{enumerate}[label=\textup{\arabic*.}, ref=\arabic*]
  \item\label{CorthmBothIntOpGensepKernellsCommutativityOnly:cond1}  for almost every $(t,s)\in  X\times G$,
            \begin{equation*}
            \sum_{k=1}^{l_B}\sum_{m=1}^{l_A}a_m(t)Q_{G_A}(b_k,c_m)e_k(s)=
             \sum_{k=1}^{l_B}\sum_{i_1=1}^{l_A}  b_k(t)Q_{G_B}(e_k,a_{i_1})c_{i_1}(s),
            \end{equation*}
\item\label{CorthmBothIntOpGensepKernellsCommutativityOnly:cond2}   for almost every $(t,s)\in  X\times (G_A\setminus G)$,
 \begin{equation*}
 \sum_{k=1}^{l_B}\sum_{i_1=1}^{l_A}  b_k(t)Q_{G_B}(e_k,a_{i_1})  c_{i_1}(s)=0,
 \end{equation*}
\item\label{CorthmBothIntOpGensepKernellsCommutativityOnly:cond3} for almost every $(t,s)\in  X\times (G_B\setminus G)$,
\begin{equation*}
 \sum_{k=1}^{l_B} \sum_{m=1}^{l_A} Q_{G_A}(b_k,c_m)a_m(t)e_k(s)=0.
\end{equation*}
\end{enumerate}
\end{corollary}

\begin{proof}
It follows by Theorem \ref{thmBothIntOPGenSeptedKernels} when $n=1$, $\delta_0=0$ and $\delta_1=1$. \QEDB
\end{proof}

The following corollary of Theorem \ref{thmBothIntOPGenSeptedKernels} is concerned with  representations by integral operators of another important family of covariance commutation relations associated to monomials $F$.
\begin{corollary}\label{corDiedroRelBothIntOPGenSeptedKernels}
Let $(X,\Sigma,\mu)$ be $\sigma$-finite measure space. Let $A:L_p(X,\mu)\to L_p(X,\mu)$,  $B:L_p(X,\mu)\to L_p(X,\mu)$, $1\le p\le\infty$ be nonzero operators defined as follows
\begin{equation*} %\label{OpeqnsSKthm}
  (Ax)(t)= \int\limits_{G_A} \sum _{i=1}^{l_A} a_i(t)c_i(s)x(s)d\mu_s,\quad (Bx)(t)= \int\limits_{G_B}\sum_{j=1}^{l_B} b_j(t)e_j(s)x(s)d\mu_s,
\end{equation*}
for almost every $t$, where the index in $\mu_s$ indicates the variable of integration, $G_A\in \Sigma$ and $G_B\in \Sigma$, $a_i,b_j\in L_p(X,\mu)$, $c_i\in L_q(G_A,\mu)$, $e_j\in L_q(G_B,\mu)$, $i,j, l_A,l_B$ are positive integers such that $1\le i\le  l_A$, $1\le j\le  l_B$ and  $1\le q\le\infty$ with $\frac{1}{p}+\frac{1}{q}=1$. Consider a polynomial $F:\mathbb{R}\to \mathbb{R}$ defined by
$F(z)=\delta z^d$, where $\delta \in\mathbb{R}$, $d\in\mathbb{Z}$, $d>0$. Let $ G=G_A\cap G_B,$ and
\[
 \gamma_{i_1}=1,\quad \gamma_{i_1,\ldots,i_m}=\prod_{l=1}^{m-1} Q_{G_A} (a_{i_{l+1}},c_{i_l}),\ m\ge 2.
\]
where $Q_{\Lambda}(u,v)$, $\Lambda\in\Sigma$, is defined by \eqref{QGpairingDefinition}. Then
$
  AB=BF(A)
$
if and only if  the following conditions are fulfilled:
\begin{enumerate}[leftmargin=*, label=\textup{\arabic*.}, ref=\arabic*]
  \item
        for almost every $(t,s)\in  X\times G$,
            \begin{equation*}\hspace{-0.4cm}
           \sum_{k=1}^{l_B}\sum_{m=1}^{l_A}a_m(t)Q_{G_A}(b_k,c_m)e_k(s)=
             \sum_{k=1}^{l_B}\sum_{i_1,\ldots,i_d=1}^{l_A} \delta b_k(t)Q_{G_B}(e_k,a_{i_1})\gamma_{i_1,\ldots,i_d}c_{i_d}(s)
            \end{equation*}
\item  for almost every $(t,s)\in  X\times (G_A\setminus G)$,
 \begin{equation*}
 \sum_{k=1}^{l_B} \sum_{i_1,\ldots, i_d=1}^{l_A} \delta b_k(t)Q_{G_B}(e_k,a_{i_1}) \gamma_{i_1,\ldots,i_d} c_{i_d}(s)=0.
 \end{equation*}
\item for almost every $(t,s)\in  X\times (G_B\setminus G)$,
\begin{equation*}
\sum_{k=1}^{l_B} \sum_{m=1}^{l_A} Q_{G_A}(b_k,c_m)a_m(t)e_k(s)=0.
\end{equation*}
\end{enumerate}
\end{corollary}

\begin{proof}
It follows from Theorem \ref{thmBothIntOPGenSeptedKernels} when $n=d$, $\delta_j=0$, for $j\in \{0,\ldots, n\}\setminus \{d\}$ and $\delta_d=\delta$. \QEDB
\end{proof}

\begin{corollary}\label{corDiedroRelBothIntOPGenSeptedKernelsOrtogonality}
Let $(X,\Sigma,\mu)$ be $\sigma$-finite measure space. Let $A:L_p(X,\mu)\to L_p(X,\mu)$,  $B:L_p(X,\mu)\to L_p(X,\mu)$, $1\le p\le\infty$ be nonzero operators defined as follows
\begin{equation*}%\label{OpeqnsSKthm}
  (Ax)(t)= \int\limits_{G} \sum _{i=1}^{l_A} a_i(t)c_i(s)x(s)d\mu_s,\quad (Bx)(t)= \int\limits_{G}\sum_{j=1}^{l_B} b_j(t)e_j(s)x(s)d\mu_s,
\end{equation*}
for almost every $t$, where the index in $\mu_s$ indicates the variable of integration, $G\in \Sigma$, $a_i,b_j\in L_p(X,\mu)$, $c_i,e_j\in L_q(G,\mu)$, $i,j, l_A,l_B$ are positive integers such that $1\leq i\leq  l_A$, $1\leq j\leq  l_B$ and
$1\leq q\leq\infty$ with $\frac{1}{p}+\frac{1}{q}=1$.  If for all positive integers $m,k$ such that $1\leq m \leq l_A$, $1\leq k\leq l_B$ the following holds true
\begin{align} \label{CondOrtogonalityCorrMonomialIntOpGensepKernels}
Q_G(b_k,c_m)=0, \quad Q_G(e_k,a_m)=0, 
\end{align}
then $AB=\delta BA^d=0$ for some $\delta\in\mathbb{R}$ and some $d\in\mathbb{Z}$, $d>0$, where $Q_{\Lambda}(u,v)$, $\Lambda\in\Sigma$, is defined by \eqref{QGpairingDefinition}.
\end{corollary}

\begin{proof}
By applying Corollary \ref{corDiedroRelBothIntOPGenSeptedKernels}
when $G_A=G_B=G$ we have $AB=\delta BA^d$ if and only if
            \begin{equation*}
           \sum_{k=1}^{l_B}\sum_{m=1}^{l_A}a_m(t)Q_{G}(b_k,c_m)e_k(s)=
             \sum_{k=1}^{l_B}\sum_{i_1,\ldots,i_d=1}^{l_A} \delta b_k(t)Q_{G}(e_k,a_{i_1})\gamma_{i_1,\ldots,i_d}c_{i_d}(s)
            \end{equation*}
for almost every $(t,s)\in  X\times G$. By direct computation one gets that
if condition \eqref{CondOrtogonalityCorrMonomialIntOpGensepKernels} holds then
the commutation relation $AB=\delta BA^d$ is satisfied and both sides are equal zero. \QEDB
\end{proof}

\begin{corollary}[\cite{DjinjaEtAll_IntOpOverMeasureSpaces}]\label{CorIntOpGenSepKernelsOnlyOneTerm}
Let $(X,\Sigma,\mu)$ be $\sigma$-finite measure space. Let $A:L_p(X,\mu)\to L_p(X,\mu)$,  $B:L_p(X,\mu)\to L_p(X,\mu)$, $1\le p\le\infty$ be nonzero operators defined as follows
\begin{equation*}%\label{OpeqnsSKthmspeckernnosums}
  (Ax)(t)= \int\limits_{G_A} a(t)c(s)x(s)d\mu_s,\quad (Bx)(t)= \int\limits_{G_B} b(t)e(s)x(s)d\mu_s,
\end{equation*}
for almost every $t$, where the index in $\mu_s$ indicates the variable of integration, $G_A\in \Sigma$ and $G_B\in \Sigma$, $a,b\in L_p(X,\mu)$, $c\in L_q(G_A,\mu)$, $e\in L_q(G_B,\mu)$, $1\le q\le\infty$, $\frac{1}{p}+\frac{1}{q}=1$. Consider a polynomial defined by
$F(z)=\sum\limits_{j=0}^{n} \delta_j z^j$, where $\delta_j \in\mathbb{R}$, $j=0,\ldots,n$. Let $ G=G_A\cap G_B,$ and
\[
  k_1=\delta_1Q_{G_B} (a,e)+\sum_{j=2}^{n} \delta_j Q_{G_A} (a,c)^{j-1}Q_{G_B} (a,e), \quad k_2=Q_{G_A} (b,c),
\]
where $Q_{\Lambda}(u,v)$, $\Lambda\in\Sigma$, is defined by \eqref{QGpairingDefinition}. Then
$
  AB=BF(A)
$
if and only if  the following conditions are fulfilled:
\begin{enumerate}[label=\textup{\arabic*.}, ref=\arabic*]
  \item\label{CorIntOpGenSepKernelsOnlyOneTerm:item1}
       \begin{enumerate}[label=\textup{\alph*)}, ref=\alph*)]
       \item if almost every $(t,s)\in {\rm supp }\, b\times [({\rm supp }\, e)\cap G]$, then
               \begin{enumerate}[label=\textup{(\roman*)}, ref=(\roman*)]
                 \item if $k_2\not=0$ then  $c(s)k_1= \lambda e(s)$  and $a(t)=\frac{(\delta_0+\lambda)b(t)}{k_2}$ for some real scalar $\lambda$,
                 \item if $k_2=0$ then  $k_1c(s)=-\delta_0e(s)$.
               \end{enumerate}
       \item If $t\not\in {\rm supp }\, b$ then either $k_2=0$ or  $a(t)=0$ for almost all $t\not\in {\rm supp }\, b$.
       \item If $s\in G\setminus {\rm supp }\, e$ then either $k_1=0$ or  $c(s)=0$ for almost all \\ $s\in G\setminus  {\rm supp }\, e$.
        \end{enumerate}
\item\label{CorIntOpGenSepKernelsOnlyOneTerm:item2} $k_2 a(t)-\delta_0 b(t)=0$ for almost every $t\in X$ or $e(s)=0$ for almost every
    \\
    $s\in G_B\setminus G$.
\item\label{CorIntOpGenSepKernelsOnlyOneTerm:item3} $k_1=0$ or  $c(s)=0$ for almost every $s\in G_A\setminus G$.
\end{enumerate}
\end{corollary}

\begin{proof}  %\smartqed
This is a particular case of Theorem \ref{thmBothIntOPGenSeptedKernels} when $l_A=l_B=1$. In order to
explain how the more detailed conditions in \ref{CorIntOpGenSepKernelsOnlyOneTerm:item1}, \ref{CorIntOpGenSepKernelsOnlyOneTerm:item2} and \ref{CorIntOpGenSepKernelsOnlyOneTerm:item3} in Corollary \ref{CorIntOpGenSepKernelsOnlyOneTerm} arise, we present the independent detailed proof
following \cite{DjinjaEtAll_IntOpOverMeasureSpaces}.
We observe that since $a,b\in L_p(X,\mu),\ 1\leq p\leq\infty$, $c\in L_q(G_A,\mu)$, $e\in L_q(G_B,\mu)$, where $1\leq q\leq \infty$, with $\frac{1}{p}+\frac{1}{q}=1,$ then by applying H\"older inequality we conclude that operators $A$ and $B$ are well-defined. By direct calculation, we have
  \begin{eqnarray*}
  (A^2x)(t)&=&\int\limits_{G_A} a(t)c(s)(Ax)(s)d\mu_s=\int\limits_{G_A} a(t)c(s)a(s)d\mu_s\int\limits_{G_A} c(\tau_1)x(\tau_1)d\mu_{\tau_1}\\
  &=&Q_{G_A}(a,c)(Ax)(t),
  \\
    (A^3x)(t)&=& A(A^2x)(t)=Q_{G_A}(a,c)(A^2x)(t)=Q_{G_A}(a,c)^2(Ax)(t) %\int\limits_{\alpha_1}^{\beta_1} a(t)b(s)a(s)ds =
\end{eqnarray*}
  for almost every $t$.  We suppose that
  \begin{equation*} % \label{AonCompositionIntegralOperatorSplittedKernelGenSepKernelpaper}
  (A^{m}x)(t)=Q_{G_A}(a,c)^{m-1} (Ax)(t),\quad m=1,2,\ldots
  \end{equation*}
 for almost every $t$. Then
  \begin{eqnarray*}
    (A^{m+1}x)(t)&=& A(A^{m}x)(t)=Q_{G_A}(a,c)^{m-1} (A^2x)(t)=Q_{G_A}(a,c)^{m} (Ax)(t)
    %\int\limits_{\alpha_1}^{\beta_1} a(t)a(s)b(s)ds\int\limits_{\alpha_1}^{\beta_1} a(\tau_n)b(\tau_n)d\tau_n \cdot \hdots \cdot\int\limits_{\alpha_1}^{\beta_1} b(\tau_1)x(\tau_1)d\tau_1=\\ \label{AonCompositionIntegralOperatorSplittedKernel}
  %&=&a(t)\left( \int\limits_{\alpha_1}^{\beta_1} a(s)b(s)ds\right)^{n-1}\int\limits_{\alpha_1}^{\beta_1} b(\tau_1)x(\tau_1)d\tau_1.
\end{eqnarray*}
for almost every $t$. Then, we compute
  \begin{align*}%\label{CompABProofThmBothIntOpSkDIGenSepKernPaper}
  & \begin{array}{lll}
  (ABx)(t)&=& \int\limits_{G_A} a(t)c(s)b(s)d\mu_s\int\limits_{G_B} e(\tau_1) x(\tau_1)d\mu_{\tau_1} \\
  &=& k_2 \int\limits_{G_B} a(t) e(\tau_1) x(\tau_1)d\mu_{\tau_1},
  \end{array}\\
  & (F(A)x)(t)=\delta_0 x(t)+a(t)\sum_{j=1}^{n} \delta_j
    \left(Q_{G_A}(a,c) \right)^{j-1}\int\limits_{G_A} c(\tau) x(\tau) d\mu_\tau, \\
   & (BF(A)x)(t)=\delta_0 b(t)\int\limits_{G_B}  e(\tau_1)x(\tau_1)d\mu_{\tau_1}\\
   &+ b(t)\sum_{j=1}^{n} \delta_j \left(Q_{G_A}(a,c) \right)^{j-1}\int\limits_{G_B} e(\tau)a(\tau)d\mu_\tau \int\limits_{G_A} c(\tau_1) x(\tau_1) d\mu_{\tau_1} \\
%\label{EqBFAProofThmBothIntOpSkDIGenSepKernelPaper}
&= \delta_0 b(t)\int\limits_{G_B} e(\tau_1) x(\tau_1)d\mu_{\tau_1}+b(t)k_1\int\limits_{G_A} c(\tau_1) x(\tau_1)d\mu_{\tau_1},
  \end{align*}
  for almost every $t$. Thus, $(AB)x=(BF(A))x$ for all $x\in L_p( X,\mu)$ if and only if
 \[
   \int\limits_{G_B} [k_2a(t)-\delta_0 b(t)]e(s)x(s)d\mu_s= \int\limits_{G_A} k_1 b(t)c(s)x(s)d\mu_s.
 \]
Then by Lemma \ref{LemmaAllowInfSetsEqLp}, we conclude that, $AB=BF(A)$ if and only if
    \begin{enumerate}[label=\textup{\arabic*.}, ref=\arabic*]
      \item for almost every $(t,s)\in X\times G$,
      \begin{equation*} %\label{Eq1Cond1ProofThmBothIntOPSeparatedKernels}
      [k_2a(t)-\delta_0 b(t)]e(s)=k_1 b(t)c(s);
      \end{equation*}
      \item $k_2a(t)-\delta_0 b(t)=0$ for almost every $t\in X$ or $e(s)=0$ for almost every
      \\$s\in G_B\setminus G;$
      \item $k_1=0$ or $b(t)=0$ for almost every $t\in X$ or $c(s)=0$ for almost every $s\in G_A\setminus G$.
     \end{enumerate}
 We can rewrite the first condition as follows:
      \begin{enumerate}[leftmargin=*,label=\textup{\alph*)}, ref=\alph*)]
        \item Suppose  $(t,s)\in {\rm supp }\, b\times [({\rm supp }\, e)\cap G]$.
        \begin{enumerate}[leftmargin=*,label=\textup{(\roman*)}, ref=(\roman*)]
        \item  If  $k_2\neq 0,$ then
      $k_1\frac{c(s)}{e(s)}=k_2\frac{a(t)}{b(t)}-\delta_0=\lambda$ for some real scalar $\lambda$. From this, it follows that
       $k_1 c(s)=e(s)\lambda$ and $a(t)=\frac{\delta_0+\lambda}{k_2} b(t)$.
       \item If $k_2=0$ then $-\delta_0 b(t)e(s)=k_1 b(t)c(s)$, which yields $-k_1c(s)=-\delta_0e(s).$
       \end{enumerate}
        \item If $t\not\in {\rm supp }\, b$ then $k_2a(t)e(s)=0$ from which we get that either $k_2=0$ or  $a(t)=0$ for almost all $t\not\in {\rm supp }\, b $ or $e(s)=0$ almost everywhere (this implies $B=0$).
      \item If $s\in G\setminus  {\rm supp }\, e,$ then $k_1b(t)c(s)=0$  which implies that either $k_1=0$ or  $c(s)=0$ for almost all $s\in G\setminus  {\rm supp }\, e$, or $b(t)=0$ for almost all $t$ (this implies $B=0$). \QEDB
       \end{enumerate}
\end{proof}

\begin{proposition}\label{PropSimilarthmBothIntOpGensepKernellsCommutativityCommutator}
Let $(X,\Sigma,\mu)$ be $\sigma$-finite measure space. Let
$$A:L_p(X,\mu)\to L_p(X,\mu),\ B:L_p(X,\mu)\to L_p(X,\mu),\ 1\le p\le\infty$$
be nonzero operators defined as follows
\begin{equation*}
  (Ax)(t)= \int\limits_{G} \sum _{i=1}^{l_A} a_i(t)c_i(s)x(s)d\mu_s,\quad (Bx)(t)= \int\limits_{G}\sum_{j=1}^{l_B} b_j(t)e_j(s)x(s)d\mu_s,
\end{equation*}
for almost every $t$, where the index in $\mu_s$ indicates the variable of integration, $G_A\in \Sigma$ and $G_B\in \Sigma$, $a_i,b_j\in L_p(X,\mu)$, $c_i\in L_q(G,\mu)$, $e_j\in L_q(G,\mu)$, $i,j, l_A,l_B$ are positive integers such that $1\leq i\leq  l_A$, $1\leq j\leq  l_B$ and  $1\leq q\leq\infty$ with $\frac{1}{p}+\frac{1}{q}=1$. Let $ G=G_A\cap G_B$ and
$Q_{\Lambda}(u,v)$, $\Lambda\in\Sigma$, is defined by \eqref{QGpairingDefinition}. Then,  for almost every $t\in  X$,
\begin{enumerate}[label=\textup{\arabic*.}, ref=\arabic*]
  \item $(AB)x(t)=\int\limits_{G} \sum\limits_{k=1}^{l_B}\sum\limits_{m=1}^{l_A}a_m(t)Q_{G_A}(b_k,c_m)e_k(s) x(s)d\mu_s$;
  \item $(BA)x(t)=\int\limits_{G} \sum\limits_{k=1}^{l_B}\sum\limits_{i_1=1}^{l_A}  b_k(t)Q_{G_B}(e_k,a_{i_1})c_{i_1}(s)x(s)d\mu_s $;
  \item $(AB-BA)x(t)=\int\limits_{G} \left(\sum\limits_{k=1}^{l_B}\sum\limits_{m=1}^{l_A}a_m(t)Q_{G_A}(b_k,c_m)e_k(s)\right.\\
  \left. \hspace{2.5cm} -\sum\limits_{k=1}^{l_B}\sum\limits_{i_1=1}^{l_A}  b_k(t)Q_{G_B}(e_k,a_{i_1})c_{i_1}(s)\right)x(s)d\mu_s$.
\end{enumerate}
\end{proposition}

\begin{proof}
It follows from \eqref{CompABProofThmBothIntOpGenSeparatedKnDI} and \eqref{EqBFAProofThmBothIntOpGenSepratKnDI} when $F(z)=z$. \QEDB
\end{proof}

\begin{proposition}\label{PropositionCommutatorZeroFourSequenceCond}
 Let $(\mathbb{R},\Sigma,\mu)$ be the standard Lebesgue measure space on the real line. Let $A:L_p(\mathbb{R},\mu)\to L_p(\mathbb{R},\mu)$, $B:L_p(\mathbb{R},\mu)\to L_p(\mathbb{R},\mu)$, $1\le p\le \infty$ be operators defined as follows
  \begin{eqnarray*}
     (A x)(t)&=&\int\limits_{\alpha_1}^{\beta_1}I_{[\alpha,\beta]}(t)[\theta_{A,1}\sin(\omega t)\cos(\omega s)+\theta_{A,2}\cos(\omega t)\cos(\omega s)\\
     %\label{OpAFourerSeq}
     && +\theta_{A,3}\sin(\omega t)\sin(\omega s)+\theta_{A,4}\cos(\omega t)\sin(\omega s)]x(s)d\mu_s,
   \\ \nonumber
      (Bx)(t)&=&  \int\limits_{\alpha_1}^{\beta_1}I_{[\alpha,\beta]}(t)\left(\theta_{B,1}\sin(\omega t)\cos(\omega s)+\theta_{B,2}\cos(\omega t)\cos(\omega s)\right. \\ \nonumber
     && \left. +\theta_{B,3}\sin(\omega t)\sin(\omega s)+\theta_{B,4}\cos(\omega t)\sin(\omega s)\right) x(s)d\mu_s,
  \end{eqnarray*}
 for almost every $t$, where $\theta_{A,i}, \theta_{B,i}, \omega \in\mathbb{R}$, $i=1,2,3,4$,  $\delta \in\mathbb{R}\setminus\{0\}$, $\alpha,\beta,\alpha_1,\beta_1 \in\mathbb{R}$, $\alpha_1<\beta_1$, $\alpha\le \alpha_1$, $\beta\ge \beta_1$, $I_E(\cdot)$ is the indicator function of the set $E$, the number $\frac{\omega }{\pi}(\beta_1-\alpha_1)\in \mathbb{Z}$ or $\frac{\omega }{\pi}(\beta_1+\alpha_1)\in \mathbb{Z}$, and
 $\sigma_1,\sigma_2\in\mathbb{R}$ are such that
 \begin{align*}%\label{ConstantSigma1Case2Omega}
       & \sigma_1=\int\limits_{\alpha_1}^{\beta_1} (\sin(\omega s))^2 d\mu_s=\left\{\begin{array}{cc}
                         0, & \mbox{ if } \omega=0 \\
                         \frac{\beta_1-\alpha_1}{2}-\frac{\cos(\omega(\alpha_1+\beta_1))\sin(\omega(\beta_1-\alpha_1))}{2\omega}, & \mbox{ if } \omega\not=0,
                       \end{array} \right.\\ %\label{ConstantSigma2Case2Omega}
    &  \sigma_2=\int\limits_{\alpha_1}^{\beta_1} (\cos(\omega s))^2 d\mu_s=\beta_1-\alpha_1-\sigma_1\\
    &=\left\{\begin{array}{cc}
                         \beta_1-\alpha_1, & \mbox{ if } \omega=0 \\
                         \frac{\beta_1-\alpha_1}{2}+\frac{\cos(\omega(\alpha_1+\beta_1))\sin(\omega(\beta_1-\alpha_1))}{2\omega}, & \mbox{ if } \omega\not=0.
                       \end{array} \right.
        \end{align*}
   Then,  for almost every $t$,
%  \resizebox{0.92\hsize}{!}{
\begin{align}\label{CommutatorABFourierSequenceKernelOmegaOrtog}
  % &&\resizebox{0.98\hsize}{!}{$
  & (AB-BA)x(t) =
   \\
&
 \resizebox{0.99\hsize}{!}{$\displaystyle
\begin{array}{c}
\int\limits_{\alpha_1}^{\beta_1}I_{[\alpha,\beta]}(t)\big((\theta_{A,3}\theta_{B,1}\sigma_1-\theta_{B,3}\theta_{A,1}\sigma_1
+\theta_{A,1}\theta_{B,2}\sigma_2-\theta_{B,1}\theta_{A,2}\sigma_2)\sin(\omega t)\cos(\omega s)  \\
+ (\theta_{B,1}\theta_{A,4} - \theta_{A,1}\theta_{B,4})\sigma_1\cos(\omega t)\cos(\omega s)
+ (\theta_{A,1}\theta_{B,4} - \theta_{B,1}\theta_{A,4})\sigma_2\sin(\omega t)\sin(\omega s) \\
+(\theta_{A,4}\theta_{B,3}\sigma_1-\theta_{B,4}\theta_{A,3}\sigma_1
   +\theta_{A,2}\theta_{B,4}\sigma_2-\theta_{B,2}\theta_{A,4}\sigma_2)\cos(\omega t)\sin(\omega s)\big)x(s)d\mu_s
\end{array} $}      \nonumber
  \end{align}
  Moreover, if $\sigma_1\not=0$, $\sigma_2\not=0$, then  $AB=BA$ if and only if
  \begin{align*}
&\theta_{A,3}\theta_{B,1}\sigma_1-\theta_{B,3}\theta_{A,1}\sigma_1=\theta_{B,1}\theta_{A,2}\sigma_2-\theta_{A,1}\theta_{B,2}\sigma_2, \quad
\theta_{A,4}\theta_{B,1}=\theta_{B,4}\theta_{A,1} \\
& \theta_{A,4}\theta_{B,3}\sigma_1-\theta_{B,4}\theta_{A,3}\sigma_1=\theta_{B,2}\theta_{A,4}\sigma_2-\theta_{A,2}\theta_{B,4}\sigma_2.
\end{align*}
\end{proposition}

\begin{proof}
  Operators $A$, $B$ are well defined
  and bounded. Equality \eqref{CommutatorABFourierSequenceKernelOmegaOrtog} follows from Proposition
  \ref{PropSimilarthmBothIntOpGensepKernellsCommutativityCommutator} when $a_i(t)=I_{[\alpha,\beta]}(t)\sin(\omega t)$, $i=1,3$,
   $a_i(t)=I_{[\alpha,\beta]}(t)\cos(\omega t)$,
   \\
   $i=2,4,$ $t\in\mathbb{R},$ $c_i(s)=\cos(\omega s),$
   $i=1,2,$
   $c_i(s)=\sin(\omega s)$,   $i=3,4$, $s\in [\alpha_1,\beta_1]$ and $\alpha_1,\beta_1,\omega$ satisfy
   $\int\limits_{\alpha_1}^{\beta_1} \sin(\omega s) \cos(\omega s)\,d\mu_s=0$.
   Then, by applying Lemma \ref{LemmaAllowInfSetsEqLp} to solve $AB-BA=0$, we complete the proof.
   \qed
\end{proof}

We use  Proposition \ref{PropositionCommutatorZeroFourSequenceCond} to establish commutativity of some operators in Case 2 of Example \ref{ExampleRepintopgensepkernelfouriersequence}.

\section{Examples}\label{SecRepreBothLIExample}

In this section, we  use the derived conditions in Section \ref{SecRepreBothLI} to construct examples of integral operators on $L_p$ spaces, with separable kernels  representing  covariance commutation relations associated to monomials, for  kernels involving multi-parameter trigonometric functions, polynomials and Laurent polynomials on bounded intervals. %Commutators of these operator are computed and exact conditions for commutativity of these operators in terms of the parameters are obtained.

\begin{example}\label{ExampleRepintopgensepkernelfouriersequence}
{\rm
Let $(\mathbb{R}, \Sigma,\mu)$ be the standard Lebesgue measure space.  %Consider  integral operators acting on $L_p(\mathbb{R},\mu)$ for $1< p < \infty$.
Let operators $$A:L_p(\mathbb{R},\mu)\to L_p(\mathbb{R},\mu),\  B:L_p(\mathbb{R},\mu)\to L_p(\mathbb{R},\mu),\ 1< p<\infty$$
be defined as follows
\[
 \textstyle (Ax)(t)= \int\limits_{\alpha_1}^{\beta_1} \sum\limits_{m=1}^4 a_m(t)c_m(s)x(s)d\mu_s,\quad (Bx)(t)= \int\limits_{\alpha_1}^{\beta_1} \sum\limits_{k=1}^4 b_k(t)e_k(s)x(s)d\mu_s,
\]
for almost every $t$, where the index in  $\mu$ indicates the variable of integration,
\begin{align*}
&\resizebox{0.93\hsize}{!}{$\displaystyle \sum_{m=1}^4 a_m(t)c_m(s)= I_{[\alpha,\beta]}(t)\left(\theta_{A,1}\sin\left( \frac{2\pi m_1 t}{\lambda_t} \right) \cos\left( \frac{2\pi k_1 s}{\lambda_s}\right)+\theta_{A_2}\cos\left( \frac{2\pi m_2 t}{\lambda_t} \right)\right. $}
\\
&
\resizebox{0.95\hsize}{!}{$\displaystyle + \cos\left( \frac{2\pi k_2 s}{\lambda_s}\right)+
 \left.\theta_{A,3}\sin\left( \frac{2\pi m_3 t}{\lambda_t} \right) \sin\left( \frac{2\pi k_3 s}{\lambda_s}\right) + \theta_{A,4}\cos\left( \frac{2\pi m_4 t}{\lambda_t} \right) \sin\left( \frac{2\pi k_4 s}{\lambda_s}\right)\right), $}
 \\
&
\resizebox{0.93\hsize}{!}{$\displaystyle \sum_{k=1}^4 b_k(t)e_k(s)= I_{[\alpha,\beta]}(t)\left(\theta_{B,1}\sin\left( \frac{2\pi m_1 t}{\lambda_t} \right) \cos\left( \frac{2\pi k_1 s}{\lambda_s}\right)+\theta_{B,2}\cos\left( \frac{2\pi m_2 t}{\lambda_t} \right)\right.
$}
\\
  &
 \resizebox{0.95\hsize}{!}{$\displaystyle +\cos\left( \frac{2\pi k_2 s}{\lambda_s}\right)+
 \left.\theta_{B,3}\sin\left( \frac{2\pi m_3 t}{\lambda_t} \right) \sin\left( \frac{2\pi k_3 s}{\lambda_s}\right) + \theta_{B4}\cos\left( \frac{2\pi m_4 t}{\lambda_t} \right) \sin\left( \frac{2\pi k_4 s}{\lambda_s}\right)\right), $}
\end{align*}
for almost every $ (t,s)\in \mathbb{R}\times[\alpha_1,\beta_1]$, $\alpha,\, \alpha_1,\, \beta,\, \beta_1$ are real constants such that $\alpha_1<\beta_1$, $\alpha\le \alpha_1,$ $\beta\ge \beta_1$  and $I_{E}(t)$ is the indicator function of the set $E$, $\lambda_t,\lambda_s\in \mathbb{R}\setminus \{0\}$, $m_1,m_2,m_3,m_4, k_1,k_2,k_3,k_4 \in \mathbb{R}$. So we have
\begin{gather*}
\begin{array}{ll}
a_1(t)=\theta_{A,1}I_{[\alpha,\beta]}(t)\sin\left( \frac{2\pi m_1 t}{\lambda_t} \right), \ & a_2(t)=\theta_{A,2}I_{[\alpha,\beta]}(t)\cos\left(\frac{2\pi m_2 t}{\lambda_t} \right), \\ a_3(t)=\theta_{A,3}I_{[\alpha,\beta]}(t)\sin\left( \frac{2\pi m_3 t}{\lambda_t} \right), &
a_4(t)=\theta_{A,4}I_{[\alpha,\beta]}(t)\cos\left( \frac{2\pi m_4 t}{\lambda_t} \right),  \\
c_1(s)= \cos\left( \frac{2\pi k_1 s}{\lambda_s} \right) & c_2(s)=\cos\left( \frac{2\pi k_2 s}{\lambda_s} \right), \\
c_3(s)=\sin\left( \frac{2\pi k_3 s}{\lambda_s} \right) & c_4(s)=\sin\left( \frac{2\pi k_4 s}{\lambda_s} \right), \\
b_1(t)=\theta_{B,1}I_{[\alpha,\beta]}(t)\sin\left( \frac{2\pi m_1 t}{\lambda_t} \right), \ & b_2(t)=\theta_{B,2} I_{[\alpha,\beta]}(t) \cos\left( \frac{2\pi m_2 t}{\lambda_t} \right), \\ b_3(t)=\theta_{B,3}I_{[\alpha,\beta]}(t)\sin\left( \frac{2\pi m_3 t}{\lambda_t} \right), &
b_4(t)=\theta_{B,4}I_{[\alpha,\beta]}(t)\cos\left( \frac{2\pi m_4 t}{\lambda_t} \right),  \\
e_1(s)= \cos\left( \frac{2\pi k_1 s}{\lambda_s} \right) & e_2(s)=\cos\left( \frac{2\pi k_2 s}{\lambda_s} \right), \\
e_3(s)=\sin\left( \frac{2\pi k_3 s}{\lambda_s} \right) & e_4(s)=\sin\left( \frac{2\pi k_4 s}{\lambda_s} \right), \\
\end{array}
\end{gather*}

These operators are well defined, indeed  for any fixed $p$, $1<p<\infty$, functions $a_i,b_j\in L_p(\mathbb{R},\mu)$, $c_i,e_j\in L_q([\alpha_1,\beta_1],\mu)$, $i=1,2,3,4$, $j=1,2,3,4$, $1<q<\infty $, $\frac{1}{p}+\frac{1}{q}=1$. In fact,
\begin{equation*}
  \int\limits_{\mathbb{R}} |a_i (s)|^p d\mu=\int\limits_{\alpha}^\beta |\tilde{a}_i (s)|^p d\mu<\infty,\ i=1,2,3,4,
\end{equation*}
where $\tilde{a}_j(s)=\theta_{A,j}\sin\left( \frac{2\pi m_j s}{\lambda_s} \right)$, $j=1,3$, $\tilde{a}_j(s)=\theta_{A,j}\cos\left( \frac{2\pi m_j s}{\lambda_s} \right)$, $j=2,4$ are  continuous functions. Similarly,
 \begin{equation*}
  \int\limits_{\mathbb{R}} |b_j (s)|^p d\mu=\int\limits_{\alpha}^\beta |\tilde{b}_j (s)|^p d\mu<\infty,\ j=1,2,3,4.
\end{equation*}
where $\tilde{b}(s)=\cos\left( \frac{2\pi m_j s}{\lambda_s} \right)$, $j=1,2$, $\tilde{b}_j(s)=\sin\left( \frac{2\pi m_j t}{\lambda_t} \right)$, $j=3,4$ are continuous functions. Analogously, we have
 \begin{equation*}
  \int\limits_{\alpha_1}^{\beta_1} |{c}_i (s)|^q d\mu<\infty,\ i=1,2,3,4,\quad  \int\limits_{\alpha_1}^{\beta_1} |{e}_j (s)|^q d\mu<\infty,\ j=1,2,3,4.
\end{equation*}
Note that in this case conditions \ref{thmBothIntOPGenSeptedKernels:cond1}, \ref{thmBothIntOPGenSeptedKernels:cond2} and \ref{thmBothIntOPGenSeptedKernels:cond3} of Corollary \ref{corDiedroRelBothIntOPGenSeptedKernels} reduces just to condition \ref{thmBothIntOPGenSeptedKernels:cond1} because the sets $G_A=G_B=[\alpha_1,\beta_1]$, and so $G=[\alpha_1,\beta_1]$, $G_A\setminus G=G_B\setminus G=\emptyset$. Therefore, according to Remark \ref{RemOpDefInSameIntervalGenSepKernels}, conditions \ref{thmBothIntOPGenSeptedKernels:cond2} and \ref{thmBothIntOPGenSeptedKernels:cond3} are taken on a set of measure zero, they are fulfilled.

Consider the polynomial $F(z)=\delta z^2$, $ \delta\in\mathbb{R}$. We will seek for conditions such that operators satisfy $AB=BF(A)$. By applying Corollary \ref{corDiedroRelBothIntOPGenSeptedKernels} when $n=2$ we have
%\begin{align}
\begin{equation}\label{CondCRWaveletExample4TermsFourierSerie}
 \sum_{k=1}^4 \sum_{m=1}^4 a_m(t)Q_{G_A}(b_k,c_m)e_k(s) =  \sum_{i=1}^4 \sum_{i_1,i_2=1}^4 [\delta b_i(t)a_m(t)Q_{G_B}(e_i,a_{i_1})\gamma_{i_1,i_2}c_{i_2}(s)%+\\
%& \quad + a_m(t)Q_{G_A}(b_2,c_m)e_2(s)] \\
%& \quad \textstyle =  a_1(t)e_1(s)+2a_2(t)e_2(s)+2a_3(t)e_2(s)\\
%& \quad \textstyle =\ I_{[\alpha,\beta]}(t)\frac{2}{\pi}(\cos t\cos s+2\sin t\sin s + 2 \cos t \sin s)
  %\end{align}
 \end{equation}
   for almost every $(t,s)\in \mathbb{R}\times [\alpha_1,\beta_1]$.
  So, we have
\begin{align*}
& a_i(t)e_j(s)=\left\{\begin{array}{ll}
  \theta_{A,i}I_{[\alpha,\beta]}(t)\sin\left(\frac{2\pi m_i t}{\lambda_t}\right)\cos\left(\frac{2\pi k_j s}{\lambda_s}\right), & i=1,3,\ j=1,2\\
  \theta_{A,i}I_{[\alpha,\beta]}(t)\sin\left(\frac{2\pi m_i t}{\lambda_t}\right)\sin\left(\frac{2\pi k_j s}{\lambda_s}\right), & i=1,3,\ j=3,4\\
   \theta_{A,i}I_{[\alpha,\beta]}(t)\cos\left(\frac{2\pi m_i t}{\lambda_t}\right)\cos\left(\frac{2\pi k_j s}{\lambda_s}\right), & i=2,4,\ j=1,2\\
  \theta_{A,i}I_{[\alpha,\beta]}(t)\cos\left(\frac{2\pi m_i t}{\lambda_t}\right)\sin\left(\frac{2\pi k_j s}{\lambda_s}\right), & i=2,4,\ j=3,4,
              \end{array}\right.  \\
& b_i(t)c_j(s)=\left\{\begin{array}{ll}
  \theta_{B,i}I_{[\alpha,\beta]}(t)\sin\left(\frac{2\pi m_i t}{\lambda_t}\right)\cos\left(\frac{2\pi k_j s}{\lambda_s}\right), & i=1,3,\ j=1,2\\
  \theta_{B,i}I_{[\alpha,\beta]}(t)\sin\left(\frac{2\pi m_i t}{\lambda_t}\right)\sin\left(\frac{2\pi k_j s}{\lambda_s}\right), & i=1,3,\ j=3,4\\
   \theta_{B,i}I_{[\alpha,\beta]}(t)\cos\left(\frac{2\pi m_i t}{\lambda_t}\right)\cos\left(\frac{2\pi k_j s}{\lambda_s}\right), & i=2,4,\ j=1,2\\
  \theta_{B,i}I_{[\alpha,\beta]}(t)\cos\left(\frac{2\pi m_i t}{\lambda_t}\right)\sin\left(\frac{2\pi k_j s}{\lambda_s}\right), & i=2,4,\ j=3,4.
              \end{array}\right.
   \end{align*}
There are many cases but we will consider two cases.
\subsection*{Case 1:}
%In Case 1, we consider that the coefficient of $t$ and $s$ in the argument of functions $\sin (wt)\cos(us)$, are  differents as ordered pairs, that is, $(u_1,w_1)\not=(u_2,w_2)$.
If $(\frac{m_i}{\lambda_t},\frac{k_j}{\lambda_s})\not=(\frac{m_u}{\lambda_t},\frac{k_w}{\lambda_s})$, whenever $((i,j)\not=(u,w))$, $i,j,u,w=1,2,3,4$ and by putting  $a_j(t)=\theta_{A,j}\tilde{a}_j(t)$, $b_j(t)=\theta_{B,j}\tilde{b}_j(t)$ we get from  \eqref{CondCRWaveletExample4TermsFourierSerie} the following
\begin{equation*}
  a_i(t)e_j(t)Q_{G_A}(b_j,c_i)=\delta b_i(t)c_j(s)\sum_{k=1}^{4}Q_{G_A}(e_i,a_k)\gamma_{k,j},\ i,j=1,2,3,4,
\end{equation*}
for almost every $(t,s)\in\mathbb{R}\times [\alpha_1,\beta_1]$, where $\gamma_{k,i}=Q_{G_A}(a_i,c_k)$. Hence, we have
\begin{equation*}
 Q_{G_A}(\tilde{b}_j,c_i)\theta_{A,i}\theta_{B,j}=\delta \theta_{B,i}\theta_{A,j} \sum_{k=1}^4 Q_{G_B}(e_i,a_k)Q_{G_A}(\tilde{a}_j,c_k), \ i,j=1,2,3,4.
\end{equation*}
 By using bilinearity and symmetry of $Q_{\Lambda}(\cdot,\cdot)$ we can expand the previous equations into the following system of equations:
\begin{eqnarray*}
&& \hspace{-1cm} Q_{G_A}(\tilde b_1,c_1)\theta_{A,1}\theta_{B,1}=\delta \theta_{B,1}\theta_{A,1}\left( \theta_{A,1}Q_{G_B}(\tilde{a}_1,e_1)Q_{G_A}(\tilde{a}_1,c_1)\right.\\
 &&+\left.\theta_{A,2}Q_{G_B}(\tilde{a}_2,e_1)Q_{G_A}(\tilde{a}_1,c_2)+\theta_{A,3}Q_{G_B}(\tilde{a}_3,e_1)Q_{G_A}(\tilde{a}_1,c_3)\right.\\
 & &+\left.\theta_{A,4}Q_{G_B}(\tilde{a}_4,e_1)Q_{G_A}(\tilde{a}_1,c_4)\right),
 \\
&& \hspace{-1cm}  Q_{G_A}(\tilde b_2,c_1)\theta_{A,2}\theta_{B,1}=\delta \theta_{B,1}\theta_{A,2}\left( \theta_{A,1}Q_{G_B}(\tilde{a}_1,e_1)Q_{G_A}(\tilde{a}_2,c_1)\right.\\
 &&+\left.\theta_{A,2}Q_{G_B}(\tilde{a}_2,e_1)Q_{G_A}(\tilde{a}_2,c_2)+\theta_{A,3}Q_{G_B}(\tilde{a}_3,e_1)Q_{G_A}(\tilde{a}_2,c_3)\right.\\
 & &+\left.\theta_{A,4}Q_{G_B}(\tilde{a}_4,e_1)Q_{G_A}(\tilde{a}_2,c_4)\right),
\\
&& \hspace{-1cm} Q_{G_A}(\tilde b_3,c_1)\theta_{A,1}\theta_{B,3}=\delta \theta_{B,1}\theta_{A,3}\left( \theta_{A,1}Q_{G_B}(\tilde{a}_1,e_1)Q_{G_A}(\tilde{a}_3,c_1)\right.\\
 & &+\left.\theta_{A,2}Q_{G_B}(\tilde{a}_2,e_1)Q_{G_A}(\tilde{a}_3,c_2)+\theta_{A,3}Q_{G_B}(\tilde{a}_3,e_1)Q_{G_A}(\tilde{a}_3,c_3)\right.\\
 & &+\left.\theta_{A,4}Q_{G_B}(\tilde{a}_4,e_1)Q_{G_A}(\tilde{a}_3,c_4)\right),
\\
&& \hspace{-1cm} Q_{G_A}(\tilde b_4,c_1)\theta_{A,1}\theta_{B,4}=\delta \theta_{B,1}\theta_{A,4}\left( \theta_{A,1}Q_{G_B}(\tilde{a}_1,e_1)Q_{G_A}(\tilde{a}_4,c_1)\right.\\
 & &+\left.\theta_{A,2}Q_{G_B}(\tilde{a}_2,e_1)Q_{G_A}(\tilde{a}_4,c_2)+\theta_{A,3}Q_{G_B}(\tilde{a}_3,e_1)Q_{G_A}(\tilde{a}_4,c_3)\right.\\
 & &+\left.\theta_{A,4}Q_{G_B}(\tilde{a}_4,e_1)Q_{G_A}(\tilde{a}_4,c_4)\right),
\\
&& \hspace{-1cm} Q_{G_A}(\tilde b_1,c_2)\theta_{A,2}\theta_{B,1}=\delta \theta_{B,2}\theta_{A,1}\left( \theta_{A,1}Q_{G_B}(\tilde{a}_1,e_2)Q_{G_A}(\tilde{a}_1,c_1)\right.\\
 & &+\left.\theta_{A,2}Q_{G_B}(\tilde{a}_2,e_2)Q_{G_A}(\tilde{a}_1,c_2)+\theta_{A,3}Q_{G_B}(\tilde{a}_3,e_2)Q_{G_A}(\tilde{a}_1,c_3)\right.\\
 & &+\left.\theta_{A,4}Q_{G_B}(\tilde{a}_4,e_2)Q_{G_A}(\tilde{a}_1,c_4)\right),
 \\
&& \hspace{-1cm} Q_{G_A}(\tilde b_2,c_2)\theta_{A,2}\theta_{B,2}=\delta \theta_{B,2}\theta_{A,2}\left( \theta_{A,1}Q_{G_B}(\tilde{a}_1,e_2)Q_{G_A}(\tilde{a}_2,c_1)\right.\\
 & &+\left.\theta_{A,2}Q_{G_B}(\tilde{a}_2,e_2)Q_{G_A}(\tilde{a}_2,c_2)+\theta_{A,3}Q_{G_B}(\tilde{a}_3,e_2)Q_{G_A}(\tilde{a}_2,c_3)\right.\\
 & &+\left.\theta_{A,4}Q_{G_B}(\tilde{a}_4,e_2)Q_{G_A}(\tilde{a}_2,c_4)\right),
 \\
&& \hspace{-1cm} Q_{G_A}(\tilde b_3,c_2)\theta_{A,2}\theta_{B,3}=\delta \theta_{B,2}\theta_{A,3}\left( \theta_{A,1}Q_{G_B}(\tilde{a}_1,e_2)Q_{G_A}(\tilde{a}_3,c_1)\right.\\
 & &+\left.\theta_{A,2}Q_{G_B}(\tilde{a}_2,e_2)Q_{G_A}(\tilde{a}_3,c_2)+\theta_{A,3}Q_{G_B}(\tilde{a}_3,e_2)Q_{G_A}(\tilde{a}_3,c_3)\right.\\
 & &+\left.\theta_{A,4}Q_{G_B}(\tilde{a}_4,e_2)Q_{G_A}(\tilde{a}_3,c_4)\right),
 \\
&& \hspace{-1cm} Q_{G_A}(\tilde b_4,c_2)\theta_{A,2}\theta_{B,4}=\delta \theta_{B,2}\theta_{A,4}\left( \theta_{A,1}Q_{G_B}(\tilde{a}_1,e_2)Q_{G_A}(\tilde{a}_2,c_1)\right.\\
 & &+\left.\theta_{A,2}Q_{G_B}(\tilde{a}_2,e_2)Q_{G_A}(\tilde{a}_4,c_2)+\theta_{A,3}Q_{G_B}(\tilde{a}_3,e_2)Q_{G_A}(\tilde{a}_4,c_3)\right.\\
 & &+\left.\theta_{A,4}Q_{G_B}(\tilde{a}_4,e_2)Q_{G_A}(\tilde{a}_4,c_4)\right),
 \\
&& \hspace{-1cm} Q_{G_A}(\tilde b_1,c_3)\theta_{A,3}\theta_{B,1}=\delta \theta_{B,3}\theta_{A,1}\left( \theta_{A,1}Q_{G_B}(\tilde{a}_1,e_3)Q_{G_A}(\tilde{a}_1,c_1)\right.\\
 & &+\left.\theta_{A,2}Q_{G_B}(\tilde{a}_2,e_3)Q_{G_A}(\tilde{a}_1,c_2)+\theta_{A,3}Q_{G_B}(\tilde{a}_3,e_3)Q_{G_A}(\tilde{a}_1,c_3)\right.\\
 & &+\left.\theta_{A,4}Q_{G_B}(\tilde{a}_4,e_3)Q_{G_A}(\tilde{a}_1,c_4)\right),
 \\
 && \hspace{-1cm} Q_{G_A}(\tilde b_2,c_3)\theta_{A,2}\theta_{B,3}=\delta \theta_{B,2}\theta_{A,3}\left( \theta_{A,1}Q_{G_B}(\tilde{a}_1,e_3)Q_{G_A}(\tilde{a}_2,c_1)\right.\\
 & &+\left.\theta_{A,2}Q_{G_B}(\tilde{a}_2,e_3)Q_{G_A}(\tilde{a}_2,c_2)+\theta_{A,3}Q_{G_B}(\tilde{a}_3,e_3)Q_{G_A}(\tilde{a}_2,c_3)\right.\\
 & &+\left.\theta_{A,4}Q_{G_B}(\tilde{a}_4,e_3)Q_{G_A}(\tilde{a}_2,c_4)\right),
 \\
 && \hspace{-1cm} Q_{G_A}(\tilde b_3,c_3)\theta_{A,3}\theta_{B,3}=\delta \theta_{B,3}\theta_{A,3}\left( \theta_{A,1}Q_{G_B}(\tilde{a}_1,e_3)Q_{G_A}(\tilde{a}_3,c_1)\right.\\
 & &+\left.\theta_{A,2}Q_{G_B}(\tilde{a}_2,e_3)Q_{G_A}(\tilde{a}_3,c_2)+\theta_{A,3}Q_{G_B}(\tilde{a}_3,e_3)Q_{G_A}(\tilde{a}_3,c_3)\right.\\
 & &+\left.\theta_{A,4}Q_{G_B}(\tilde{a}_4,e_3)Q_{G_A}(\tilde{a}_3,c_4)\right),
 \\
 && \hspace{-1cm} Q_{G_A}(\tilde b_4,c_3)\theta_{A,3}\theta_{B,4}=\delta \theta_{B,3}\theta_{A,4}\left( \theta_{A,1}Q_{G_B}(\tilde{a}_1,e_3)Q_{G_A}(\tilde{a}_4,c_1)\right.\\
 & &+\left.\theta_{A,2}Q_{G_B}(\tilde{a}_2,e_3)Q_{G_A}(\tilde{a}_4,c_2)+\theta_{A,3}Q_{G_B}(\tilde{a}_3,e_3)Q_{G_A}(a_4,c_3)\right.\\
 & &+\left.\theta_{A,4}Q_{G_B}(\tilde{a}_4,e_3)Q_{G_A}(\tilde{a}_4,c_4)\right),
 \\
 && \hspace{-1cm} Q_{G_A}(\tilde b_1,c_4)\theta_{A,4}\theta_{B,1}=\delta \theta_{B,4}\theta_{A,1}\left( \theta_{A,1}Q_{G_B}(\tilde{a}_1,e_4)Q_{G_A}(\tilde{a}_1,c_1)\right.\\
 & &+\left.\theta_{A,2}Q_{G_B}(\tilde{a}_2,e_4)Q_{G_A}(\tilde{a}_1,c_2)+\theta_{A,3}Q_{G_B}(\tilde{a}_3,e_4)Q_{G_A}(a_1,c_3)\right.\\
 & &+\left.\theta_{A,4}Q_{G_B}(\tilde{a}_4,e_4)Q_{G_A}(\tilde{a}_1,c_4)\right),
 \\
 && \hspace{-1cm} Q_{G_A}(\tilde b_2,c_4)\theta_{A,4}\theta_{B,2}=\delta \theta_{B,4}\theta_{A,2}\left( \theta_{A,1}Q_{G_B}(\tilde{a}_1,e_4)Q_{G_A}(\tilde{a}_2,c_1)\right.\\
 & &+\left.\theta_{A,2}Q_{G_B}(\tilde{a}_2,e_4)Q_{G_A}(\tilde{a}_2,c_2)+\theta_{A,3}Q_{G_B}(\tilde{a}_3,e_4)Q_{G_A}(a_2,c_3)\right.\\
 & &+\left.\theta_{A,4}Q_{G_B}(\tilde{a}_4,e_4)Q_{G_A}(\tilde{a}_2,c_4)\right),
\\
 && \hspace{-1cm} Q_{G_A}(\tilde b_3,c_4)\theta_{A,4}\theta_{B,3}=\delta \theta_{B,4}\theta_{A,3}\left( \theta_{A,1}Q_{G_B}(\tilde{a}_1,e_4)Q_{G_A}(\tilde{a}_3,c_1)\right.\\
 & &+\left.\theta_{A,2}Q_{G_B}(\tilde{a}_2,e_4)Q_{G_A}(\tilde{a}_3,c_2)+\theta_{A,3}Q_{G_B}(\tilde{a}_3,e_4)Q_{G_A}(\tilde{a}_3,c_3)\right.\\
 & &+\left.\theta_{A,4}Q_{G_B}(\tilde{a}_4,e_4)Q_{G_A}(\tilde{a}_3,c_4)\right),
 \\
 && \hspace{-1cm}  Q_{G_A}(\tilde b_4,c_4)\theta_{A,4}\theta_{B,4}=\delta \theta_{B,4}\theta_{A,4}\left( \theta_{A,1}Q_{G_B}(\tilde{a}_1,e_4)Q_{G_A}(\tilde{a}_4,c_1)\right.\\
 & &+\left.\theta_{A,2}Q_{G_B}(\tilde{a}_2,e_4)Q_{G_A}(\tilde{a}_4,c_2)+\theta_{A,3}Q_{G_B}(\tilde{a}_3,e_4)Q_{G_A}(\tilde{a}_4,c_3)\right.\\
 & &+\left.\theta_{A,4}Q_{G_B}(\tilde{a}_4,e_4)Q_{G_A}(\tilde{a}_4,c_4)\right),
\end{eqnarray*}
We compute $Q_{G_A}(\tilde{b}_i,c_j)$, $Q_{G_A}(\tilde{a}_i,e_j)$,  $Q_{G_B}(e_j,\tilde{a}_i)$ and $Q_{G_A}(\tilde{a}_i, c_j)$:
   \begin{align*}
  & %\hspace{-1,9cm}
  Q_{G_A}(\tilde{b}_i,c_j)  = \int\limits_{\alpha_1}^{\beta_1} \tilde{b}_i(s)c_jd\mu_s\\
  & %\hspace{-0,2cm}
  =\frac{\lambda_t\lambda_s}{2\pi(m_i \lambda_s+k_j\lambda_t)}\sin\pi\left(\frac{m_i}{\lambda_t}+\frac{k_j}{\lambda_s}\right)(\beta_1+\alpha_1)\sin \pi\left(\frac{m_i}{\lambda_t}+\frac{k_j}{\lambda_s}\right)(\beta_1-\alpha_1)
\\
      & %\hspace{-2mm}
     +\frac{\lambda_t\lambda_s}{2\pi(m_i \lambda_s-k_j\lambda_t)}\sin\pi\left(\frac{m_i}{\lambda_t}-\frac{k_j}{\lambda_s}\right)(\beta_1+\alpha_1)\sin \pi\left(\frac{m_i}{\lambda_t}-\frac{k_j}{\lambda_s}\right)(\beta_1-\alpha_1),\\
      &   \mbox{ if } \frac{m_i}{\lambda_t}\not=\frac{k_j}{\lambda_s},\, \frac{m_i}{\lambda_t}\not=-\frac{k_j}{\lambda_s},\, i=1,3, j=1,2,
    \\
   & %\hspace{-1,9cm}
   Q_{G_A}(\tilde{b}_i,c_j)  = \frac{\lambda_t}{2\pi m_i}
      \left(\left(\sin\left(\frac{2\pi m_i}{\lambda_t}\beta_1\right)\right)^2-\left(\sin\left(\frac{2\pi m_i}{\lambda_t}\alpha_1\right)\right)^2\right),\\
      &  \ \mbox{ if }\ \frac{m_i}{\lambda_t}=\frac{k_j}{\lambda_s},\ i=1,3, j=1,2,
 \\
    & %\hspace{-1,9cm}
   \resizebox{0.99\hsize}{!}{$\displaystyle Q_{G_A}(\tilde{b}_i,c_j)= \frac{\lambda_t\lambda_s}{2\pi(m_i\lambda_s-k_j\lambda_t)}\sin\pi\left(\frac{ m_i}{\lambda_t}-\frac{k_j}{\lambda_s}\right)(\beta_1-\alpha_1)\cos\pi\left(\frac{ m_i}{\lambda_t}-\frac{k_j}{\lambda_s}\right)(\beta_1+\alpha_1)$}\\
   & %\hspace{-0,3cm}
   -\frac{\lambda_t\lambda_s}{2\pi(m_i\lambda_s+k_j\lambda_t)}\sin\pi\left(\frac{ m_i}{\lambda_t}+\frac{k_j}{\lambda_s}\right)(\beta_1-\alpha_1)\cos\pi\left(\frac{ m_i}{\lambda_t}+\frac{k_j}{\lambda_s}\right)(\beta_1+\alpha_1),
   \\
     &  \ \mbox{ if }\ \frac{m_i}{\lambda_t}\not=\frac{k_j}{\lambda_s},\, \frac{m_i}{\lambda_t}\not=-\frac{k_j}{\lambda_s},\, i=1,3,\, j=3,4,\,\\
  &%\hspace{-1,9cm}
  Q_{G_A}(\tilde{b}_i,c_j)= \frac{1}{2}(\beta_1-\alpha_1)-\frac{\lambda_t}{4\pi m_i}\sin2\pi\frac{m_i}{\lambda_t}(\beta_1-\alpha_1)\cos2\pi\frac{m_i}{\lambda_t}(\beta_1+\alpha_1), \\
    & \frac{m_i}{\lambda_t}=\frac{k_j}{\lambda_s},\ i=1,3, j=3,4.\\
  & %\hspace{-2,2cm}
   \resizebox{0.99\hsize}{!}{$\displaystyle
   Q_{G_A}(\tilde{b}_i,c_j)= \frac{\lambda_t\lambda_s}{2\pi(m_i\lambda_s+k_j\lambda_t)}\sin\pi\left(\frac{ m_i}{\lambda_t}+\frac{k_j}{\lambda_s}\right)(\beta_1-\alpha_1)\cos\pi\left(\frac{ m_i}{\lambda_t}+\frac{k_j}{\lambda_s}\right)(\beta_1+\alpha_1)$}\\
   & %\hspace{-0,3cm}
   +\frac{\lambda_t\lambda_s}{2\pi(m_i\lambda_s-k_j\lambda_t)}\sin\pi\left(\frac{ m_i}{\lambda_t}-\frac{k_j}{\lambda_s}\right)(\beta_1-\alpha_1)\cos\pi\left(\frac{ m_i}{\lambda_t}-\frac{k_j}{\lambda_s}\right)(\beta_1+\alpha_1), \\
   &  \ \mbox{ if }\ \frac{m_i}{\lambda_t}\not=\frac{k_j}{\lambda_s},\, \frac{m_i}{\lambda_t}\not=-\frac{k_j}{\lambda_s},\, i=2,4,\, j=1,2,\,\\
  & % \hspace{-1,9cm}
  Q_{G_A}(\tilde{b}_i,c_j)= \frac{1}{2}(\beta_1-\alpha_1)+\frac{\lambda_t}{4\pi m_i}\sin2\pi\frac{m_i}{\lambda_t}(\beta_1-\alpha_1)\cos2\pi\frac{m_i}{\lambda_t}(\beta_1+\alpha_1), \\
   &
    \ \mbox{ if }\ \frac{m_i}{\lambda_t}=\frac{k_j}{\lambda_s},\ i=2,4, j=1,2.
 \\
  & %\hspace{-1,9cm}
  \resizebox{0.99\hsize}{!}{$\displaystyle Q_{G_A}(\tilde{b}_i,c_j)= \frac{\lambda_t\lambda_s}{2\pi(m_i\lambda_s+k_j\lambda_t)}\sin\pi\left(\frac{ m_i}{\lambda_t}+\frac{k_j}{\lambda_s}\right)(\beta_1+\alpha_1)\sin\pi\left(\frac{ m_i}{\lambda_t}+\frac{k_j}{\lambda_s}\right)(\beta_1-\alpha_1)$}\\
    &%\hspace{-0,6cm}-
    \ \cdot\frac{\lambda_t\lambda_s}{2\pi(m_i\lambda_s-k_j\lambda_t)}\sin\pi\left(\frac{ m_i}{\lambda_t}-\frac{k_j}{\lambda_s}\right)(\beta_1+\alpha_1)\sin\pi\left(\frac{ m_i}{\lambda_t}-\frac{k_j}{\lambda_s}\right)(\beta_1-\alpha_1),
    \\
   &
   \ \mbox{ if }\ \frac{m_i}{\lambda_t}\not=\frac{k_j}{\lambda_s},\, \frac{m_i}{\lambda_t}\not=-\frac{k_j}{\lambda_s},\, i=2,4,\, j=3,4,\,
    \\
   & Q_{G_A}(\tilde{b}_i,c_j)  = \frac{\lambda_t}{2\pi m_i}
      \left(\left(\sin\left(\frac{2\pi m_i}{\lambda_t}\beta_1\right)\right)^2-
      \left(\sin\left(\frac{2\pi m_i}{\lambda_t}\alpha_1\right)\right)^2\right),
      \\
   &
   \ \mbox{ if }\   \frac{m_i}{\lambda_t}=\frac{k_j}{\lambda_s},\ i=2,4, j=3,4.
   \end{align*}

   If  $\frac{m_i}{\lambda_t}\not=\frac{k_j}{\lambda_s}$ for all $i,j=1,2,3,4$ except when $i=j=2$, $i=j=3$ and the  numbers $\left(\frac{m_i}{\lambda_t}-\frac{k_j}{\lambda_s}\right)(\beta_1-\alpha_1)$ and
    $\left(\frac{m_i}{\lambda_t}+\frac{k_j}{\lambda_s}\right)(\beta_1-\alpha_1)$,
   $i,j=1,2,3,4$
    are integers, then the system
   \begin{equation*}
   \left\{\sin\left(\frac{2\pi m_i}{\lambda_t}t\right), \cos\left(\frac{2\pi k_j}{\lambda_s}t\right), \cos\left(\frac{2\pi m_i}{\lambda_t}t\right),\sin\left(\frac{2\pi k_j}{\lambda_s}t\right) \right\}_{i,j=1,2,3,4}
   \end{equation*}
   is orthogonal, that is, $Q_{G_A}(\tilde a_i,c_j)=0$ if $\tilde{a}_i\not=c_j$.
   In this case, if $G_A=G_B=[\alpha_1,\beta_1]$,  $\frac{m_2}{\lambda_t}=\frac{k_2}{\lambda_s}$ and $\frac{m_3}{\lambda_t}=\frac{k_3}{\lambda_s}$ then relation \eqref{CondCRWaveletExample4TermsFourierSerie}
   reduces to
      \begin{gather}\label{SystemEqsWaveletExample4TermsFourierSerieArgDiff}
       \left\{\begin{array}{cc}
        Q_{G_A}(\tilde{b}_2,c_2)\theta_{A,2}\theta_{B,2} &= \delta\theta_{B,2}\theta_{A,2}^2Q_{G_A}(\tilde{b}_2,c_2)Q_{G_B}(\tilde{a}_2,e_2) \\
        Q_{G_A}(\tilde{b}_3,c_3)\theta_{A,3}\theta_{B,3} &= \delta\theta_{B,3}\theta_{A,3}^2Q_{G_A}(\tilde{b}_3,c_3)Q_{G_B}(\tilde{a}_3,e_3).
        \end{array}\right.
      \end{gather}
     Since $G_A=G_B=[\alpha_1,\beta_1]$, $\tilde{a}_i=\tilde{b}_i$, $c_j=e_j$, $i,j=2,3$, we put \\
     $\sigma_1=Q_{G_A}(\sin (\omega_3 s), \sin(\omega_3 s) )$, $\sigma_2=Q_{G_A}(\cos (\omega_2 s), \cos(\omega_2 s) )$, where $\omega_2=\frac{2\pi m_2}{\lambda_t}$, $\omega_3=\frac{2\pi m_3}{\lambda_t}$. Therefore, if $\sigma_1\not=0,$ $\sigma_2\not=0,$ then the system of equations \eqref{SystemEqsWaveletExample4TermsFourierSerieArgDiff} has the following solutions:
     \begin{enumerate}
       \item\label{SolSystemEqsWaveletExample4TermsFourierSerieArgDiff:item1} $\left\{ \theta_{A,2}=0, \theta_{A,3}=0,\, \theta_{A,1}, \theta_{A,4}, \theta_{B,1},\theta_{B,2},\theta_{B,3},\theta_{B,4}\in\mathbb{R} \right\}$;
       \item\label{SolSystemEqsWaveletExample4TermsFourierSerieArgDiff:item2} $\left\{ \theta_{A,2}=0, \theta_{B,3}=0, \theta_{A,1},  \theta_{A,3}, \theta_{A,4}, \theta_{B,1},\theta_{B,2},\theta_{B,4}\in\mathbb{R} \right\}$;
       \item\label{SolSystemEqsWaveletExample4TermsFourierSerieArgDiff:item3} $\left\{ \theta_{A,2}=0, \theta_{A,3}=\frac{1}{\delta \sigma_1}, \delta\not=0, \theta_{A,1}, \theta_{A,4}, \theta_{B,1},\theta_{B,2},\theta_{B,3},\theta_{B,4}\in\mathbb{R} \right\}$;
        \item\label{SolSystemEqsWaveletExample4TermsFourierSerieArgDiff:item4} $\left\{ \theta_{B,2}=0, \theta_{A,3}=0, \theta_{A,1}, \theta_{A,2}, \theta_{A,4}, \theta_{B,1},\theta_{B,3},\theta_{B,4}\in\mathbb{R} \right\}$;
        \item\label{SolSystemEqsWaveletExample4TermsFourierSerieArgDiff:item5} $\left\{ \theta_{B,2}=0, \theta_{B,3}=0,\, \theta_{A,1}, \theta_{A,2}, \theta_{A,3}, \theta_{A,4}, \theta_{B,1},\theta_{B,4}\in\mathbb{R} \right\}$;
        \item\label{SolSystemEqsWaveletExample4TermsFourierSerieArgDiff:item6} $\left\{ \theta_{B,2}=0, \theta_{A,3}=\frac{1}{\delta \sigma_1},\delta\not=0,  \theta_{A,1}, \theta_{A,2}, \theta_{A,4}, \theta_{B,1},\theta_{B,3},\theta_{B,4}\in\mathbb{R} \right\}$;
         \item\label{SolSystemEqsWaveletExample4TermsFourierSerieArgDiff:item7} $\left\{ \theta_{A,2}=\frac{1}{\delta \sigma_2}, \theta_{A,3}=0, \delta\not=0,\, \theta_{A,1}, \theta_{A,4}, \theta_{B,1}, \theta_{B,2}, \theta_{B,3},\theta_{B,4}\in\mathbb{R} \right\}$;
          \item\label{SolSystemEqsWaveletExample4TermsFourierSerieArgDiff:item8} $\left\{\theta_{A,2}=\frac{1}{\delta \sigma_2}, \theta_{B,3}=0, \delta\not=0, \theta_{A,1}, \theta_{A,3}, \theta_{A,4}, \theta_{B,1},\theta_{B,2},\theta_{B,4}\in\mathbb{R} \right\}$;
          \item\label{SolSystemEqsWaveletExample4TermsFourierSerieArgDiff:item9} $\left\{ \theta_{A,2}=\frac{1}{\delta \sigma_2}, \theta_{A,3}=\frac{1}{\delta \sigma_1}, \delta\not=0, \theta_{A,1},  \theta_{A,4}, \theta_{B,1}, \theta_{B,2},\theta_{B,3},\theta_{B,4}\in\mathbb{R} \right\}$.
     \end{enumerate}

     The corresponding pairs of operators $(A,B)$ are \\
     \noindent\ref{SolSystemEqsWaveletExample4TermsFourierSerieArgDiff:item1}.
     For $\left\{ \theta_{A,2}=0, \theta_{A,3}=0, \theta_{A,1}, \theta_{A,4}, \theta_{B,1},\theta_{B,2},\theta_{B,3},\theta_{B,4}\in\mathbb{R},  \sigma_1\not=0,\sigma_2\not=0 \right\}$,
           \begin{align*}
           (Ax)(t)  =& \int\limits_{\alpha_1}^{\beta_1} I_{[\alpha,\beta]}(t)\left(\theta_{A,1} \sin(\frac{2\pi m_1}{\lambda_t} t)\cos(\frac{2\pi k_1}{\lambda_s} s)\right. \\
             &
            \left. +
            \theta_{A,4} \cos(\frac{2\pi m_4}{\lambda_t} t)\sin(\frac{2\pi k_4}{\lambda_s} s)\right) x(s)d\mu_s,\\
            (Bx)(t)=&\int\limits_{\alpha_1}^{\beta_1}I_{[\alpha,\beta]}(t) \left(\theta_{B,1} \sin(\frac{2\pi m_1}{\lambda_t} t)\cos(\frac{2\pi k_1}{\lambda_s} s)\right.
            \\
             + &
              \theta_{B,2} \cos(\frac{2\pi m_2}{\lambda_t} t)\cos(\frac{2\pi k_2}{\lambda_s} s)
             + \theta_{B,3} \sin(\frac{2\pi m_3}{\lambda_t} t)\sin(\frac{2\pi k_3}{\lambda_s} s)\\
            +& \left.
            \theta_{B,4} \cos(\frac{2\pi m_4}{\lambda_t} t)\sin(\frac{2\pi k_4}{\lambda_s} s)\right)x(s)d\mu_s,
           \end{align*}
          for almost every $t$. These operators satisfy $AB=\delta BA^2=BA=0$.

        \noindent\ref{SolSystemEqsWaveletExample4TermsFourierSerieArgDiff:item2}. For $\left\{ \theta_{A,2}=0, \theta_{B,3}=0, \theta_{A,1},  \theta_{A,3}, \theta_{A,4}, \theta_{B,1},\theta_{B,2},\theta_{B,4}\in\mathbb{R},  \sigma_1\not=0,\sigma_2\not=0 \right\}$,
       \begin{align*}
          & (Ax)(t)  = \int\limits_{\alpha_1}^{\beta_1} I_{[\alpha,\beta]}(t)\left(\theta_{A,1} \sin(\frac{2\pi m_1}{\lambda_t} t)\cos(\frac{2\pi k_1}{\lambda_s} s)\right. \\
         &   +  %\hspace{-0,2cm}
             \theta_{A,3} \sin(\frac{2\pi m_3}{\lambda_t} t)\sin(\frac{2\pi k_3}{\lambda_s} s)
             \left. +
            \theta_{A,4} \cos(\frac{2\pi m_4}{\lambda_t} t)\sin(\frac{2\pi k_4}{\lambda_s} s)\right) x(s)d\mu_s,\\
         &   (Bx)(t)=\int\limits_{\alpha_1}^{\beta_1}I_{[\alpha,\beta]}(t) \left(\theta_{B,1} \sin(\frac{2\pi m_1}{\lambda_t} t)\cos(\frac{2\pi k_1}{\lambda_s} s)\right. \\
         &  +   %\hspace{-0,2cm}
              \theta_{B,2} \cos(\frac{2\pi m_2}{\lambda_t} t)\cos(\frac{2\pi k_2}{\lambda_s} s)
            \left.
     + \theta_{B,4} \cos(\frac{2\pi m_4}{\lambda_t} t)\sin(\frac{2\pi k_4}{\lambda_s} s)\right)x(s)d\mu_s,
           \end{align*}
         for  almost every $t$. These operators satisfy $AB=\delta BA^2=BA=0$.

       \noindent\ref{SolSystemEqsWaveletExample4TermsFourierSerieArgDiff:item3}.  For $\theta_{A,2}=0,\, \theta_{A,3}=\frac{1}{\delta \sigma_1},\, \delta\not=0, \sigma_1\not=0,\sigma_2\not=0, \theta_{A,1}, \theta_{A,4},\, \theta_{B,1},\theta_{B,2},\theta_{B,3},\theta_{B,4}\in\mathbb{R} $,
           \begin{align*}
          & (Ax)(t)  = \int\limits_{\alpha_1}^{\beta_1} I_{[\alpha,\beta]}(t)\left(\theta_{A,1} \sin(\frac{2\pi m_1}{\lambda_t} t)\cos(\frac{2\pi k_1}{\lambda_s} s)\right. \\
          & \ +  %\hspace{-0,2cm}
             \frac{1}{\delta \sigma_1} \sin(\frac{2\pi m_3}{\lambda_t} t)\sin(\frac{2\pi k_3}{\lambda_s} s)
             \left. +
            \theta_{A,4} \cos(\frac{2\pi m_4}{\lambda_t} t)\sin(\frac{2\pi k_4}{\lambda_s} s)\right) x(s)d\mu_s,
\\
          &  (Bx)(t)=\int\limits_{\alpha_1}^{\beta_1}I_{[\alpha,\beta]}(t) \left(\theta_{B,1} \sin(\frac{2\pi m_1}{\lambda_t} t)\cos(\frac{2\pi k_1}{\lambda_s} s)\right. \\
          & \ +  % \hspace{-0,2cm}
              \theta_{B,2} \cos(\frac{2\pi m_2}{\lambda_t} t)\cos(\frac{2\pi k_2}{\lambda_s} s)
              + \theta_{B,3} \sin(\frac{2\pi m_3}{\lambda_t} t)\sin(\frac{2\pi k_3}{\lambda_s} s)\\
          &  % \hspace{-0,2cm}
        \ + \left.
      \theta_{B,4} \cos(\frac{2\pi m_4}{\lambda_t} t)\sin(\frac{2\pi k_4}{\lambda_s} s)\right)x(s)d\mu_s,
           \end{align*}
         for  almost every $t$. These operators satisfy $AB=\delta BA^2=BA=0$.

         \noindent\ref{SolSystemEqsWaveletExample4TermsFourierSerieArgDiff:item4}. For $\left\{ \theta_{B,2}=0,\, \theta_{A,3}=0,\, \theta_{A,1}, \theta_{A,2}, \theta_{A,4},\,  \sigma_1\not=0,\sigma_2\not=0, \theta_{B,1},\theta_{B,3},\theta_{B,4}\in\mathbb{R} \right\}$,
           \begin{align*}
           &(Ax)(t)  = \int\limits_{\alpha_1}^{\beta_1} I_{[\alpha,\beta]}(t)\left(\theta_{A,1} \sin(\frac{2\pi m_1}{\lambda_t} t)\cos(\frac{2\pi k_1}{\lambda_s} s)\right. \\
            % \hspace{-0,2cm}
           & +  \theta_{A,2} \cos(\frac{2\pi m_2}{\lambda_t} t)\cos(\frac{2\pi k_2}{\lambda_s} s) \left. +
            \theta_{A,4} \cos(\frac{2\pi m_4}{\lambda_t} t)\sin(\frac{2\pi k_4}{\lambda_s} s)\right) x(s)d\mu_s,\\
           & (Bx)(t)=\int\limits_{\alpha_1}^{\beta_1}I_{[\alpha,\beta]}(t) \left(\theta_{B,1} \sin(\frac{2\pi m_1}{\lambda_t} t)\cos(\frac{2\pi k_1}{\lambda_s} s)\right. \\
             %\hspace{-0,2cm}
          & +    \theta_{B,3} \sin(\frac{2\pi m_3}{\lambda_t} t)\sin(\frac{2\pi k_3}{\lambda_s} s) \left.
     + \theta_{B,4} \cos(\frac{2\pi m_4}{\lambda_t} t)\sin(\frac{2\pi k_4}{\lambda_s} s)\right)x(s)d\mu_s,
           \end{align*}
          for almost every $t$. These operators satisfy $AB=\delta BA^2=BA=0$.

         \noindent\ref{SolSystemEqsWaveletExample4TermsFourierSerieArgDiff:item5}. For $\left\{ \theta_{B,2}=0,\, \theta_{B,3}=0,\, \theta_{A,1}, \theta_{A,2}, \theta_{A,3}, \theta_{A,4},\,  \sigma_1\not=0,\sigma_2\not=0, \theta_{B,1},\theta_{B,4}\in\mathbb{R} \right\}$,
             \begin{align*}
          & (Ax)(t)  = \int\limits_{\alpha_1}^{\beta_1} I_{[\alpha,\beta]}(t)\left(\theta_{A,1} \sin(\frac{2\pi m_1}{\lambda_t} t)\cos(\frac{2\pi k_1}{\lambda_s} s)\right. \\
            & +
             \theta_{A,2} \cos(\frac{2\pi m_2}{\lambda_t} t)\cos(\frac{2\pi k_2}{\lambda_s} s)
             +\theta_{A,3} \sin(\frac{2\pi m_3}{\lambda_t} t)\sin(\frac{2\pi k_3}{\lambda_s} s)\\
            & + \left.
            \theta_{A,4} \cos(\frac{2\pi m_4}{\lambda_t} t)\sin(\frac{2\pi k_4}{\lambda_s} s)\right) x(s)d\mu_s,\\
          &  (Bx)(t)=\int\limits_{\alpha_1}^{\beta_1}I_{[\alpha,\beta]}(t) \left(\theta_{B,1} \sin(\frac{2\pi m_1}{\lambda_t} t)\cos(\frac{2\pi k_1}{\lambda_s} s)\right. \\
               & + \left.
      \theta_{B,4} \cos(\frac{2\pi m_4}{\lambda_t} t)\sin(\frac{2\pi k_4}{\lambda_s} s)\right)x(s)d\mu_s,
           \end{align*}
         for  almost every $t$. These operators satisfy $AB=\delta BA^2=BA=0$.

       \noindent\ref{SolSystemEqsWaveletExample4TermsFourierSerieArgDiff:item6}. For  $ \theta_{B,2}=0,\, \theta_{A,3}=\frac{1}{\delta \sigma_1},\,\delta\not=0,  \sigma_1\not=0,\sigma_2\not=0,  \theta_{A,1}, \theta_{A,2}, \theta_{A,4},\, \theta_{B,1},\theta_{B,3},\theta_{B,4}\in\mathbb{R} $,
            \begin{align*}
          & (Ax)(t)  =\int\limits_{\alpha_1}^{\beta_1} I_{[\alpha,\beta]}(t)\left(\theta_{A,1} \sin(\frac{2\pi m_1}{\lambda_t} t)\cos(\frac{2\pi k_1}{\lambda_s} s)\right. \\
           %  \hspace{-0,2cm}
         & +    \theta_{A,2} \cos(\frac{2\pi m_2}{\lambda_t} t)\cos(\frac{2\pi k_2}{\lambda_s} s)
             +\frac{1}{\delta \sigma_1} \sin(\frac{2\pi m_3}{\lambda_t} t)\sin(\frac{2\pi k_3}{\lambda_s} s)\\
            % \hspace{-0,2cm}
          & +   \left.
            \theta_{A,4} \cos(\frac{2\pi m_4}{\lambda_t} t)\sin(\frac{2\pi k_4}{\lambda_s} s)\right) x(s)d\mu_s,
            \\
           & (Bx)(t)=\int\limits_{\alpha_1}^{\beta_1}I_{[\alpha,\beta]}(t) \left(\theta_{B,1} \sin(\frac{2\pi m_1}{\lambda_t} t)\cos(\frac{2\pi k_1}{\lambda_s} s)\right. \\
              %  \hspace{-0,2cm}
            & \ +    \theta_{B,3} \sin(\frac{2\pi m_3}{\lambda_t} t)\sin(\frac{2\pi k_3}{\lambda_s} s) \left.
     + \theta_{B,4} \cos(\frac{2\pi m_4}{\lambda_t} t)\sin(\frac{2\pi k_4}{\lambda_s} s)\right)x(s)d\mu_s,
           \end{align*}
           for almost every $t$. These operators satisfy $AB=\delta BA^2=BA$.

         \noindent\ref{SolSystemEqsWaveletExample4TermsFourierSerieArgDiff:item7}. If $ \theta_{A,2}=\frac{1}{\delta \sigma_2},\, \theta_{A,3}=0,\, \delta\not=0,  \sigma_1\not=0,\sigma_2\not=0, \theta_{A,1}, \theta_{A,4},\, \theta_{B,1}, \theta_{B,2}, \theta_{B,3},\theta_{B,4}\in\mathbb{R} $,
              \begin{align*}
             &  (Ax)(t)=I_{[\alpha,\beta]}(t)\int\limits_{\alpha_1}^{\beta_1} \left(\theta_{A,1}\sin(\frac{2\pi m_1}{\lambda_t} t)\cos(\frac{2\pi k_1}{\lambda_s} s) \right.
               \\
              %  \hspace{-0,2cm}
             & +  \frac{1}{\delta \sigma_2} \cos(\frac{2\pi m_2}{\lambda_t} t)\cos(\frac{2\pi k_2}{\lambda_s} s)+
        \left.\theta_{A,4}\cos(\frac{2\pi m_4}{\lambda_t} t)\sin(\frac{2\pi k_4}{\lambda_s} s)\right)x(s)d\mu_s,
              \\
           & (Bx)(t)=I_{[\alpha,\beta]}(t)\int\limits_{\alpha_1}^{\beta_1} \left(\theta_{B,1} \sin(\frac{2\pi m_1}{\lambda_t} t)\cos(\frac{2\pi k_1}{\lambda_s} s) \right.
            \\
            % \hspace{-0,2cm}
           & + \theta_{B,2} \cos(\frac{2\pi m_2}{\lambda_t} t)\cos(\frac{2\pi k_2}{\lambda_s} s) +\theta_{B,3} \sin(\frac{2\pi m_3}{\lambda_t} t)\sin(\frac{2\pi k_3}{\lambda_s} s)\\
       %\hspace{-0,2cm}
     &+\left. \theta_{B,4} \cos(\frac{2\pi m_4}{\lambda_t} t)\sin(\frac{2\pi k_4}{\lambda_s} s)\right)x(s)d\mu_s,
              \end{align*}
            for almost every $t$. These operators satisfy $AB=\delta BA^2=BA$.

         \noindent\ref{SolSystemEqsWaveletExample4TermsFourierSerieArgDiff:item8}. For $\theta_{A,2}=\frac{1}{\delta \sigma_2},\, \theta_{B,3}=0,\, \delta\not=0, \sigma_1\not=0,\sigma_2\not=0, \theta_{A,1}, \theta_{A,3}, \theta_{A,4},\, \theta_{B,1},\theta_{B,2},\theta_{B,4}\in\mathbb{R} $,
             \begin{align*}
        & (Ax)(t)=I_{[\alpha,\beta]}(t)\int\limits_{\alpha_1}^{\beta_1} \left(\theta_{A,1}\sin(\frac{2\pi m_1}{\lambda_t} t)\cos(\frac{2\pi k_1}{\lambda_s} s) \right.
               \\
              % \hspace{-0,2cm}
             & +  \frac{1}{\delta \sigma_2} \cos(\frac{2\pi m_2}{\lambda_t} t)\cos(\frac{2\pi k_2}{\lambda_s} s)
               +\theta_{A,3} \sin(\frac{2\pi m_3}{\lambda_t} t)\sin(\frac{2\pi k_3}{\lambda_s} s)\\
              % \hspace{-0,2cm}
        \\
        &+ \left.\theta_{A,4}\cos(\frac{2\pi m_4}{\lambda_t} t)\sin(\frac{2\pi k_4}{\lambda_s} s)\right)x(s)d\mu_s,
              \\
          &  (Bx)(t)=I_{[\alpha,\beta]}(t)\int\limits_{\alpha_1}^{\beta_1} \left(\theta_{B,1} \sin(\frac{2\pi m_1}{\lambda_t} t)\cos(\frac{2\pi k_1}{\lambda_s} s) \right.
            \\
            % \hspace{-0,2cm}
          & +  \theta_{B,2} \cos(\frac{2\pi m_2}{\lambda_t} t)\cos(\frac{2\pi k_2}{\lambda_s} s)
           \left.
        + \theta_{B,4} \cos(\frac{2\pi m_4}{\lambda_t} t)\sin(\frac{2\pi k_4}{\lambda_s} s)\right)x(s)d\mu_s,
              \end{align*}
           for  almost every $t$. These operators satisfy $AB=\delta BA^2=BA$.

         \noindent\ref{SolSystemEqsWaveletExample4TermsFourierSerieArgDiff:item9}. For $\theta_{A,2}=\frac{1}{\delta \sigma_2}, \theta_{A,3}=\frac{1}{\delta \sigma_1},\delta\not=0, \sigma_1\not=0,\sigma_2\not=0, \theta_{A,1},  \theta_{A,4}, \theta_{B,1}, \theta_{B,2},\theta_{B,3}, $ \\
$\theta_{B,4}\in\mathbb{R} $,
             \begin{align*}
           (Ax)(t)=\,&I_{[\alpha,\beta]}(t)\int\limits_{\alpha_1}^{\beta_1} \left(\theta_{A,1}\sin(\frac{2\pi m_1}{\lambda_t} t)\cos(\frac{2\pi k_1}{\lambda_s} s) \right.
              \\
             + &  %\hspace{-0,2cm}
               \frac{1}{\delta \sigma_2} \cos(\frac{2\pi m_2}{\lambda_t} t)\cos(\frac{2\pi k_2}{\lambda_s} s)
               +\frac{1}{\delta \sigma_1} \sin(\frac{2\pi m_3}{\lambda_t} t)\sin(\frac{2\pi k_3}{\lambda_s} s)\\
             + &  %\hspace{-0,2cm}
        \left.\theta_{A,4}\cos(\frac{2\pi m_4}{\lambda_t} t)\sin(\frac{2\pi k_4}{\lambda_s} s)\right)x(s)d\mu_s,
              \\
            (Bx)(t)=\,&I_{[\alpha,\beta]}(t)\int\limits_{\alpha_1}^{\beta_1} \left(\theta_{B,1} \sin(\frac{2\pi m_1}{\lambda_t} t)\cos(\frac{2\pi k_1}{\lambda_s} s) \right.
            \\
           + &  %\hspace{-0,2cm}
            \theta_{B,2} \cos(\frac{2\pi m_2}{\lambda_t} t)\cos(\frac{2\pi k_2}{\lambda_s} s)
             + \theta_{B,3} \sin(\frac{2\pi m_3}{\lambda_t} t)\sin(\frac{2\pi k_3}{\lambda_s} s)\\
           + &  %\hspace{-0,2cm}
           \left.
         \theta_{B,4} \cos(\frac{2\pi m_4}{\lambda_t} t)\sin(\frac{2\pi k_4}{\lambda_s} s)\right)x(s)d\mu_s,
              \end{align*}
             for almost every $t$. These operators satisfy $AB=\delta BA^2=BA$.

     \subsection*{Case 2:}
     In Case 2, we consider that the coefficient of $t$ and $s$ in the argument of functions $\sin (\cdot)$, $\cos(\cdot)$ is the same ($\omega $), and functions $\sin(\cdot)$ and $\cos(\cdot)$ are orthogonal in $[\alpha_1,\beta_1]$, that is, $\int\limits_{\alpha_1}^{\beta_1} \sin(\omega s)\cos(\omega s)d\mu_s=0$. Among obtained  operators, some satisfy commutation relation $AB=\delta BA^2$ but do not commute, whereas, for those  operators which commute, the commutativity conditions are obtained by direct computation or by Proposition \ref{PropositionCommutatorZeroFourSequenceCond}.
     If $\frac{m_i}{\lambda_t}=\frac{k_j}{\lambda_s}$  for all $i,j=1,2,3,4$ then we set $\omega=\frac{2\pi m_i}{\lambda_t}=\frac{ 2\pi k_j}{\lambda_s}$ for all $i,j=1,2,3,4$.
      The system $   \left\{\sin\left(\omega s\right), \cos\left(\omega s\right) \right\} $
   is orthogonal in $G_A=[\alpha_1,\beta_1]$, that is, $Q_{G_A}(\sin(\omega s),\,  \cos(\omega s))=0$ if and only if  either $\frac{\omega }{\pi}(\beta_1-\alpha_1)\in \mathbb{Z}$ or $\frac{\omega }{\pi}(\beta_1+\alpha_1)\in \mathbb{Z}$.
   %or ($\frac{\omega }{\pi}(\beta_1-\alpha_1)\in \mathbb{Z}$ and odd) or
    % ($\frac{\omega }{\pi}(\beta_1+\alpha_1)\in \mathbb{Z}$ and odd).
   Under these conditions and $G_A=G_B=[\alpha_1,\beta_1]$, we have that  $Q_{G_A}(\sin(\omega s),\,  \cos(\omega s))=0$. Hence, we get from  \eqref{CondCRWaveletExample4TermsFourierSerie} the following system of equations:
      \begin{gather}\label{SystEquationSolvingCase2FourierSeqExample}
        \left\{\begin{array}{ll}
                 \theta_{A,3}\theta_{B,1}\sigma_1+\theta_{A,1}\theta_{B,2}\sigma_2 & =\delta \theta_{A,3}\theta_{A,1}\theta_{B,3}\sigma_1^2+\delta \theta_{A,4}\theta_{A,1}\theta_{B,1}\sigma_1\sigma_2 \\
                 &  +\delta \theta_{A,1}\theta_{A,2}\theta_{B,3}\sigma_1\sigma_2+\delta \theta_{A,2}^2\theta_{B,1}\sigma_2^2 \\ %\medskip
                 \theta_{A,4}\theta_{B,1}\sigma_1+\theta_{A,2}\theta_{B,2}\sigma_2 & =\delta \theta_{A,3}\theta_{A,1}\theta_{B,4}\sigma_1^2+\delta \theta_{A,4}\theta_{A,1}\theta_{B,2}\sigma_1\sigma_2 \\
                 &  +\delta \theta_{A,1}\theta_{A,2}\theta_{B,4}\sigma_1\sigma_2+\delta \theta_{A,2}^2\theta_{B,2}\sigma_2^2 \\ %\medskip
                 \theta_{A,3}\theta_{B,3}\sigma_1+\theta_{A,1}\theta_{B,4}\sigma_2 & =\delta \theta_{A,3}^2\theta_{B,3}\sigma_1^2+\delta \theta_{A,4}\theta_{A,3}\theta_{B,1}\sigma_1\sigma_2 \\
                 &  +\delta \theta_{A,4}\theta_{A,1}\theta_{B,3}\sigma_1\sigma_2+\delta \theta_{A,2}\theta_{A,4}\theta_{B,1}\sigma_2^2 \\ %\medskip
                   \theta_{A,4}\theta_{B,3}\sigma_1+\theta_{A,2}\theta_{B,4}\sigma_2 & =\delta \theta_{A,3}^2\theta_{B,4}\sigma_1^2+\delta \theta_{A,4}\theta_{A,3}\theta_{B,2}\sigma_1\sigma_2 \\
                  &  +\delta \theta_{A,4}\theta_{A,1}\theta_{B,4}\sigma_1\sigma_2+\delta \theta_{A,2}\theta_{A,4}\theta_{B,2}\sigma_2^2
                  \end{array}\right.
        \end{gather}
       where  $\sigma_1=Q_{G_A}(\sin(\omega s),\sin(\omega s))$ and $\sigma_2=Q_{G_A}(\cos(\omega s),\cos(\omega s))$, that is,
       \begin{align}\label{ConstantSigma1Case2Omega}
       & \sigma_1=\int\limits_{\alpha_1}^{\beta_1} (\sin(\omega s))^2 d\mu_s=\left\{\begin{array}{cc}
                         0, & \mbox{ if } \omega=0 \\
                         \frac{\beta_1-\alpha_1}{2}-\frac{\cos(\omega(\alpha_1+\beta_1))\sin(\omega(\beta_1-\alpha_1))}{2\omega}, & \mbox{ if } \omega\not=0
                       \end{array} ,\right.\\ \label{ConstantSigma2Case2Omega}
    &  \sigma_2=\int\limits_{\alpha_1}^{\beta_1} (\cos(\omega s))^2 d\mu_s={\beta_1-\alpha_1}-\sigma_1 \nonumber \\
    & \hspace{3cm} =\left\{\begin{array}{cc}
                         \beta_1-\alpha_1, & \mbox{ if } \omega=0 \\
                         \frac{\beta_1-\alpha_1}{2}+\frac{\cos(\omega(\alpha_1+\beta_1))\sin(\omega(\beta_1-\alpha_1))}{2\omega}, & \mbox{ if } \omega\not=0
                       \end{array} \right. .
        \end{align}
        The  system of equations \eqref{SystEquationSolvingCase2FourierSeqExample} is linear in $\theta_{B,i}$, $i=1,2,3,4$ and can be written as follows:
       \begin{align*}&
        \left\{\begin{array}{ll}
               0 & =  (\theta_{A,3}\sigma_1-\delta\theta_{A,4}\theta_{A,1}\sigma_1\sigma_2-\delta \theta_{A,2}^2\sigma_2^2)\theta_{B,1}+\theta_{A,1}\sigma_2\theta_{B,2} \\
                &\ -(\delta \theta_{A,3}\theta_{A,1}\sigma_1^2 +\delta \theta_{A,1}\theta_{A,2}\sigma_1\sigma_2)\theta_{B,3}
                 \\ %\medskip
              0& =   \theta_{A,4}\sigma_1\theta_{B,1}+(\theta_{A,2}\sigma_2-\delta \theta_{A,4}\theta_{A,1}\sigma_1\sigma_2-\delta \theta_{A,2}^2\sigma_2^2)\theta_{B,2} \\
              & \ -(\delta \theta_{A,3}\theta_{A,1}\sigma_1^2+\delta \theta_{A,1}\theta_{A,2}\sigma_1\sigma_2)\theta_{B,4}
               \\ %\medskip
              0&= (-\delta \theta_{A,2}\theta_{A,4}\sigma_2^2-\delta \theta_{A,4}\theta_{A,3}\sigma_1\sigma_2)\theta_{B,1}\\
               & \ +  (\theta_{A,3}\sigma_1-\delta \theta_{A,3}^2\sigma_1^2-\delta \theta_{A,4}\theta_{A,1}\sigma_1\sigma_2)\theta_{B,3}
              +\theta_{A,1}\sigma_2\theta_{B,4}
              \\ %\medskip
               0 &=  (-\delta \theta_{A,4}\theta_{A,3}\sigma_1\sigma_2-\delta \theta_{A,2}\theta_{A,4}\sigma_2^2)\theta_{B,2} + \theta_{A,4}\sigma_1\theta_{B,3}\\
               & \ +(\theta_{A,2}\sigma_2-\delta \theta_{A,3}^2\sigma_1^2-\delta \theta_{A,4}\theta_{A,1}\sigma_1\sigma_2)\theta_{B,4}.
                 \end{array}\right.
       \end{align*}

       \subsection*{Case 2A:} We consider that operator $A$ is given and we look for operator $B$. Considering that $\theta_{A,i}$, $i=1,2,3,4$ are given, the  matrix equation corresponding  to \eqref{SystEquationSolvingCase2FourierSeqExample} is $V\Theta_B=0$, where $V=[v_{ij}]_{4\times 4}$,
       \begin{align*}
        & \begin{array}{llll}
                 v_{11}&= \theta_{A,3}\sigma_1-\delta\theta_{A,4}\theta_{A,1}\sigma_1\sigma_2-\delta \theta_{A,2}^2\sigma_2^2, & v_{12}&= \theta_{A,1}\sigma_2 \\
                 v_{13}&= -(\delta \theta_{A,3}\theta_{A,1}\sigma_1^2 +\delta \theta_{A,1}\theta_{A,2}\sigma_1\sigma_2), & v_{14}&= 0 \\
                 v_{21}&= \theta_{A,4}\sigma_1, & v_{22}&=\theta_{A,2}\sigma_2-\delta \theta_{A,4}\theta_{A,1}\sigma_1\sigma_2-\delta \theta_{A,2}^2\sigma_2^2 \\
                 v_{23}&=  0, & v_{24}&=-(\delta \theta_{A,3}\theta_{A,1}\sigma_1^2+\delta \theta_{A,1}\theta_{A,2}\sigma_1\sigma_2) \\
                 v_{31}&=  -\delta \theta_{A,2}\theta_{A,4}\sigma_2^2-\delta \theta_{A,4}\theta_{A,3}\sigma_1\sigma_2, & v_{32}&= 0 \\
                 v_{33}&=\theta_{A,3}\sigma_1-\delta \theta_{A,3}^2\sigma_1^2-\delta \theta_{A,4}\theta_{A,1}\sigma_1\sigma_2, & v_{34}&=\theta_{A,1}\sigma_2 \\
                 v_{41}&= 0, & v_{42}&=-\delta \theta_{A,4}\theta_{A,3}\sigma_1\sigma_2-\delta \theta_{A,2}\theta_{A,4}\sigma_2^2 \\
                 v_{43}&= \theta_{A,4}\sigma_1, & v_{44}&= \theta_{A,2}\sigma_2-\delta \theta_{A,3}^2\sigma_1^2-\delta \theta_{A,4}\theta_{A,1}\sigma_1\sigma_2
         \end{array}
       \end{align*}
       and $\Theta_B=[\theta_{B,1j}]_{1\times4}$, $\theta_{B,1j}=\theta_{B,j}$, $j=1,2,3,4$.
       The corresponding determinant is
        \begin{align*}
       % \hspace{-1cm}
       &\mbox{det}(V)=\sigma_1\sigma_2\left(\theta_{A,1}^4\theta_{A,4}^4\delta^4\sigma_1^3\sigma_2^3+
          \theta_{A,1}^3\theta_{A,2}^2\theta_{A,4}^3\delta^4\sigma_1^2\sigma_2^4\right.\\
         &- 2\theta_{A,1}^3\theta_{A,2}\theta_{A,3}\theta_{A,4}^3\delta^4\sigma_1^3\sigma_2^3-
          \theta_{A,1}^3\theta_{A,2}\theta_{A,4}^4\delta^4\sigma_1^2\sigma_2^4
           \\
         &\hspace{6cm} +\theta_{A,1}^3\theta_{A,2}^2\theta_{A,4}^3\delta^4\sigma_1^4\sigma_2^2\\
         & +\theta_{A,1}^3\theta_{A,3}\theta_{A,4}^4\delta^4\sigma_1^3\sigma_2^3
           - 2\theta_{A,1}^2\theta_{A,2}^3\theta_{A,3}\theta_{A,4}^2\delta^4\sigma_1^2\sigma_2^4+
          2\theta_{A,1}^2\theta_{A,2}^2\theta_{A,3}^2\theta_{A,4}^2\delta^4\sigma_1^3\sigma_2^3
         \\
         &  +2\theta_{A,1}^2\theta_{A,2}^2\theta_{A,3}\theta_{A,4}^3\delta^4\sigma_1^2\sigma_2^4-
          2\theta_{A,1}^2\theta_{A,2}\theta_{A,3}^3\theta_{A,4}^2\delta^4\sigma_1^4\sigma_2^2
         \\
         &\hspace{6cm}   +\theta_{A,1}^2\theta_{A,2}\theta_{A,3}^2\theta_{A,4}^3\delta^4\sigma_1^3\sigma_2^3
        \\
         & -\theta_{A,1}^2\theta_{A,3}^3\theta_{A,4}^3\delta^4\sigma_1^4\sigma_2^2
          +\theta_{A,1}^2\theta_{A,2}^4\theta_{A,3}^2\theta_{A,4}\delta^4\sigma_1^2\sigma_2^4-
          2\theta_{A,1}\theta_{A,2}^3\theta_{A,3}^3\theta_{A,4}\delta^4\sigma_1^3\sigma_2^3\\
          & -\theta_{A,1}\theta_{A,2}^3\theta_{A,3}^2\theta_{A,4}^2\delta^4\sigma_1^3\sigma_2^3+
          \theta_{A,1}\theta_{A,2}^2\theta_{A,3}^4\theta_{A,4}\delta^4\sigma_1^4\sigma_2^2+
          \theta_{A,1}\theta_{A,2}^2\theta_{A,3}^3\theta_{A,4}^2\delta^4\sigma_1^3\sigma_2^3
         \\
         & +2\theta_{A,1}\theta_{A,2}\theta_{A,3}^4\theta_{A,4}^2\delta^4\sigma_1^4\sigma_2^2+
            \theta_{A,2}^4\theta_{A,3}^4\delta^4\sigma_1^3\sigma_2^3-
          \theta_{A,2}^3\theta_{A,3}^4\theta_{A,4}\delta^4\sigma_1^3\sigma_2^3
         \\
         &  -\theta_{A,2}^2\theta_{A,3}^5\theta_{A,4}\delta^4\sigma_1^4\sigma_2^2-
          \theta_{A,1}^3\theta_{A,2}\theta_{A,4}^3\delta^3\sigma_1^2\sigma_2^3
          -2\theta_{A,1}^3\theta_{A,3}\theta_{A,4}^3\delta^3\sigma_1^3\sigma_2^2
          \\
         & -\theta_{A,1}^3\theta_{A,4}^4\delta^3\sigma_1^2\sigma_2^3
           -\theta_{A,1}^2\theta_{A,2}^2\theta_{A,3}\theta_{A,4}^2\delta^3\sigma_1^2\sigma_2^3-
          \theta_{A,1}^2\theta_{A,2}^2\theta_{A,4}^3\delta^3\sigma_1\sigma_2^4\\
         & -\theta_{A,1}^2\theta_{A,2}\theta_{A,3}^2\theta_{A,4}^2\delta^3\sigma_1^3\sigma_2^2+
          4\theta_{A,1}^2\theta_{A,2}\theta_{A,3}\theta_{A,4}^3\delta^3\sigma_1^2\sigma_2^3
          -3\theta_{A,1}^2\theta_{A,3}^3\theta_{A,4}^2\delta^3\sigma_1^4\sigma_2^1\\
         &+ 2\theta_{A,1}^2\theta_{A,3}^2\theta_{A,4}^3\delta^3\sigma_1^3\sigma_2^2
           -\theta_{A,1}\theta_{A,2}^4\theta_{A,3}\theta_{A,4}^2\delta^3\sigma_1\sigma_2^4-
          \theta_{A,1}\theta_{A,2}^3\theta_{A,3}^2\theta_{A,4}\delta^3\sigma_1^2\sigma_2^3
          \\
         & +3\theta_{A,1}\theta_{A,2}^3\theta_{A,3}\theta_{A,4}^2\delta^3\sigma_1\sigma_2^4
          -2\theta_{A,1}\theta_{A,2}\theta_{A,3}^3\theta_{A,4}\delta^3\sigma_1^3\sigma_2^2
          \\
         &\hspace{6cm} +3\theta_{A,1}\theta_{A,2}\theta_{A,3}^3\theta_{A,4}^2\delta^3\sigma_1^3\sigma_2^2\\
          &-\theta_{A,1}\theta_{A,3}^5\theta_{A,4}\delta^3\sigma_1^5
           +3\theta_{A,1}\theta_{A,3}^4\theta_{A,4}^2\delta^3\sigma_1^4\sigma_2
          -\theta_{A,2}^4\theta_{A,3}^3\delta^3\sigma_1^2\sigma_2^3
          \\
         &\hspace{6cm}  -\theta_{A,1}^4\theta_{A,3}^2\theta_{A,4}\delta^3\sigma_1\sigma_2^4\\
          &-\theta_{A,2}^3\theta_{A,3}^4\delta^3\sigma_1^3\sigma_2^2+
          \theta_{A,2}^3\theta_{A,3}^3\theta_{A,4}\delta^3\sigma_1^2\sigma_2^3-
          \theta_{A,2}^2\theta_{A,3}^5\delta^3\sigma_1^4\sigma_2+
           \theta_{A,2}^2\theta_{A,3}^4\theta_{A,4}\delta^3\sigma_1^3\sigma_2^2\\
          &+\theta_{A,2}\theta_{A,3}^5\theta_{A,4}\delta^3\sigma_1^4\sigma_2
           +\theta_{A,3}^6\theta_{A,4}\delta^3\sigma_1^5-
          2\theta_{A,1}^3\theta_{A,4}^3\delta^2\sigma_1^2\sigma_2^2
          -3\theta_{A,1}^2\theta_{A,2}^2\theta_{A,4}^2\delta^2\sigma_1\sigma_2^3
          \\
         & -2\theta_{A,1}^2\theta_{A,3}^2\theta_{A,4}^2\delta^2\sigma_1^3\sigma_2 +
          \theta_{A,1}^2\theta_{A,3}\theta_{A,4}^3\delta^2\sigma_1^2\sigma_2^2-
          \theta_{A,1}\theta_{A,2}^4\theta_{A,4}\delta^2\sigma_2^4\\
         & +\theta_{A,1}\theta_{A,2}^3\theta_{A,3}\theta_{A,4}\delta^2\sigma_1\sigma_2^3
          +\theta_{A,1}\theta_{A,2}^2\theta_{A,3}^2\theta_{A,4}\delta^2\sigma_1^2\sigma_2^2
           +\theta_{A,1}\theta_{A,2}^2\theta_{A,3}\theta_{A,4}^2\delta^2\sigma_1\sigma_2^3
          \\
          &+3\theta_{A,1}\theta_{A,2}\theta_{A,3}^3\theta_{A,4}\delta^2\sigma_1^3\sigma_2
          -2\theta_{A,1}\theta_{A,2}\theta_{A,3}^3\theta_{A,4}\delta^2\sigma_1^3\sigma_2-
          \theta_{A,1}\theta_{A,3}^3\theta_{A,4}^2\delta^2\sigma_1^3\sigma_2\\
          & +\theta_{A,2}^4\theta_{A,3}\theta_{A,4}\delta^2\sigma_2^4+
          \theta_{A,2}^3\theta_{A,3}^3\delta^2\sigma_1^2\sigma_2^2
          +\theta_{A,2}^3\theta_{A,3}^2\theta_{A,4}\delta^2\sigma_1\sigma_2^3+
          \theta_{A,2}^2\theta_{A,3}^4\delta^2\sigma_1^3\sigma_2
          \\
          & +\theta_{A,2}^2\theta_{A,3}^3\theta_{A,4}\delta^2\sigma_1^2\sigma_2^2
          +\theta_{A,2}\theta_{A,3}^5\delta^2\sigma_1^4
           -\theta_{A,2}\theta_{A,3}^4\theta_{A,4}\delta^2\sigma_1^3\sigma_2
          -\theta_{A,3}^5\theta_{A,4}\delta^2\sigma_1^4
          \\
          &+\theta_{A,1}^2\theta_{A,2}\theta_{A,4}^2\delta\sigma_1\sigma_2^2+
            2\theta_{A,1}^2\theta_{A,3}\theta_{A,4}^2\delta\sigma_1^2\sigma_2 +
          \theta_{A,1}^2\theta_{A,4}^3\delta\sigma_1\sigma_2^2 +
          \theta_{A,1}\theta_{A,2}^3\theta_{A,4}\delta\sigma_2^3\\
          & +\theta_{A,1}\theta_{A,2}^2\theta_{A,3}\theta_{A,4}\delta\sigma_1\sigma_2^2-
          \theta_{A,1}\theta_{A,2}\theta_{A,3}^2\theta_{A,4}\delta\sigma_1^2\sigma_2
          -2\theta_{A,1}\theta_{A,2}\theta_{A,3}\theta_{A,4}^2\delta\sigma_1\sigma_2^2\\
          &+\theta_{A,1}\theta_{A,3}^3\theta_{A,4}\delta\sigma_2^3-
          \theta_{A,2}^3\theta_{A,3}\theta_{A,4}\delta\sigma_2^3
           -\theta_{A,2}^2\theta_{A,3}^2\theta_{A,4}\delta\sigma_1\sigma_2^2-
          \theta_{A,2}\theta_{A,3}^4\delta\sigma_1^3\\
          &-\theta_{A,2}\theta_{A,3}^3\theta_{A,4}\delta\sigma_1^2\sigma_2 +
          \theta_{A,1}^2\theta_{A,4}^2\sigma_1\sigma_2-
         \theta_{A,1}\theta_{A,2}\theta_{A,3}\theta_{A,4}\sigma_1\sigma_2\\
          &  \left. -\theta_{A,1}\theta_{A,3}\theta_{A,4}^2\sigma_1\sigma_2+
          \theta_{A,2}\theta_{A,3}^2\theta_{A,4}\sigma_1\sigma_2\right).
        \end{align*}
        The system of linear equations \eqref{SystEquationSolvingCase2FourierSeqExample} has nonzero solution if and only if determinant of $V$ is zero.
        For the case $\theta_{A,4}=0$, the determinant is
       \begin{eqnarray*}
         \mbox{det}(V)&=&\delta\theta_{A,3}^3\theta_{A,2}\sigma_1^3\sigma_2 (\theta_{A,2}^2 \delta \sigma_2^2-\theta_{A,3}\sigma_1)(\delta\theta_{A,2}\sigma_2-1)(\delta\theta_{A,3}\sigma_1-1).
       \end{eqnarray*}
       Therefore, when $\theta_{A,4}=0$, determinant of $V$ is equal zero if and only if
        one of the following cases hold:
        \begin{enumerate}
         \item\label{SolSystemEqsWaveletExample4TermsFourierSerieSameArg:item1} $
           \{\theta_{A,1}, \theta_{A,2}, \theta_{A,3}, \delta\in\mathbb{R}, \sigma_2\not=0,  \sigma_1=0\}$; %\mbox{ or }
         \item\label{SolSystemEqsWaveletExample4TermsFourierSerieSameArg:item2} $  \{\theta_{A,1},\theta_{A,2}, \theta_{A,3}, \sigma_1, \delta\in\mathbb{R},\sigma_2=0\}$; %\mbox{ or }
          \item\label{SolSystemEqsWaveletExample4TermsFourierSerieSameArg:item3} $
            \left\{\theta_{A,1},\theta_{A,2}, \theta_{A,3}\in\mathbb{R},\sigma_1\not=0, \sigma_2\not=0, \delta=0\right\}$;%\mbox{ or }
           \item\label{SolSystemEqsWaveletExample4TermsFourierSerieSameArg:item4} $
            \{\theta_{A,1}\in\mathbb{R},   \theta_{A,3}\not=\frac{1}{\delta\sigma_1}, \theta_{A,3}\not=0,  \sigma_1\not=0,\sigma_2\not=0,\delta\not=0, \theta_{A,2}=0\}$; %\mbox{ or }
           \item\label{SolSystemEqsWaveletExample4TermsFourierSerieSameArg:item5} $
           \{\theta_{A,1}\in\mathbb{R},\theta_{A,2}\not=0, \delta\not=0,\sigma_1\not=0,\sigma_2\not=0,  \theta_{A,3}=0\}$; %\mbox{ or }
           \item\label{SolSystemEqsWaveletExample4TermsFourierSerieSameArg:item6}
           $\left\{\theta_{A,1}, \theta_{A,2}\in\mathbb{R}, \delta\not=0, \sigma_1\not=0,  \sigma_2\not=0, \theta_{A,3}=\frac{1}{\delta \sigma_1}\right\}$; %\mbox{ or }
           \item\label{SolSystemEqsWaveletExample4TermsFourierSerieSameArg:item7}
           $ \left\{\theta_{A,1}\in\mathbb{R}, \theta_{A,3}\not=\frac{1}{\delta\sigma_1},\, \theta_{A,3}\not=0, \delta\not=0,\sigma_1\not=0,\sigma_2\not=0, \theta_{A,2}=\frac{1}{\delta \sigma_2}\right\}$;
                  % \mbox{ or }
            \item\label{SolSystemEqsWaveletExample4TermsFourierSerieSameArg:item8}
            $ \left\{\theta_{A,1}\in\mathbb{R},\theta_{A,2}\not=0, \theta_{A,2}^2\not=\frac{1}{\delta^2 \sigma_2^2}, \delta \not=0,\, \sigma_1 \not=0,\, \sigma_2\not=0,\, \theta_{A,3}=\frac{\delta \theta_{A,2}^2\sigma_2^2}{ \sigma_1}\right\}$.
         \end{enumerate}

        We search for the corresponding solutions for $\theta_{B,i}$, $i=1,2,3,4$ in these cases.
        \\
          \noindent\ref{SolSystemEqsWaveletExample4TermsFourierSerieSameArg:item1}. If
          $\{\theta_{A,1}, \theta_{A,2}, \theta_{A,3}, \delta\in\mathbb{R}, \sigma_2\not=0,  \sigma_1=0\}$
          then we subdivide into the following subcases.
          \begin{itemize}[leftmargin=*]
            \item If $\theta_{A,1}=\theta_{A,2}=0$ then $\theta_{B,1},\theta_{B,2},\theta_{B,3},\theta_{B,4}\in\mathbb{R}$ can be any. The corresponding
                operators are
                \begin{align*}
                  (Ax)(t) & = \int\limits_{\alpha_1}^{\beta_1} I_{[\alpha,\beta]}(t) \theta_{A,3}\sin(\omega t)\sin(\omega s)x(s)d\mu_s, \\
                  (Bx)(t) & = \int\limits_{\alpha_1}^{\beta_1} I_{[\alpha,\beta]}(t)[ \theta_{B,1}\sin(\omega t)\cos(\omega s)+\theta_{B,2}\cos(\omega t)\cos(\omega s)
                  \\
                  &+\theta_{B,3}\sin(\omega t)\sin(\omega s)+\theta_{B,4}\cos(\omega t)\sin(\omega s)]x(s)d\mu_s,
                \end{align*}
              for almost every $t$. These operators satisfy $AB=\delta BA^2=BA=0$.

              \item If $\delta\not=0, \theta_{A,2}=\frac{1}{\delta \sigma_2}$ then $\theta_{B,4}=0,\ \theta_{B,1}=\delta \theta_{A,1}\theta_{B,2}\sigma_2,\ \theta_{B,2},\theta_{B,3}\in\mathbb{R}$.
                  The corresponding
                operators are
                \begin{align*}
                  (Ax)(t) & = \int\limits_{\alpha_1}^{\beta_1} I_{[\alpha,\beta]}(t)[\theta_{A,1}\sin(\omega t)\cos(\omega s)+\frac{1}{\delta \sigma_2}\cos(\omega t)\cos(\omega s)
                  \\
                 & +\theta_{A,3}\sin(\omega t)\sin(\omega s)]x(s)d\mu_s, \\
                  (Bx)(t) & = \int\limits_{\alpha_1}^{\beta_1} I_{[\alpha,\beta]}(t)[ \delta\theta_{A,1}\theta_{B,2}\sigma_2\sin(\omega t)\cos(\omega s)+\theta_{B,2}\cos(\omega t)\cos(\omega s)
                  \\
                  &+\theta_{B,3}\sin(\omega t)\sin(\omega s)]x(s)d\mu_s,
                \end{align*}
              for almost every $t$. These operators satisfy $AB=\delta BA^2=BA$.

              \item If $\delta=0$, $\theta_{A,1}=\theta_{A,2}=0$, then  $\theta_{B,1},\theta_{B,2},\theta_{B,3},\theta_{B,4}\in\mathbb{R}$ can be any. The corresponding
                operators are
                \begin{align*}
                  (Ax)(t) & = \int\limits_{\alpha_1}^{\beta_1} I_{[\alpha,\beta]}(t) \theta_{A,3}\sin(\omega t)\sin(\omega s)x(s)d\mu_s, \\
                  (Bx)(t) & = \int\limits_{\alpha_1}^{\beta_1} I_{[\alpha,\beta]}(t)[ \theta_{B,1}\sin(\omega t)\cos(\omega s)+\theta_{B,2}\cos(\omega t)\cos(\omega s)
                  \\
                  &+\theta_{B,3}\sin(\omega t)\sin(\omega s)+\theta_{B,4}\cos(\omega t)\sin(\omega s)]x(s)d\mu_s,
                \end{align*}
              for almost every $t$. These operators satisfy $AB=0$.

              \item If $\delta=0$, $\theta_{A,1}\not=0$ or $\theta_{A,2}\not=0$ then  $\theta_{B,2}=\theta_{B,4}=0,$  and $\theta_{B,1},\theta_{B,3}\in\mathbb{R}$. The corresponding
                operators are
                \begin{align*}
                  (Ax)(t) & = \int\limits_{\alpha_1}^{\beta_1} I_{[\alpha,\beta]}(t)[\theta_{A,1}\sin(\omega t)\cos(\omega s)+\theta{A,2}\cos(\omega t)\cos(\omega s)
                  \\
                 & +\theta_{A,3}\sin(\omega t)\sin(\omega s)]x(s)d\mu_s, \\
                  (Bx)(t) & = \int\limits_{\alpha_1}^{\beta_1} I_{[\alpha,\beta]}(t)[ \theta_{B,1}\sin(\omega t)\cos(\omega s)+\theta_{B,3}\sin(\omega t)\sin(\omega s)]x(s)d\mu_s,
                 \end{align*}
              for almost every $t$. These operators satisfy $AB=0$.
          \end{itemize}
          %$\sigma_1=0$ then $\alpha_1=\beta_1$ which imply that operators $A=0$ and $B=0$;

          \noindent\ref{SolSystemEqsWaveletExample4TermsFourierSerieSameArg:item2}. If $\{\theta_{A,1},\theta_{A,2}, \theta_{A,3}, \sigma_1, \delta\in\mathbb{R},\sigma_2=0\}$, then we subdivide into  the following subcases.
          \begin{itemize}[leftmargin=*]
            \item If $\sigma_1=0$ then $\alpha_1=\beta_1$ which implies that operators $A=0$ and $B=0$;

            \item If $\sigma_1\not=0$ and $\theta_{A,3}=0$, then $\theta_{B,1},\theta_{B,2},\theta_{B,3},\theta_{B,4}\in\mathbb{R}$ can be any. The corresponding
                operators are
                \begin{align*}
                  (Ax)(t) & = \int\limits_{\alpha_1}^{\beta_1} I_{[\alpha,\beta]}(t) [\theta_{A,1}\sin(\omega t)\cos(\omega s)+\theta_{A,2}\cos(\omega t)\cos(\omega s)]x(s)d\mu_s, \\
                  (Bx)(t) & = \int\limits_{\alpha_1}^{\beta_1} I_{[\alpha,\beta]}(t)[ \theta_{B,1}\sin(\omega t)\cos(\omega s)+\theta_{B,2}\cos(\omega t)\cos(\omega s)
                  \\
                  &+\theta_{B,3}\sin(\omega t)\sin(\omega s)+\theta_{B,4}\cos(\omega t)\sin(\omega s)]x(s)d\mu_s,
                \end{align*}
              for almost every $t$. These operators satisfy $AB=\delta BA^2=0$. Moreover, for all $x\in L_p(\mathbb{R},\mu)$, $1<p<\infty$ we have, for almost every $t$,
\begin{multline*}
(AB-BA)x(t)=\displaystyle{\int\limits_{\alpha_1}^{\beta_1}} I_{[\alpha,\beta]}(t)\sigma_1 \theta_{A,1}[\theta_{B,3}\sin(\omega t)\cos(\omega s)\\
+\theta_{B,4}\cos(\omega t)\cos(\omega s)] x(s)d\mu_s.
\end{multline*}
              By applying Proposition \ref{PropositionCommutatorZeroFourSequenceCond}, since $\alpha_1<\beta_1$, $[\alpha,\beta]\supseteq [\alpha_1,\beta_1]$ and $\sigma_1\neq0$, operators $A$ and $B$ commute  if and only if either  $\theta_{A,1}=0$ or ($\theta_{B,3}=0$ and $\theta_{B,4}=0$).

             \item If $\sigma_1\not=0$, $\delta=0$ and $\theta_{A,3}\not=0$, then $\theta_{B,1}=\theta_{B,3}=0,$ and $\theta_{B,2},\theta_{B,4}\in\mathbb{R}$. The corresponding
                operators are
                \begin{align*}
                  (Ax)(t) & = \int\limits_{\alpha_1}^{\beta_1} I_{[\alpha,\beta]}(t) [\theta_{A,1}\sin(\omega t)\cos(\omega s)+\theta_{A,2}\cos(\omega t)\cos(\omega s)
                  \\
                  &
                  \hspace{5cm} +\theta_{A,3}\sin(\omega t)\sin(\omega s)]x(s)d\mu_s, \\
                  (Bx)(t) & = \int\limits_{\alpha_1}^{\beta_1} I_{[\alpha,\beta]}(t)[\theta_{B,2}\cos(\omega t)\cos(\omega s)
                  +\theta_{B,4}\cos(\omega t)\sin(\omega s)]x(s)d\mu_s,
                \end{align*}
              for almost every $t$. These operators satisfy $AB=0$.

              \item If $\sigma_1\not=0$, $\delta\neq0$, $\theta_{A,1}=0$, $\theta_{A,3}\not=0$ and $\theta_{A,3}\neq\frac{1}{\delta \sigma_1}$, then $\theta_{B,1}=\theta_{B,3}=\theta_{B,4}=0,$ and $\theta_{B,2}\in\mathbb{R}$. The corresponding
                operators are
                \begin{align*}
                  (Ax)(t) & = \int\limits_{\alpha_1}^{\beta_1} I_{[\alpha,\beta]}(t) [\theta_{A,2}\cos(\omega t)\cos(\omega s)
                                    +\theta_{A,3}\sin(\omega t)\sin(\omega s)]x(s)d\mu_s, \\
                  (Bx)(t) & = \int\limits_{\alpha_1}^{\beta_1} I_{[\alpha,\beta]}(t)\theta_{B,2}\cos(\omega t)\cos(\omega s)
                  x(s)d\mu_s,
                \end{align*}
              for almost every $t$. These operators satisfy $AB=\delta BA^2=BA=0$.

              \item If $\sigma_1\not=0$, $\delta\neq0$, and $\theta_{A,3}=\frac{1}{\delta \sigma_1}$, then $\theta_{B,4}=0,$ $\theta_{B,1}=\delta \theta_{A,1}\theta_{B,3}\sigma_1$,  $\theta_{B,3},\theta_{B,2}\in\mathbb{R}$. The corresponding
                operators are
                \begin{align*}
                  (Ax)(t) & = \int\limits_{\alpha_1}^{\beta_1} I_{[\alpha,\beta]}(t) [\theta_{A,1}\sin(\omega t)\cos(\omega s)+ \theta_{A,2}\cos(\omega t)\cos(\omega s)\\
                       &             +\frac{1}{\delta\sigma_1}\sin(\omega t)\sin(\omega s)]x(s)d\mu_s, \\
                  (Bx)(t) & = \int\limits_{\alpha_1}^{\beta_1} I_{[\alpha,\beta]}(t)[\delta \theta_{A,1}\theta_{B,3}\sigma_1\sin(\omega t)\cos(\omega s) + \theta_{B,2}\cos(\omega t)\cos(\omega s)\\
                  &+\theta_{B,3}\sin(\omega t)\sin(\omega s)]
                  x(s)d\mu_s,
                \end{align*}
              for almost every $t$. These operators satisfy $AB=\delta BA^2=BA$.

          \end{itemize}

          \noindent\ref{SolSystemEqsWaveletExample4TermsFourierSerieSameArg:item3}. If  $\left\{\theta_{A,1},\theta_{A,2}, \theta_{A,3}\in\mathbb{R},\sigma_1\not=0, \sigma_2\not=0, \delta=\theta_{A,4}=0\right\}$, then we subdivide into the following subcases.
            \begin{itemize}[leftmargin=*]
              \item If $\theta_{A,2}=0$, $\theta_{A,3}\not=0$, then
              we  get the following solution
            \begin{equation*}
              \left\{\theta_{B,1}=-\frac{\theta_{A,1}\sigma_2\theta_{B,2}}{\theta_{A,3\sigma_1}}, \theta_{B,3}=-\frac{\theta_{A,1}\sigma_2\theta_{B,4}}{\theta_{A,3\sigma_1}},\theta_{B,2},\theta_{B,4}\in\mathbb{R} \right\}.
            \end{equation*}
            This corresponds to the following operators:
       \begin{eqnarray*}
         (Ax)(t)&=&\int\limits_{\alpha_1}^{\beta_1}I_{[\alpha,\beta]}(t)[\theta_{A,1}\sin(\omega t)\cos(\omega s)+\theta_{A,3}\sin(\omega t)\sin(\omega s)]x(s)d\mu_s,\\
         (Bx)(t)&=&\int\limits_{\alpha_1}^{\beta_1}I_{[\alpha,\beta]}(t)\left(-\frac{\theta_{A,1}\sigma_2\theta_{B,2}}{\theta_{A,3}\sigma_1}\sin(\omega t)\cos(\omega s)+\theta_{B,2}\cos(\omega t)\cos(\omega s)
         \right. \\
         && \left.
         -\frac{\theta_{A,1}\sigma_2\theta_{B,4}}{\theta_{A,3}\sigma_1}\sin(\omega t)\sin(\omega s)+\theta_{B,4}\cos(\omega t)\sin(\omega s)
         \right)x(s)d\mu_s,
       \end{eqnarray*}
        for almost every $t$. These operators satisfy $AB=0$.

          \item If $\theta_{A,3}=0$ and ($\theta_{A,1}\not=0$ or $\theta_{A,2}\not=0$), then
              we  get the following solution
            \begin{equation*}
              \left\{\theta_{B,2}=\theta_{B,4}=0,\, \theta_{B,1},\theta_{B,3}\in\mathbb{R} \right\}.
            \end{equation*}
            This corresponds to the following operators:
             \begin{gather*}
         (Ax)(t)=\int\limits_{\alpha_1}^{\beta_1}I_{[\alpha,\beta]}(t)[\theta_{A,1}\sin(\omega t)\cos(\omega s)+\theta_{A,2}\cos(\omega t)\cos(\omega s)]x(s)d\mu_s,\\
         (Bx)(t)=\int\limits_{\alpha_1}^{\beta_1}I_{[\alpha,\beta]}(t)\left(\theta_{B,1}\sin(\omega t)\cos(\omega s)+\theta_{B,3}\sin(\omega t)\sin(\omega s)\right)x(s)d\mu_s,
       \end{gather*}
         for almost every $t$. These operators satisfy $AB=0$.
          \end{itemize}

       \noindent\ref{SolSystemEqsWaveletExample4TermsFourierSerieSameArg:item4}. If  $\{\theta_{A,1}\in\mathbb{R},\,\theta_{A,3}\not=\frac{1}{\delta \sigma_1},   \delta\not=0,\sigma_1\not=0,\sigma_2\not=0, \theta_{A,2}=\theta_{A,4}=0\}$, then we subdivide into the following subcases.
       \begin{itemize}[leftmargin=*]
         \item If $\theta_{A,3}\not=0$, then we  get the following solution:
       \begin{equation*}
         \left\{\theta_{B,1}=-\frac{\theta_{A,1}\sigma_2\theta_{B,2}}{\theta_{A,3\sigma_1}},\theta_{B,2}\in\mathbb{R}, \theta_{B,3}=\theta_{B,4}=0\right\}.
       \end{equation*}
       This corresponds to the following operators:
       \begin{eqnarray*}
         (Ax)(t)&=&\int\limits_{\alpha_1}^{\beta_1}I_{[\alpha,\beta]}(t)[\theta_{A,1}\sin(\omega t)\cos(\omega s)+\theta_{A,3}\sin(\omega t)\sin(\omega s)]x(s)d\mu_s,\\
         (Bx)(t)&=&\int\limits_{\alpha_1}^{\beta_1}I_{[\alpha,\beta]}(t)\left(-\frac{\theta_{A,1}\sigma_2\theta_{B,2}}{\theta_{A,3}\sigma_1}\sin(\omega t)\cos(\omega s)+\theta_{B,2}\cos(\omega t)\cos(\omega s)\right)\\
         && \cdot x(s)d\mu_s,
       \end{eqnarray*}
        for almost every $t$. These operators satisfy $AB=\delta BA^2=BA=0$.
        \item If $\theta_{A,3}=0$ and $\theta_{A,1}\not=0$, then we get the following solution
        \begin{equation*}
         \left\{\theta_{B,1},\theta_{B,3}\in\mathbb{R},\ \theta_{B,2}=\theta_{B,4}=0\right\}.
       \end{equation*}
       This corresponds to the following operators:
       \begin{eqnarray*}
         && (Ax)(t)=\int\limits_{\alpha_1}^{\beta_1}I_{[\alpha,\beta]}(t)\theta_{A,1}\sin(\omega t)\cos(\omega s)x(s)d\mu_s,\\
         &&(Bx)(t)=\int\limits_{\alpha_1}^{\beta_1}I_{[\alpha,\beta]}(t)\left(\theta_{B,1}\sin(\omega t)\cos(\omega s)+\theta_{B,3}\sin(\omega t)\sin(\omega s)\right)%\\
         %&& \cdot
         x(s)d\mu_s,
       \end{eqnarray*}
        for almost every $t$. They satisfy $AB=\delta BA^2=0$. Moreover, for all $x\in L_p(\mathbb{R},\mu)$, $1<p<\infty$ we have
         \begin{equation*}
           (AB-BA)x(t)=-(BA)x(t)=-\theta_{B,3}\theta_{A,1}\sigma_1\int\limits_{\alpha_1}^{\beta_1} I_{[\alpha,\beta]}(t) \sin(\omega t)\cos(\omega s)x(s)d\mu_s
         \end{equation*}
         for almost every $t$.  By applying Proposition \ref{PropositionCommutatorZeroFourSequenceCond}, since $\alpha_1<\beta_1$, $[\alpha,\beta]\supseteq [\alpha_1,\beta_1]$ and $\sigma_1\neq0$, operators $A$ and $B$ commute  if and only if either $\theta_{A,1}=0$ ($A=0$) or $\theta_{B,3}=0$.
        \end{itemize}

        \noindent\ref{SolSystemEqsWaveletExample4TermsFourierSerieSameArg:item5}. If  $\{\theta_{A,1}\in\mathbb{R},\,\theta_{A,2}\not=0,  \delta\not=0,\sigma_1\not=0,\sigma_2\not=0, \theta_{A,3}=\theta_{A,4}=0\}$, then we subdivide into the following subcases.
        \begin{itemize}[leftmargin=*]
        \item If $\theta_{A,2}\not=\frac{1}{\delta \sigma_2}$ then we get the following solution:
       \begin{equation*}
         \left\{\theta_{B,1}=-\frac{\theta_{A,1}\sigma_1\theta_{B,3}}{\theta_{A,2\sigma_2}},\theta_{B,3}\in\mathbb{R}, \theta_{B,2}=\theta_{B,4}=0\right\}.
       \end{equation*}
        This corresponds to the following operators:
       \begin{eqnarray*}
         &&\hspace{-7mm}(Ax)(t)=\int\limits_{\alpha_1}^{\beta_1}I_{[\alpha,\beta]}(t)[\theta_{A,1}\sin(\omega t)\cos(\omega s)+\theta_{A,2}\cos(\omega t)\cos(\omega s)] x(s)d\mu_s,\\
        && \hspace{-7mm} (Bx)(t)=\int\limits_{\alpha_1}^{\beta_1}\hspace{-1mm}I_{[\alpha,\beta]}(t)\left(-\frac{\theta_{A,1}\sigma_1}{\theta_{A,2}\sigma_2}\theta_{B,3}\sin(\omega t)\cos(\omega s)+\theta_{B,3}\sin(\omega t)\sin(\omega s)\right)\hspace{-1mm} %\\
        % && \cdot
         x(s)d\mu_s,
       \end{eqnarray*}
        for almost every $t$. These operators satisfy $AB=\delta BA^2=BA=0$.

         \item If $\theta_{A,2}=\frac{1}{\delta \sigma_2}$, then we get the following solution:
           \begin{equation*}
         \left\{\theta_{B,1}=\delta\theta_{A,1}(\sigma_2 \theta_{B,2}-\sigma_1\theta_{B,3}), \, \theta_{B,2},\theta_{B,3}\in\mathbb{R}, \theta_{B,4}=0\right\}.
       \end{equation*}
       This corresponds to the following operators:
       \begin{eqnarray*}
         (Ax)(t)&=&\int\limits_{\alpha_1}^{\beta_1}I_{[\alpha,\beta]}(t)\left(\theta_{A,1}\sin(\omega t)\cos(\omega s)+\frac{1}{\delta \sigma_2}\cos(\omega t)\cos(\omega s)\right)x(s)d\mu_s,\\
         (Bx)(t)&=&\int\limits_{\alpha_1}^{\beta_1}I_{[\alpha,\beta]}(t)\left(\delta\theta_{A,1}(\sigma_2\theta_{B,2}-\sigma_1\theta_{B,3})\sin(\omega t)\cos(\omega s)\right.\\
         & & \left.+\theta_{B,2}\cos(\omega t)\cos(\omega s)+\theta_{B,3}\sin(\omega t)\sin(\omega s)\right)x(s)d\mu_s,
       \end{eqnarray*}
        for almost every $t$. These operators satisfy $AB=\delta BA^2=BA$.
        \end{itemize}

       \noindent\ref{SolSystemEqsWaveletExample4TermsFourierSerieSameArg:item6}. If $\{\theta_{A,1}, \theta_{A,2} \in\mathbb{R}, \sigma_1\not=0, \sigma_2\not=0,\ \delta\not=0, \theta_{A,3}=\frac{1}{\delta\sigma_1},\, \theta_{A,4}=0\}$ then we divide in the following subcases.
        \begin{itemize}[leftmargin=*]
            \item If $\theta_{A,2}\sigma_2\delta\not=1$, $\theta_{A,2}\sigma_2\delta\not=-1$, $\theta_{A,1}\not=0$ and $\theta_{A,2}\not=0$, then
         we get the following solution:
       \begin{equation*}
         \left\{\theta_{B,3}=-\frac{\theta_{B,1}(\delta\sigma_2\theta_{A,2}-1)}{\delta\theta_{A,1\sigma_1}},\theta_{B,1}\in\mathbb{R}, \theta_{B,2}=\theta_{B,4}=0\right\}.
       \end{equation*}
       This corresponds to the following operators:
       \begin{eqnarray*}
         (Ax)(t)&=&\int\limits_{\alpha_1}^{\beta_1}I_{[\alpha,\beta]}(t)[\theta_{A,1}\sin(\omega t)\cos(\omega s)+ \theta_{A,2}\cos(\omega t)\cos(\omega s) \\
          & & +\frac{1}{\delta\sigma_1}\sin(\omega t)\sin(\omega s)]x(s)d\mu_s,\\
         (Bx)(t)&=&\int\limits_{\alpha_1}^{\beta_1}I_{[\alpha,\beta]}(t)\left(\theta_{B,1}\sin(\omega t)\cos(\omega s)\right. \\
         && \left. -\frac{\theta_{B,1}(\delta\sigma_2\theta_{A,2}-1)}{\delta\theta_{A,1}\sigma_1}\sin(\omega t)\sin(\omega s)\right)
         x(s)d\mu_s,
       \end{eqnarray*}
       for almost every $t$. %$\omega=\frac{2\pi m}{\lambda_t}$, $\frac{m}{\lambda_t}(\beta_1-\alpha_1)\in\mathbb{Z}$.
       They satisfy $AB=\delta BA^2=BA$.% Moreover $BA^2\not=0$, $AB\not=BA$ under some choices of parameters.

        \item If $\theta_{A,2}=-\frac{1}{\delta \sigma_2}$, $\theta_{A,1}\not=0$ and $\theta_{A,2}\not=0$, then
         we get the following solution:
       \begin{equation*}
         \left\{\theta_{B,1},\theta_{B,3}\in\mathbb{R}, \theta_{B,2}=\theta_{B,4}=0\right\}.
       \end{equation*}
       This corresponds to the following operators:
       \begin{eqnarray*}
         (Ax)(t)&=&\int\limits_{\alpha_1}^{\beta_1}I_{[\alpha,\beta]}(t)\left(\theta_{A,1}\sin(\omega t)\cos(\omega s)- \frac{1}{\delta \sigma_2}\cos(\omega t)\cos(\omega s) \right. \\
          & & \left. +\frac{1}{\delta\sigma_1}\sin(\omega t)\sin(\omega s)\right)x(s)d\mu_s,\\
         (Bx)(t)&=&\int\limits_{\alpha_1}^{\beta_1}I_{[\alpha,\beta]}(t)\left(\theta_{B,1}\sin(\omega t)\cos(\omega s)+ \theta_{B,3}\sin(\omega t)\sin(\omega s)\right)x(s)d\mu_s,
       \end{eqnarray*}
        for almost every $t$. They satisfy $AB=\delta BA^2$. Moreover, for all $x\in L_p(\mathbb{R},\mu)$, $1<p<\infty$ we have
         \begin{equation*}
           (AB-BA)x(t)=\left(\frac{2\theta_{B,1}}{\delta}-\sigma_1 \theta_{A,1}\theta_{B,3}\right)\int\limits_{\alpha_1}^{\beta_1} I_{[\alpha,\beta]}(t) \sin(\omega t)\cos(\omega s)x(s)d\mu_s
         \end{equation*}
         for almost every $t$.  By applying Proposition \ref{PropositionCommutatorZeroFourSequenceCond}, since $\alpha_1<\beta_1$, $[\alpha,\beta]\supseteq [\alpha_1,\beta_1]$ and $\delta\not=0$,  operators $A$ and $B$ commute  if and only if  $\theta_{B,1}=\frac{\delta}{2} \sigma_1 \theta_{A,1}\theta_{B,3}$.

        \item If $\theta_{A,2}\sigma_2\delta\not=1$, $\theta_{A,1}=0$ and $\theta_{A,2}\not=0$, then
        we get the following solution:
       \begin{equation*}
         \left\{\theta_{B,3}\in\mathbb{R}, \theta_{B,1}=\theta_{B,2}=\theta_{B,4}=0\right\}.
       \end{equation*}
       This corresponds to the following operators:
       \begin{eqnarray*}
         (Ax)(t)&=&\int\limits_{\alpha_1}^{\beta_1}I_{[\alpha,\beta]}(t)[\theta_{A,2}\cos(\omega t)\cos(\omega s)  +\frac{1}{\delta\sigma_1}\sin(\omega t)\sin(\omega s)] x(s)d\mu_s,\\
         (Bx)(t)&=&\int\limits_{\alpha_1}^{\beta_1}I_{[\alpha,\beta]}(t)\theta_{B,3}\sin(\omega t)\sin(\omega s)x(s)d\mu_s,
       \end{eqnarray*}
       for almost every $t$. These operators satisfy $AB=\delta BA^2=BA$.

        \item If $\theta_{A,2}=0$ and $\theta_{A,1}=0$, then
        we get the following solution:
       \begin{equation*}
         \left\{\theta_{B,2},\theta_{B,3}\in\mathbb{R}, \theta_{B,1}=\theta_{B,4}=0\right\}.
       \end{equation*}
        This corresponds to the following operators:
       \begin{eqnarray*}
         (Ax)(t)&=&\int\limits_{\alpha_1}^{\beta_1}I_{[\alpha,\beta]}(t)\frac{1}{\delta\sigma_1}\sin(\omega t)\sin(\omega s)x(s)d\mu_s,\\
         (Bx)(t)&=&\int\limits_{\alpha_1}^{\beta_1}I_{[\alpha,\beta]}(t)\left(\theta_{B,2}\cos(\omega t)\cos(\omega s) +\theta_{B,3}\sin(\omega t)\sin(\omega s)\right)x(s)d\mu_s,
       \end{eqnarray*}
       for almost every $t$. These operators satisfy $AB=\delta BA^2=BA$.

       \item If $\theta_{A,2}=0$ and $\theta_{A,1}\not=0$, then
        we get the following solution:
       \begin{equation*}
         \left\{\theta_{B,1}=\delta\theta_{A,1}(\sigma_1\theta_{B,3}-\sigma_2\theta_{B,2}),\, \theta_{B,2},\theta_{B,3}\in\mathbb{R}, \, \theta_{B,4}=0\right\}.
       \end{equation*}
        This corresponds to the following operators:
       \begin{eqnarray*}
         (Ax)(t)&=&\int\limits_{\alpha_1}^{\beta_1}[\theta_{A,1}\sin(\omega t)\cos(\omega s)+ \frac{1}{\delta\sigma_1}\sin(\omega t)\sin(\omega s)]x(s)d\mu_s,\\
         (Bx)(t)&=&\int\limits_{\alpha_1}^{\beta_1}\left(\delta\theta_{A,1}(\sigma_1\theta_{B,3}-\sigma_2\theta_{B,2})\sin(\omega t)\cos(\omega s) \right.\\
         & & \left.
         +\theta_{B,2}\cos(\omega t)\cos(\omega s) +\theta_{B,3}\sin(\omega t)\sin(\omega s)\right)x(s)d\mu_s,
       \end{eqnarray*}
       for almost every $t$. These operators satisfy $AB=\delta BA^2=BA$.

       \item If $\theta_{A,2}=\frac{1}{\delta\sigma_2}$ and $\theta_{A,1}\not=0$, then
        we get the following solution:
       \begin{equation*}
         \left\{\theta_{B,2}=\frac{2\sigma_1}{\sigma_2}\theta_{B,3},\, \theta_{B,1},\theta_{B,3}\in\mathbb{R}, \, \theta_{B,4}=0\right\}.
       \end{equation*}
        This corresponds to the following operators:
       \begin{eqnarray*}
         (Ax)(t)&=&\int\limits_{\alpha_1}^{\beta_1}I_{[\alpha,\beta]}(t)\left(\theta_{A,1}\sin(\omega t)\cos(\omega s)+ \frac{1}{\delta\sigma_2}\cos(\omega t)\cos(\omega s) \right. \\
         & &
         \left. +\frac{1}{\delta\sigma_1}\sin(\omega t)\sin(\omega s)\right)x(s)d\mu_s,\\
         (Bx)(t)&=&\int\limits_{\alpha_1}^{\beta_1}I_{[\alpha,\beta]}(t)\left(\theta_{B,1}\sin(\omega t)\cos(\omega s)
         +\frac{2\sigma_1}{\sigma_2}\theta_{B,3}\cos(\omega t)\cos(\omega s) \right. \\
         & & \left.
         +\theta_{B,3}\sin(\omega t)\sin(\omega s)\right)x(s)d\mu_s,
       \end{eqnarray*}
       for almost every $t$. They satisfy $AB=\delta BA^2$. Moreover, for all $x\in L_p(\mathbb{R},\mu)$, $1<p<\infty$ we have, for almost every $t$,
         \begin{equation*}
           (AB-BA)x(t)=\sigma_1 \theta_{A,1}\theta_{B,3}\int\limits_{\alpha_1}^{\beta_1} I_{[\alpha,\beta]}(t) \sin(\omega t)\cos(\omega s)x(s)d\mu_s.
         \end{equation*}
         By applying Proposition \ref{PropositionCommutatorZeroFourSequenceCond}, since $\alpha_1<\beta_1$, $[\alpha,\beta]\supseteq [\alpha_1,\beta_1]$ and $\sigma_1\neq0$, operators $A$ and $B$ commute  if and only if either $\theta_{A,1}=0$ or $\theta_{B,3}=0$.

       \item If $\theta_{A,2}=\frac{1}{\delta\sigma_2}$ and $\theta_{A,1}=0$, then
        we get the following solution:
       \begin{equation*}
         \left\{\theta_{B,1}, \theta_{B,2},\theta_{B,3},\theta_{B,4}\in\mathbb{R}\right\}.
       \end{equation*}
        This corresponds to the following operators:
       \begin{eqnarray*}
         (Ax)(t)&=&\int\limits_{\alpha_1}^{\beta_1}I_{[\alpha,\beta]}(t)\left(\frac{1}{\delta\sigma_2}\cos(\omega t)\cos(\omega s)+  \frac{1}{\delta\sigma_1}\sin(\omega t)\sin(\omega s)\right)x(s)d\mu_s,\\
         (Bx)(t)&=&\int\limits_{\alpha_1}^{\beta_1}I_{[\alpha,\beta]}(t)\left(\theta_{B,1}\sin(\omega t)\cos(\omega s)
         +\theta_{B,2}\cos(\omega t)\cos(\omega s)\right. \\
         & & \left.
         +\theta_{B,3}\sin(\omega t)\sin(\omega s)+\theta_{B,4}\cos(\omega t)\sin(\omega s)\right)x(s)d\mu_s,
       \end{eqnarray*}
       for almost every $t$. These operators satisfy $AB=\delta BA^2=BA$.
     \end{itemize}

       \noindent\ref{SolSystemEqsWaveletExample4TermsFourierSerieSameArg:item7}. If $\{\theta_{A,1}\in\mathbb{R}, \theta_{A,3}\not=0, \theta_{A,3}\not=\frac{1}{\delta\sigma_1}, \sigma_1\not=0, \sigma_2\not=0,\ \delta\not=0, \theta_{A,2}=\frac{1}{\delta\sigma_2},\, \theta_{A,4}=0\}$, then we get the following solution:
       \begin{equation*}
        \left\{\theta_{B,1}=\frac{\delta \theta_{A,1}\sigma_2}{1-\delta \theta_{A,3}\sigma_1}\theta_{B,2}, \theta_{B,2}\in\mathbb{R},\, \theta_{B,3}=\theta_{B,4}=0\right\}.
       \end{equation*}
        This corresponds to the following operators:
       \begin{eqnarray*}
         (Ax)(t)&=&\int\limits_{\alpha_1}^{\beta_1}I_{[\alpha,\beta]}(t)\left(\theta_{A,1}\sin(\omega t)\cos(\omega s)+ \frac{1}{\delta\sigma_2}\cos(\omega t)\cos(\omega s) \right. \\
         & & \left. + \theta_{A,3}\sin(\omega t)\sin(\omega s)\right)x(s)d\mu_s,\\
         (Bx)(t)&=&\int\limits_{\alpha_1}^{\beta_1}I_{[\alpha,\beta]}(t)\left(\frac{\delta\theta_{A,1}\sigma_2}{1-\delta\theta_{A,3}\sigma_1}
         \theta_{B,2}\sin(\omega t)\cos(\omega s)\right. \\
         & & \left.+\theta_{B,2}\cos(\omega t)\cos(\omega s)\right)x(s)d\mu_s,
       \end{eqnarray*}
       for almost every $t$. These operators satisfy $AB=\delta BA^2=BA$.

       \noindent\ref{SolSystemEqsWaveletExample4TermsFourierSerieSameArg:item8}. If $ \left\{\theta_{A,1}\in\mathbb{R},\theta_{A,2}\not=0, \theta_{A,2}^2\not=\frac{1}{\delta^2 \sigma_2^2},\, \delta \not=0,\, \sigma_1 \not=0,\, \sigma_2\not=0,\, \theta_{A,3}=\frac{\delta \theta_{A,2}^2\sigma_2^2}{ \sigma_1}\right\}$,
        then we get the following solution:
            \begin{equation*}
                    \left\{\theta_{B,1}\in\mathbb{R},\, \theta_{B,2}=\theta_{B,3}=\theta_{B,4}=0\right\}.
            \end{equation*}
         This corresponds to the following operators:
         \begin{eqnarray*}
         (Ax)(t)&=&\int\limits_{\alpha_1}^{\beta_1}I_{[\alpha,\beta]}(t)\left(\theta_{A,1}\sin(\omega t)\cos(\omega s)+ \theta_{A,2}\cos(\omega t)\cos(\omega s) \right. \\
         & & \left. + \frac{\delta\theta_{A,2}^2\sigma_2^2}{\sigma_1}\sin(\omega t)\sin(\omega s)\right)x(s)d\mu_s,\\
         (Bx)(t)&=&\int\limits_{\alpha_1}^{\beta_1}I_{[\alpha,\beta]}(t)\theta_{B,1}
         \sin(\omega t)\cos(\omega s)x(s)d\mu_s,
        \end{eqnarray*}
         for almost every $t$. They satisfy $AB=\delta BA^2$.
          Moreover, for all $x\in L_p(\mathbb{R},\mu)$, $1<p<\infty$ we have, for almost every $t$,
         \begin{equation*}
           (AB-BA)x(t)=\theta_{B,1}\theta_{A,2}\sigma_2(\delta \sigma_2 \theta_{A,2}-1)\int\limits_{\alpha_1}^{\beta_1} I_{[\alpha,\beta]}(t) \sin(\omega t)\cos(\omega s)x(s)d\mu_s.
         \end{equation*}
          By applying Proposition \ref{PropositionCommutatorZeroFourSequenceCond}, since $\alpha_1<\beta_1$, $[\alpha,\beta]\supseteq [\alpha_1,\beta_1]$, $\sigma_2\neq0$, $\delta\not=0$ and $\theta_{A,2}\not=\frac{1}{\delta \sigma_2}$, operators $A$ and $B$ commute  if and only if $\theta_{A,2}=0$ or  $\theta_{B,1}=0$ ($B=0$).

      \subsection*{Case 2B:}  We consider that operator $B$ is given and we look for operator $A$. One can search for operator $A$ by solving the system of equation in $\theta_{A,i}$, $i=1,2,3,4$. In this case the system of equations \eqref{SystEquationSolvingCase2FourierSeqExample} is nonlinear and generally more difficult to solve. Therefore, we consider particular cases.  If operator $B$ is such that the coefficients satisfy $\theta_{B,1}=\theta_{B,3}$ and $\theta_{B,2}=\theta_{B,4}$ then, we have, for instance, the following non-zero solution for $\{\theta_{A,i}\}$:
      \begin{gather}\nonumber
        \left\{\theta_{A,1}=\theta_{A,3}=\frac{\theta_{B,1}}{\delta(\theta_{B,1}\sigma_1+\theta_{B,2}\sigma_2)},\
        \theta_{A,2}=\theta_{A,4}=\frac{\theta_{B,2}}{\delta(\theta_{B,1}\sigma_1+\theta_{B,2}\sigma_2)},\right. \\
        \label{SolSystEquationSolvinForAGivenBpairsCoefEqual1}
        \left. {\delta(\theta_{B,1}\sigma_1+\theta_{B,2}\sigma_2)}\not=0
        \right\},
       \end{gather}
        %\\ \nonumber
        %\left\{\theta_{A,1}=\frac{\theta_{B,1}(2\theta_{B,1}\sigma_1+3\theta_{B,2}\sigma_2)}{\theta_{B,2}\delta(3\theta_{B,1}\sigma_1+4\theta_{B,2}\sigma_2)\sigma_2},
        % \,  \theta_{A,2}=\frac{\theta_{B,1}\sigma_1+2\theta_{B,2}\sigma_2}{\sigma_2\delta(3\theta_{B,1}\sigma_1+4\theta_{B,2}\sigma_2)},\, \right. \\
         % \nonumber
         % \theta_{A,3}=-\frac{\theta_{B,1}\sigma_1+2\theta_{B,2}\sigma_2}{\sigma_1\delta(3\theta_{B,1}\sigma_1+4\theta_{B,2}\sigma_2)},\,
        % \theta_{A,4}=\frac{\theta_{B,2}}{\delta(3\theta_{B,1}\sigma_1+4\theta_{B,2}\sigma_2)},\\ \label{SolSystEquationSolvinForAGivenBpairsCoefEqual2}
        % \left. \delta\not=0,\,\sigma_1\not=0,\, \sigma_2\not=0, \, \theta_{B,2}\not=0,\, 3\theta_{B,1}\sigma_1+4\theta_{B,2}\sigma_2\not=0
        % \right\}.
      %\end{gather}
       where $\sigma_1$ is given by \eqref{ConstantSigma1Case2Omega} and $\sigma_2$ is given by \eqref{ConstantSigma2Case2Omega}.
       The corresponding pair of operators $(A,B)$, if coefficients of operator $A$ are given in \eqref{SolSystEquationSolvinForAGivenBpairsCoefEqual1},  is
      \begin{eqnarray*}
        (Ax)(t)&=&\frac{1}{\delta(\theta_{B,1}\sigma_1+\theta_{B,2}\sigma_2)}\int\limits_{\alpha_1}^{\beta_1} \left(\theta_{B,1}\sin(\omega t)\cos(\omega s)\right.
        +\theta_{B,2}\cos(\omega t)\cos(\omega s)\\
        && +\theta_{B,1}\sin(\omega t)\sin(\omega s)+
        \left.\theta_{B,2}\cos(\omega t)\sin(\omega s)\right)x(s)d\mu_s\\
        (Bx)(t)&=&\int\limits_{\alpha_1}^{\beta_1} [\theta_{B,1}\sin(\omega t)\cos(\omega s)+\theta_{B,2}\cos(\omega t)\cos(\omega s)\\
        & & +\theta_{B,1}\sin(\omega t)\sin(\omega s)+\theta_{B,2}\cos(\omega t)\sin(\omega s)]x(s)d\mu_s,
        \end{eqnarray*}
        for almost every $t$.  Moreover, they satisfy $AB=\delta BA^2=BA$.
       %}
   %\end{example}
     % \noindent The corresponding pair of operators $(A,B)$ if coefficients of operator $A$ are given in \eqref{SolSystEquationSolvinForAGivenBpairsCoefEqual2}  is
      %  \begin{eqnarray*}
      %  (Ax)(t)&=&\int\limits_{\alpha_1}^{\beta_1} \left(\frac{\theta_{B,1}(2\theta_{B,1}\sigma_1+3\theta_{B,2}\sigma_2)}{\theta_{B,2}\delta(3\theta_{B,1}\sigma_1+4\theta_{B,2}\sigma_2)\sigma_2}
      %  \sin(\omega t)\cos(\omega s)+\right.\\
      %  &&+\frac{\theta_{B,1}\sigma_1+2\theta_{B,2}\sigma_2}{\sigma_2\delta(3\theta_{B,1}\sigma_1+4\theta_{B,2}\sigma_2)}\cos(\omega t)\cos(\omega s)-\\
       % &&-\frac{\theta_{B,1}\sigma_1+2\theta_{B,2}\sigma_2}{\sigma_1\delta(3\theta_{B,1}\sigma_1+4\theta_{B,2}\sigma_2)}
       % \sin(\omega t)\sin(\omega s)+\\
       % &&\left.+\frac{\theta_{B,2}}{\delta(3\theta_{B,1}\sigma_1+4\theta_{B,2}\sigma_2)}\cos(\omega t)\sin(\omega s)\right)x(s)d\mu_s, \\
       % (Bx)(t)&=&\int\limits_{\alpha_1}^{\beta_1} [\theta_{B,1}\sin(\omega t)\cos(\omega s)+\theta_{B,2}\cos(\omega t)\cos(\omega s)+\\
       % & & +\theta_{B,1}\sin(\omega t)\sin(\omega s)+\theta_{B,2}\cos(\omega t)\sin(\omega s)]x(s)d\mu_s,
     % \end{eqnarray*}
      % almost everywhere, $\omega=\frac{2\pi m}{\lambda_t}$, $\frac{2m}{\lambda_t}(\beta_1-\alpha_1)\in\mathbb{Z}$ and they satisfy the commutation relation $AB=\delta BA^2$.
   %
   %check calculation above
   %

       If operator $B$ is such that the coefficients satisfy $\theta_{B,3}=\frac{2\sigma_2}{\sigma_1}\theta_{B,2}$, $\theta_{B,4}=0$, $\theta_{B,1},\, \theta_{B,2}\in\mathbb{R}$, $\sigma_1\not=0$ is given by \eqref{ConstantSigma1Case2Omega}, $\sigma_2\not=0$ is given by \eqref{ConstantSigma2Case2Omega}, then $AB=\delta BA^2$ if and only if one of the following hold
       \begin{enumerate}[leftmargin=*]
         \item\label{ExampleFourierWaveletOpGivenBCase2BfindB:item1}  $\{\theta_{A,1}=\theta_{A,2}=\theta_{A,3}=\theta_{A,4}=0\}$
         \item\label{ExampleFourierWaveletOpGivenBCase2BfindB:item2}  $\{\theta_{B,1}=\theta_{B,2}=\theta_{B,3}=\theta_{B,4}=0\}$
          \item\label{ExampleFourierWaveletOpGivenBCase2BfindB:item3}  $\{\theta_{A,3}=\frac{\delta  \sigma_2^2}{\sigma_1}\theta^2_{A,2}, \theta_{A,1}, \theta_{A,2}, \delta\in\mathbb{R},\theta_{A,4}= \theta_{B,2}=\theta_{B,3}=\theta_{B,4}=0, \sigma_1,\sigma_2\in\mathbb{R}\setminus\{0\},\\
              \theta_{B,1}\in\mathbb{R}\setminus\{0\}\}$
          \item\label{ExampleFourierWaveletOpGivenBCase2BfindB:item4} $\left\{\theta_{A,1}=\frac{\theta_{B,1}}{\theta_{B,2}\delta \sigma_2}, \theta_{A,3}=\frac{1}{\delta\sigma_1},\
       \theta_{A,2}=\theta_{A,4}=0,
          \delta\not=0,\sigma_1\not=0, \sigma_2\not=0,  \theta_{B,2}\not=0
        \right\}$
        \item\label{ExampleFourierWaveletOpGivenBCase2BfindB:item5} $\left\{\theta_{A,1}=-\frac{\theta_{B,1}}{\theta_{B,2}\delta \sigma_2}, \theta_{A,2}=\frac{1}{\delta\sigma_2},
        \theta_{A,3}=\theta_{A,4}=0, \delta\not=0, \sigma_1\not=0, \sigma_2\not=0,  \theta_{B,2}\not=0
        \right\}$
        \item\label{ExampleFourierWaveletOpGivenBCase2BfindB:item6} $\left\{\theta_{A,1}=0, \theta_{A,2}=\frac{1}{\delta\sigma_2},
        \theta_{A,3}=\frac{1}{\delta \sigma_1}, \theta_{A,4}=0,  \delta\not=0,\,\sigma_1\not=0, \sigma_2\not=0, \theta_{B,1}\not=0\right.$  \\ $\left.    \theta_{B,2}\not=0
        \right\}$
        \item\label{ExampleFourierWaveletOpGivenBCase2BfindB:item7} $\left\{\theta_{A,1}\in\mathbb{R}\setminus\{0\}, \theta_{A,2}=\frac{1}{\delta\sigma_2},
        \theta_{A,3}=\frac{1}{\delta \sigma_1}, \theta_{A,4}=0,  \delta\not=0,\sigma_1\not=0, \sigma_2\not=0,
           \theta_{B,1}\not=0, \right.$\\  $\left. \theta_{B,2}=\theta_{B,3}=\theta_{B,4}=0
        \right\}$

         \item\label{ExampleFourierWaveletOpGivenBCase2BfindB:item8}  $\left\{\theta_{A,1}=0,\  \theta_{A,2}=\frac{1}{\delta\sigma_2},\
        \theta_{A,3}=\frac{1}{\delta \sigma_1}, \ \theta_{A,4}\in\mathbb{R},\,  \delta\not=0,\,\sigma_1\not=0,\, \sigma_2\not=0,\, \theta_{B,1}=0,\right.$ \\ $\left. \theta_{B,2}\not=0
        \right\}$.
       \end{enumerate}
       We now search for the corresponding operators in these different cases.

       \noindent\ref{ExampleFourierWaveletOpGivenBCase2BfindB:item1}. If  $\{\theta_{A,1}=\theta_{A,2}=\theta_{A,3}=\theta_{A,4}=0\}$, then $A=0$.

         \noindent\ref{ExampleFourierWaveletOpGivenBCase2BfindB:item2}. If  $\{\theta_{B,1}=\theta_{B,2}=\theta_{B,3}=\theta_{B,4}=0\}$, then $B=0$.

         \noindent\ref{ExampleFourierWaveletOpGivenBCase2BfindB:item3} If  $\{\theta_{A,3}=\frac{\delta  \sigma_2^2}{\sigma_1}\theta^2_{A,2}, \theta_{A,1}, \theta_{A,2}, \delta\in\mathbb{R}, \theta_{A,4}= \theta_{B,2}=\theta_{B,3}=\theta_{B,4}=0, \sigma_1,\sigma_2,\in\mathbb{R}\setminus\{0\},
         \\ \theta_{B,1}\in\mathbb{R}\setminus\{0\}\},$ then
         the corresponding pairs of non-zero operators $(A,B)$ are
        \begin{eqnarray*}
          (A x)(t) &=& \int\limits_{\alpha_1}^{\beta_1} I_{[\alpha,\beta]}(t) \left({\theta_{A,1}}\sin(\omega t)\cos(\omega s) +\theta_{A,2}\cos(\omega t)\cos(\omega s) \right. \\
          && \hspace{3cm} +\left.
        \frac{\delta \sigma_2^2 \theta_{A,2}^2}{\sigma_1}\sin(\omega t)\sin(\omega s)\right)x(s)d\mu_s,\\
        (Bx)(t)&=&\int\limits_{\alpha_1}^{\beta_1} I_{[\alpha,\beta]}(t)\theta_{B,1}\sin(\omega t)\cos(\omega s)x(s)d\mu_s,
        \end{eqnarray*}
        for almost every $t$. They satisfy $AB=\delta BA^2$. This pair of operators was found in case 2A, item \ref{SolSystemEqsWaveletExample4TermsFourierSerieSameArg:item8}.

        \noindent\ref{ExampleFourierWaveletOpGivenBCase2BfindB:item4}. If  $\left\{\theta_{A,1}=\frac{\theta_{B,1}}{\theta_{B,2}\delta \sigma_2},\ \theta_{A,3}=\frac{1}{\delta\sigma_1},\
       \theta_{A,2}=\theta_{A,4}=0,
        \,  \delta\not=0,\,\sigma_1\not=0,\, \sigma_2\not=0, \, \theta_{B,2}\not=0,\right. \\
        \left. \theta_{B,1}\in\mathbb{R}
        \right\},$  then the corresponding pairs of non-zero operators $(A,B)$ are
        \begin{eqnarray*}
          (A x)(t) &=& \int\limits_{\alpha_1}^{\beta_1} I_{[\alpha,\beta]}(t) \left(\frac{\theta_{B,1}}{\delta\theta_{B,2}\sigma_2}\sin(\omega t)\cos(\omega s) +\frac{1}{\sigma_1\delta}
        \sin(\omega t)\sin(\omega s)\right)x(s)d\mu_s,\\
        (Bx)(t)&=&\int\limits_{\alpha_1}^{\beta_1} I_{[\alpha,\beta]}(t)\big(\theta_{B,1}\sin(\omega t)\cos(\omega s)+\theta_{B,2}\cos(\omega t)\cos(\omega s)\\
        & & \hspace{3cm} +\frac{2\sigma_2}{\sigma_1}\theta_{B,2}\sin(\omega t)\sin(\omega s)\big)x(s)d\mu_s,
        \end{eqnarray*}
        for almost every $t$. They satisfy $AB=\delta BA^2=BA$.

        \noindent\ref{ExampleFourierWaveletOpGivenBCase2BfindB:item5}. If  $\left\{\theta_{A,1}=-\frac{\theta_{B,1}}{\theta_{B,2}\delta \sigma_2},\  \theta_{A,2}=\frac{1}{\delta\sigma_2},\
        \theta_{A,3}=\theta_{A,4}=0,\,  \delta\not=0,\,\sigma_1\not=0,\, \sigma_2\not=0, \, \theta_{B,2}\not=0,
        \right.\\
        \left.
         \theta_{B,1}\in\mathbb{R}
        \right\},$ then the corresponding pairs of non-zero operators $(A,B)$ are
$$   \begin{array}{lll}
          (A x)(t)&=& {\displaystyle \int\limits_{\alpha_1}^{\beta_1}} I_{[\alpha,\beta]}(t) \left(-\frac{\theta_{B,1}}{\delta\theta_{B,2}\sigma_2}\sin(\omega t)\cos(\omega s) +\frac{1}{\sigma_2\delta}
        \cos(\omega t)\cos(\omega s)\right)x(s)d\mu_s,\\
(Bx)(t)&=& {\displaystyle \int\limits_{\alpha_1}^{\beta_1}} I_{[\alpha,\beta]}(t)\big(\theta_{B,1}\sin(\omega t)\cos(\omega s)+\theta_{B,2}\cos(\omega t)\cos(\omega s)\\
        & & +\frac{2\sigma_2}{\sigma_1}\theta_{B,2}\sin(\omega t)\sin(\omega s)\big)x(s)d\mu_s,
        \end{array}
$$
        for almost every $t$. They satisfy $AB=\delta BA^2=BA$.

        \noindent\ref{ExampleFourierWaveletOpGivenBCase2BfindB:item6}. If  $\left\{\theta_{A,1}=0,\  \theta_{A,2}=\frac{1}{\delta\sigma_2},\
        \theta_{A,3}=\frac{1}{\delta \sigma_1},\, \theta_{A,4}=0, \delta\not=0,\,\sigma_1\not=0, \sigma_2\not=0, \theta_{B,1}\not=0, \right.$ \\ $\left.\theta_{B,2}\not=0
        \right\},$ then the corresponding pairs of non-zero operators $(A,B)$ are
        \begin{eqnarray*}
          (A x)(t)&=&\int\limits_{\alpha_1}^{\beta_1} I_{[\alpha,\beta]}(t) \left(\frac{1}{\sigma_2\delta}
        \cos(\omega t)\cos(\omega s)+\frac{1}{\delta \sigma_1}\sin(\omega t)\sin(\omega s)\right)x(s)d\mu_s,\\
        (Bx)(t)&=&\int\limits_{\alpha_1}^{\beta_1} I_{[\alpha,\beta]}(t)\big(\theta_{B,1}\sin(\omega t)\cos(\omega s)+\theta_{B,2}\cos(\omega t)\cos(\omega s)\\
        & & \hspace{2cm} +\frac{2\sigma_2}{\sigma_1}\theta_{B,2}\sin(\omega t)\sin(\omega s)\big)x(s)d\mu_s,
        \end{eqnarray*}
        for almost every $t$. They satisfy $AB=\delta BA^2=BA$.

        \noindent\ref{ExampleFourierWaveletOpGivenBCase2BfindB:item7}. If  $\theta_{A,1}\in\mathbb{R}\setminus\{0\},  \theta_{A,2}=\frac{1}{\delta\sigma_2},
        \theta_{A,3}=\frac{1}{\delta \sigma_1}, \theta_{A,4}=0,\,  \delta\not=0,\,\sigma_1\not=0, \sigma_2\not=0, \theta_{B,1}\not=0,$\\ $\theta_{B,2}=\theta_{B,3}=\theta_{B,4}=0,$ then the corresponding pairs of non-zero operators $(A,B)$ are
        \begin{eqnarray*}
          (A x)(t)&=&\int\limits_{\alpha_1}^{\beta_1} I_{[\alpha,\beta]}(t) \big(\theta_{A,1}\sin(\omega t)\cos(\omega s) +\frac{1}{\sigma_2\delta}
        \cos(\omega t)\cos(\omega s)  \\
        && \hspace{3cm} +  \frac{1}{\delta \sigma_1}\sin(\omega t)\sin(\omega s)\big)x(s)d\mu_s,\\
        (Bx)(t)&=&\int\limits_{\alpha_1}^{\beta_1} I_{[\alpha,\beta]}(t)\theta_{B,1}\sin(\omega t)\cos(\omega s)x(s)d\mu_s,
        \end{eqnarray*}
        for almost every $t$. They satisfy $AB=\delta BA^2=BA$.

        \noindent\ref{ExampleFourierWaveletOpGivenBCase2BfindB:item8}. If  $\left\{\theta_{A,1}=0, \theta_{A,2}=\frac{1}{\delta\sigma_2},
        \theta_{A,3}=\frac{1}{\delta \sigma_1}, \theta_{A,4}\in\mathbb{R}\setminus\{0\},  \delta\not=0, \sigma_1\not=0,
          \sigma_2\not=0, \theta_{B,1}=0,\right.$
          $\left.  \theta_{B,2}\not=0
        \right\},$ then the corresponding pairs of non-zero operators $(A,B)$ are
        \begin{eqnarray*}
          (A x)(t)&=&\int\limits_{\alpha_1}^{\beta_1} I_{[\alpha,\beta]}(t) \big(\frac{1}{\sigma_2\delta}
        \cos(\omega t)\cos(\omega s) + \frac{1}{\delta\sigma_1}\sin(\omega t)\sin(\omega s) \\
        && \hspace{3cm} +  \theta_{A,4}\cos(\omega t)\sin(\omega s)\big)x(s)d\mu_s,\\
        (Bx)(t)&=&\int\limits_{\alpha_1}^{\beta_1} I_{[\alpha,\beta]}(t)\big(\theta_{B,2}\cos(\omega t)\cos(\omega s)+ \frac{2\sigma_2}{\sigma_1}\theta_{B,2}\sin(\omega t)\sin(\omega s)\big)x(s)d\mu_s,
        \end{eqnarray*}
        for almost every $t$. They satisfy $AB=\delta BA^2$. Moreover, for $x\in L_p(\mathbb{R},\mu)$, $1< p< \infty$,
       \begin{equation*}
         (AB-BA)x(t)=\theta_{A,4}\sigma_2 \theta_{B,2}\int\limits_{\alpha_1}^{\beta_1} I_{[\alpha,\beta]}(t)\cos(\omega t)\sin(\omega s)x(s)d\mu_s,
       \end{equation*}
       for almost every $t$.  By applying Proposition \ref{PropositionCommutatorZeroFourSequenceCond}, since $\alpha_1<\beta_1$, $[\alpha,\beta]\supseteq [\alpha_1,\beta_1]$ and $\sigma_2\neq0$, operators $A$ and $B$ commute  if and only if either $\theta_{A,4}=0$ or $\theta_{B,2}=0$ ($B=0$).

       If operator $B$ is such that coefficients satisfy $\theta_{B,1}=\theta_{B,4}=0$, and the first coefficient of operator $A$ is zero, that is, $\theta_{A,1}=0$, then by  \eqref{CondCRWaveletExample4TermsFourierSerie} the commutation relations  $AB=\delta BA^2$ holds true if and only if
       \begin{gather}\label{SystemEqWaveletExTermsFourierArgEqualGivenBFindAa1=0b4=0b3=b2const}
          \left\{\begin{array}{ll}
            \theta_{A,2}\theta_{B,2}\sigma_2 &= \delta \theta_{A,2}^2\theta_{B,2}\sigma_2^2\ , \\
           \theta_{A,3}\theta_{B,3}\sigma_1 & =\delta \theta_{A,3}^2\theta_{B,3}\sigma_1^2\ , \\
      \theta_{A,4}\theta_{B,3}\sigma_1 &= \delta \theta_{A,4}\theta_{A,3}\theta_{B,2}\sigma_1\sigma_2+\delta\theta_{A,2}\theta_{A,4}\theta_{B,2}\sigma_2^2.
          \end{array}\right.
       \end{gather}
       If $\sigma_1\not=0$ is given by \eqref{ConstantSigma1Case2Omega}, $\sigma_2\not=0$ is given by \eqref{ConstantSigma2Case2Omega}, then the system of equations \eqref{SystemEqWaveletExTermsFourierArgEqualGivenBFindAa1=0b4=0b3=b2const} is satisfied  if and only if one of the following hold
       \begin{enumerate}
       \item\label{solSystemEqWaveletExFourierGivenBfindAa1=0:item1}
          $\{\theta_{B,1}=\theta_{B,4}=0, \theta_{B,2}, \theta_{B,3},\delta\in\mathbb{R},\, \theta_{A,1}=\theta_{A,2}=\theta_{A,3}=\theta_{A,4}=0\}$;
       \item\label{solSystemEqWaveletExFourierGivenBfindAa1=0:item2}
          $\{\theta_{B,1}=\theta_{B,2}=\theta_{B,4}=0, \theta_{B,3}\in\mathbb{R}\setminus\{0\},\, \theta_{A,1}=\theta_{A,3}=\theta_{A,4}, \,\theta_{A,2}\in\mathbb{R}\setminus\{0\}, \delta=0\}$;
        \item\label{solSystemEqWaveletExFourierGivenBfindAa1=0:item3}
          $\{\theta_{B,1}=\theta_{B,3}=\theta_{B,4}=0, \theta_{B,2}\in\mathbb{R}\setminus\{0\}, \theta_{A,1}=\theta_{A,2}=0, \theta_{A,3},\theta_{A,4}\in\mathbb{R}, \delta=0,\\ |\theta_{A,3}|+|\theta_{A,4}|>0\}$;
       \item\label{solSystemEqWaveletExFourierGivenBfindAa1=0:item4}
          $\{\theta_{B,1}=\theta_{B,4}=0, \theta_{B,2}\in\mathbb{R}\setminus
           \{0\},\theta_{B,3}\in\mathbb{R}\setminus\{\theta_{B,2}\frac{\sigma_2}{\sigma_1}\}, \theta_{A,1}= \theta_{A,2}=\theta_{A,4}=0,\\ \theta_{A,3}=\frac{1}{\delta\sigma_1}, \delta\not=0\}$;
        \item\label{solSystemEqWaveletExFourierGivenBfindAa1=0:item5}
          $\{\theta_{B,1}=\theta_{B,4}=0, \theta_{B,2}\in\mathbb{R}\setminus
           \{0\},\theta_{B,3}=\theta_{B,2}\frac{\sigma_2}{\sigma_1}, \theta_{A,1}= \theta_{A,2}=0, \theta_{A,3}=\frac{1}{\delta\sigma_1},\\ \theta_{A,4}\in\mathbb{R}, \delta\not=0\}$;
         \item\label{solSystemEqWaveletExFourierGivenBfindAa1=0:item6}
          $\{\theta_{B,1}=\theta_{B,4}=0, \theta_{B,2}\in\mathbb{R}\setminus
           \{0\},\,\theta_{B,3}\in\mathbb{R}\setminus\{\theta_{B,2}\frac{\sigma_2}{\sigma_1}\}, \theta_{A,1}= \theta_{A,3}=\theta_{A,4}=0,\\ \theta_{A,2}=\frac{1}{\delta\sigma_2}, \delta\not=0\}$;
            \item\label{solSystemEqWaveletExFourierGivenBfindAa1=0:item7}
          $\{\theta_{B,1}=\theta_{B,4}=0, \theta_{B,2}\in\mathbb{R}\setminus
           \{0\},\theta_{B,3}=\theta_{B,2}\frac{\sigma_2}{\sigma_1}, \theta_{A,1}= \theta_{A,3}=0, \theta_{A,2}=\frac{1}{\delta\sigma_2},\\ \theta_{A,4}\in\mathbb{R}, \delta\not=0 \}$;
           \item\label{solSystemEqWaveletExFourierGivenBfindAa1=0:item8}
          $\{\theta_{B,1}=\theta_{B,4}=0, \theta_{B,2}\in\mathbb{R}\setminus
           \{0\}, \theta_{B,3}\in\mathbb{R}\setminus\{\theta_{B,2}\frac{2\sigma_2}{\sigma_1}\}, \theta_{A,1}= \theta_{A,4}=0,  \delta\not=0,\\ \theta_{A,2}=\frac{1}{\delta\sigma_2}, \theta_{A,3}=\frac{1}{\delta\sigma_1} \}$;
           \item\label{solSystemEqWaveletExFourierGivenBfindAa1=0:item9}
          $\{\theta_{B,1}=\theta_{B,4}=0, \theta_{B,2}\in\mathbb{R}\setminus
           \{0\},\theta_{B,3}=\theta_{B,2}\frac{2\sigma_2}{\sigma_1}, \theta_{A,1}= 0, \theta_{A,2}=\frac{1}{\delta\sigma_2}, \delta\not=0,\\ \theta_{A,3}=\frac{1}{\delta\sigma_1}, \theta_{A,4}\in\mathbb{R}\}$;
           \item\label{solSystemEqWaveletExFourierGivenBfindAa1=0:item10}
          $\{\theta_{B,1}=\theta_{B,2}=\theta_{B,4}=0, \theta_{B,3}\in\mathbb{R}\setminus
           \{0\},\, \theta_{A,1}= \theta_{A,3}=\theta_{A,4}=0, \theta_{A,2}\in\mathbb{R}\setminus\{0\},\\ \delta\not=0\};$
           \item\label{solSystemEqWaveletExFourierGivenBfindAa1=0:item11}
          $\{\theta_{B,1}=\theta_{B,2}=\theta_{B,4}=0, \theta_{B,3}\in\mathbb{R}\setminus
           \{0\}, \theta_{A,1}= \theta_{A,4}=0, \theta_{A,3}=\frac{1}{\delta\sigma_1} \theta_{A,2}\in\mathbb{R},\\ \delta\not=0\}$;
           \item\label{solSystemEqWaveletExFourierGivenBfindAa1=0:item12}
          $\{\theta_{B,1}=\theta_{B,3}=\theta_{B,4}=0, \theta_{B,2}\in\mathbb{R}\setminus
           \{0\}, \theta_{A,1}= \theta_{A,2}=\theta_{A,4}=0, \theta_{A,3}\in\mathbb{R}\setminus\{0\},\\ \delta\not=0\};$
           \item\label{solSystemEqWaveletExFourierGivenBfindAa1=0:item13}
          $\{\theta_{B,1}=\theta_{B,3}=\theta_{B,4}=0, \theta_{B,2}\in\mathbb{R}\setminus
           \{0\}, \theta_{A,1}=\theta_{A,2}= \theta_{A,3}=0, \theta_{A,4}\in\mathbb{R}, \delta\not=0\};$
           \item\label{solSystemEqWaveletExFourierGivenBfindAa1=0:item14}
          $\{\theta_{B,1}=\theta_{B,3}=\theta_{B,4}=0, \theta_{B,2}\in\mathbb{R}\setminus
           \{0\}, \theta_{A,1}= \theta_{A,4}=0, \theta_{A,2}=\frac{1}{\delta\sigma_1}, \theta_{A,3}\in\mathbb{R},\\  \delta\not=0\};$
           \item\label{solSystemEqWaveletExFourierGivenBfindAa1=0:item15}
          $\{\theta_{B,1}=\theta_{B,3}=\theta_{B,4}=0, \theta_{B,2}\in\mathbb{R}\setminus
           \{0\}, \theta_{A,1}= 0,\, \theta_{A,2}=\frac{1}{\delta\sigma_2}, \theta_{A,3}=-\frac{1}{\delta \sigma_1}, \delta\not=0,\\ \theta_{A,4}\in\mathbb{R}\}$.
       \end{enumerate}

       We now consider the corresponding operators in these different cases.\\
       \noindent\ref{solSystemEqWaveletExFourierGivenBfindAa1=0:item1}.
       If $\{\theta_{B,1}=\theta_{B,4}=0, \theta_{B,2}, \theta_{B,3},\delta\in\mathbb{R}, \theta_{A,1}=\theta_{A,2}=\theta_{A,3}=\theta_{A,4}=0, \sigma_1\neq0,\sigma_2\neq0\},$ then $A=0$;

       \noindent\ref{solSystemEqWaveletExFourierGivenBfindAa1=0:item2}.
         If $\{\theta_{B,1}=\theta_{B,2}=\theta_{B,4}=0, \theta_{B,3}\in\mathbb{R}\setminus\{0\}, \theta_{A,1}=\theta_{A,3}=\theta_{A,4}=0, \theta_{A,2}\in\mathbb{R}\setminus\{0\},\\ \delta=0, \sigma_1\neq0,\sigma_2\neq0\}$, then the corresponding operators are
         \begin{eqnarray*}
           (Ax)(t) &=&  \int\limits_{\alpha_1}^{\beta_1} \theta_{A,2}I_{[\alpha,\beta]}(t) \cos(\omega t)\cos(\omega s)x(s)d\mu_s, \\
           (Bx)(t) &=&  \int\limits_{\alpha_1}^{\beta_1} \theta_{B,3}I_{[\alpha,\beta]}(t) \sin(\omega t)\sin(\omega s)x(s)d\mu_s,
         \end{eqnarray*}
         for almost every $t$. Moreover, these operators satisfy $AB=0$.

         \noindent\ref{solSystemEqWaveletExFourierGivenBfindAa1=0:item3}.
           If $\{\theta_{B,1}=\theta_{B,3}=\theta_{B,4}=0, \theta_{B,2}\in\mathbb{R}\setminus\{0\},\, \theta_{A,1}=\theta_{A,2}=0, \theta_{A,3},\theta_{A,4}\in\mathbb{R}, \delta=0,\\|\theta_{A,3}|+|\theta_{A,4}|>0, \sigma_1\neq0,\sigma_2\neq0\}$, then the corresponding operators are
          \begin{eqnarray*}
           (Ax)(t) &=&  \int\limits_{\alpha_1}^{\beta_1} I_{[\alpha,\beta]}(t)[\theta_{A,3} \sin(\omega t)\sin(\omega s)+\theta_{A,4}\cos(\omega t)\sin(\omega s)]x(s)d\mu_s, \\
           (Bx)(t) &=&  \int\limits_{\alpha_1}^{\beta_1} \theta_{B,2}I_{[\alpha,\beta]}(t) \cos(\omega t)\cos(\omega s)x(s)d\mu_s,
         \end{eqnarray*}
         for almost every $t$. These operators satisfy $AB=0$.

          \noindent\ref{solSystemEqWaveletExFourierGivenBfindAa1=0:item4}. If
          $\{\theta_{B,1}=\theta_{B,4}=0, \theta_{B,2}\in\mathbb{R}\setminus
           \{0\},\,\theta_{B,3}\in\mathbb{R}\setminus\left\{\theta_{B,2}\frac{\sigma_2}{\sigma_1}\right\},\, \theta_{A,1}= \theta_{A,2}=0, \theta_{A,4}=0,\\ \theta_{A,3}=\frac{1}{\delta\sigma_1}, \delta\not=0,\sigma_1\neq0,\sigma_2\neq0\}$, then the corresponding operators are
          \begin{eqnarray*}
           (Ax)(t) &=&  \int\limits_{\alpha_1}^{\beta_1}  I_{[\alpha,\beta]}(t)\frac{1}{\delta\sigma_1}\sin(\omega t)\sin(\omega s)x(s)d\mu_s, \\
           (Bx)(t) &=&  \int\limits_{\alpha_1}^{\beta_1} I_{[\alpha,\beta]}(t) \left(\theta_{B,2} \cos(\omega t)\cos(\omega s) +\theta_{B,3}\sin(\omega t)\sin(\omega s)\right)x(s)d\mu_s,
         \end{eqnarray*}
          for almost every $t$. These operators satisfy $AB=\delta BA^2=BA$.

          \noindent\ref{solSystemEqWaveletExFourierGivenBfindAa1=0:item5}. If
          $\{\theta_{B,1}=\theta_{B,4}=0, \theta_{B,2}\in\mathbb{R}\setminus
           \{0\},\,\theta_{B,3}=\theta_{B,2}\frac{\sigma_2}{\sigma_1},\, \theta_{A,1}= \theta_{A,2}=0,\\ \theta_{A,3}=\frac{1}{\delta\sigma_1}, \theta_{A,4}\in\mathbb{R},\, \delta\not=0, \sigma_1\neq0,\sigma_2\neq0\}$, then the corresponding operators are
           \begin{eqnarray*}
           (Ax)(t) &=&  \int\limits_{\alpha_1}^{\beta_1}  I_{[\alpha,\beta]}(t)\left(\frac{1}{\delta\sigma_1}\sin(\omega t)\sin(\omega s)+\theta_{A,4}\cos(\omega t)\sin(\omega s)\right) x(s)d\mu_s, \\
           (Bx)(t) &=&  \int\limits_{\alpha_1}^{\beta_1}I_{[\alpha,\beta]}(t) \left(\theta_{B,2} \cos(\omega t)\cos(\omega s) +\frac{\sigma_2\theta_{B,2}}{\sigma_1}\sin(\omega t)\sin(\omega s)\right)x(s)d\mu_s,
           \end{eqnarray*}
           for almost every $t$. These operators satisfy $AB=\delta BA^2=BA$.

           \noindent\ref{solSystemEqWaveletExFourierGivenBfindAa1=0:item6}. If
          $\{\theta_{B,1}=\theta_{B,4}=0, \theta_{B,2}\in\mathbb{R}\setminus
           \{0\},\,\theta_{B,3}\in\mathbb{R}\setminus\{\theta_{B,2}\frac{\sigma_2}{\sigma_1}\}, \theta_{A,1}= \theta_{A,3}=\theta_{A,4}=0,\\ \theta_{A,2}=\frac{1}{\delta\sigma_2}, \delta\not=0,\sigma_1\neq0,\sigma_2\neq0\}$,
           the corresponding operators are
\begin{eqnarray*}
           (Ax)(t) &=&  \int\limits_{\alpha_1}^{\beta_1}  I_{[\alpha,\beta]}(t)\frac{1}{\delta\sigma_2}\cos(\omega t)\cos(\omega s)x(s)d\mu_s, \\
           (Bx)(t) &=&  \int\limits_{\alpha_1}^{\beta_1} I_{[\alpha,\beta]}(t)\left(\theta_{B,2} \cos(\omega t)\cos(\omega s) +\theta_{B,3}\sin(\omega t)\sin(\omega s)\right)x(s)d\mu_s,
           \end{eqnarray*}
           for almost every $t$. These operators satisfy $AB=\delta BA^2=BA$.

           \noindent\ref{solSystemEqWaveletExFourierGivenBfindAa1=0:item7}. If
          $\{\theta_{B,1}=\theta_{B,4}=0, \theta_{B,2}\in\mathbb{R}\setminus
           \{0\},\theta_{B,3}=\theta_{B,2}\frac{\sigma_2}{\sigma_1}, \theta_{A,1}= \theta_{A,3}=0, \theta_{A,2}=\frac{1}{\delta\sigma_2}, \delta\not=0,\\ \theta_{A,4}\in\mathbb{R},  \sigma_1\neq0,\sigma_2\neq0\}$, then
           the corresponding operators are
           \begin{eqnarray*}
           (Ax)(t) &=&  \int\limits_{\alpha_1}^{\beta_1}  I_{[\alpha,\beta]}(t)\left(\frac{1}{\delta\sigma_2}\cos(\omega t)\cos(\omega s)+\theta_{A,4}\cos(\omega t)\sin(\omega s)\right)x(s)d\mu_s, \\
           (Bx)(t) &=&  \int\limits_{\alpha_1}^{\beta_1} I_{[\alpha,\beta]}(t)\left(\theta_{B,2} \cos(\omega t)\cos(\omega s) +\frac{\sigma_2\theta_{B,2}}{\sigma_1}\sin(\omega t)\sin(\omega s)\right)x(s)d\mu_s,
           \end{eqnarray*}
           for almost every $t$. These operators satisfy $AB=\delta BA^2=BA$.

           \noindent\ref{solSystemEqWaveletExFourierGivenBfindAa1=0:item8}. If
          $\{\theta_{B,1}=\theta_{B,4}=0, \theta_{B,2}\in\mathbb{R}\setminus
           \{0\}, \theta_{B,3}\in\mathbb{R}\setminus\{\theta_{B,2}\frac{2\sigma_2}{\sigma_1}\}, \theta_{A,1}= \theta_{A,4}=0, \theta_{A,2}=\frac{1}{\delta\sigma_2},\\ \theta_{A,3}=\frac{1}{\delta\sigma_1},\, \delta\not=0, \sigma_1\neq0,\sigma_2\neq0\}$, then
           the corresponding operators are
           \begin{eqnarray*}
           (Ax)(t) &=&  \int\limits_{\alpha_1}^{\beta_1}  I_{[\alpha,\beta]}(t)\left(\frac{1}{\delta\sigma_2}\cos(\omega t)\cos(\omega s)+\frac{1}{\delta \sigma_1}\sin(\omega t)\sin(\omega s)\right)x(s)d\mu_s, \\
           (Bx)(t) &=&  \int\limits_{\alpha_1}^{\beta_1} I_{[\alpha,\beta]}(t)\left(\theta_{B,2} \cos(\omega t)\cos(\omega s) +\theta_{B,3}\sin(\omega t)\sin(\omega s)\right)x(s)d\mu_s,
           \end{eqnarray*}
           for almost every $t$. These operators satisfy $AB=\delta BA^2=BA$.

           \noindent\ref{solSystemEqWaveletExFourierGivenBfindAa1=0:item9}. If
          $\{\theta_{B,1}=\theta_{B,4}=0, \theta_{B,2}\in\mathbb{R}\setminus
           \{0\},\,\theta_{B,3}=\theta_{B,2}\frac{2\sigma_2}{\sigma_1},\, \theta_{A,1}= 0, \theta_{A,2}=\frac{1}{\delta\sigma_2}, \\ \theta_{A,3}=\frac{1}{\delta\sigma_1},\, \theta_{A,4}\in\mathbb{R}, \, \delta\not=0, \sigma_1\neq 0, \sigma_2\neq0\}$, then
           the corresponding operators are
           \begin{eqnarray*}
           (Ax)(t) &=&  \int\limits_{\alpha_1}^{\beta_1}  I_{[\alpha,\beta]}(t)\left(\frac{1}{\delta\sigma_2}\cos(\omega t)\cos(\omega s)+\frac{1}{\delta\sigma_1}\sin(\omega t)\sin(\omega s)\right. \\
           & & \left. +\theta_{A,4}\cos(\omega t)\sin(\omega s)\right)x(s)d\mu_s, \\
           (Bx)(t) &=&  \int\limits_{\alpha_1}^{\beta_1}  I_{[\alpha,\beta]}(t)\left(\theta_{B,2}\cos(\omega t)\cos(\omega s) +\frac{2\sigma_2\theta_{B,2}}{\sigma_1}\sin(\omega t)\sin(\omega s)\right)x(s)d\mu_s,
           \end{eqnarray*}
          for almost every $t$. They satisfy $AB=\delta BA^2$. The operator $AB-BA$ was computed in Case 2B, item \ref{ExampleFourierWaveletOpGivenBCase2BfindB:item8}.

           \noindent\ref{solSystemEqWaveletExFourierGivenBfindAa1=0:item10}. If
          $\{\theta_{B,1}=\theta_{B,2}=\theta_{B,4}=0, \theta_{B,3}\in\mathbb{R}\setminus
           \{0\}, \theta_{A,1}= \theta_{A,3}=\theta_{A,4}=0, \theta_{A,2}\in\mathbb{R}\setminus\{0\},\\  \delta\not=0, \sigma_1\neq0,\sigma_2\neq0\}$, then
           the corresponding operators are
           \begin{eqnarray*}
           (Ax)(t) &=&  \int\limits_{\alpha_1}^{\beta_1}  I_{[\alpha,\beta]}(t)\theta_{A,2}\cos(\omega t)\cos(\omega s)x(s)d\mu_s, \\
           (Bx)(t) &=&  \int\limits_{\alpha_1}^{\beta_1}  I_{[\alpha,\beta]}(t) \theta_{B,3}\sin(\omega t)\sin(\omega s)x(s)d\mu_s,
           \end{eqnarray*}
           for almost every $t$. These operators satisfy $AB=\delta BA^2=BA=0$.

           \noindent\ref{solSystemEqWaveletExFourierGivenBfindAa1=0:item11}. If
          $\{\theta_{B,1}=\theta_{B,2}=\theta_{B,4}=0, \theta_{B,3}\in\mathbb{R}\setminus
           \{0\}, \theta_{A,1}= \theta_{A,4}=0, \theta_{A,3}=\frac{1}{\delta\sigma_1} \theta_{A,2}\in\mathbb{R}, \\ \delta\not=0, \sigma_1\neq0,\sigma_2\neq0\}$, then
           the corresponding operators are
           \begin{eqnarray*}
           (Ax)(t) &=&  \int\limits_{\alpha_1}^{\beta_1}  I_{[\alpha,\beta]}(t)\left(\theta_{A,2}\cos(\omega t)\cos(\omega s)+\frac{1}{\delta\sigma_1}\sin(\omega t)\sin(\omega s) \right)x(s)d\mu_s, \\
           (Bx)(t) &=&  \int\limits_{\alpha_1}^{\beta_1}  I_{[\alpha,\beta]}(t) \theta_{B,3}\sin(\omega t)\sin(\omega s)x(s)d\mu_s,
           \end{eqnarray*}
           for almost every $t$. These operators satisfy $AB=\delta BA^2=BA$.

           \noindent\ref{solSystemEqWaveletExFourierGivenBfindAa1=0:item12}. If
           $\{\theta_{B,1}=\theta_{B,3}=\theta_{B,4}=0, \theta_{B,2}\in\mathbb{R}\setminus
           \{0\}, \theta_{A,1}= \theta_{A,2}=\theta_{A,4}=0, \theta_{A,3}\in\mathbb{R}\setminus\{0\},\\ \delta\not=0, \sigma_1\neq0,\sigma_2\neq0\}$, then
           the corresponding operators are
           \begin{eqnarray*}
           (Ax)(t) &=&  \int\limits_{\alpha_1}^{\beta_1}  I_{[\alpha,\beta]}(t)\theta_{A,3}\sin(\omega t)\sin(\omega s)x(s)d\mu_s, \\
           (Bx)(t) &=&  \int\limits_{\alpha_1}^{\beta_1}  I_{[\alpha,\beta]}(t) \theta_{B,2}\cos(\omega t)\cos(\omega s)x(s)d\mu_s,
           \end{eqnarray*}
          for almost every $t$. These operators satisfy $AB=\delta BA^2=BA=0$.

           \noindent\ref{solSystemEqWaveletExFourierGivenBfindAa1=0:item13}. If
           $\{\theta_{B,1}=\theta_{B,3}=\theta_{B,4}=0, \theta_{B,2}\in\mathbb{R}\setminus
           \{0\},\, \theta_{A,1}=\theta_{A,2}= \theta_{A,3}=0,\, \\ \theta_{A,4}\in\mathbb{R}, \, \delta\not=0, \sigma_1\neq0,\sigma_2\neq0\}$, then
           the corresponding operators are
           \begin{eqnarray*}
           (Ax)(t) &=&  \int\limits_{\alpha_1}^{\beta_1}  I_{[\alpha,\beta]}(t)\theta_{A,4}\cos(\omega t)\sin(\omega s)x(s)d\mu_s, \\
           (Bx)(t) &=&  \int\limits_{\alpha_1}^{\beta_1}  I_{[\alpha,\beta]}(t) \theta_{B,2}\cos(\omega t)\cos(\omega s)x(s)d\mu_s,
           \end{eqnarray*}
           for almost every $t$. They satisfy $AB=\delta BA^2=0$. Moreover, for all $x\in L_p(\mathbb{R},\mu)$, $1< p< \infty$ we have, for almost every $t$,
       \begin{equation*}
         (AB-BA)x(t)=-(BA)x(t)=-\theta_{A,4}\sigma_2 \theta_{B,2}\int\limits_{\alpha_1}^{\beta_1} I_{[\alpha,\beta]}(t)\cos(\omega t)\sin(\omega s)x(s)d\mu_s.
       \end{equation*}
         By applying Proposition \ref{PropositionCommutatorZeroFourSequenceCond}, since $\alpha_1<\beta_1$, $[\alpha,\beta]\supseteq [\alpha_1,\beta_1]$ and $\sigma_2\neq0$, operators $A$ and $B$ commute  if and only if either $\theta_{A,4}=0$ ($A=0$) or $\theta_{B,2}=0$ ($B=0$).

           \noindent\ref{solSystemEqWaveletExFourierGivenBfindAa1=0:item14}. If
          $\{\theta_{B,1}=\theta_{B,3}=\theta_{B,4}=0, \theta_{B,2}\in\mathbb{R}\setminus
           \{0\}, \theta_{A,1}= \theta_{A,4}=0, \theta_{A,2}=\frac{1}{\delta\sigma_1}, \theta_{A,3}\in\mathbb{R},\\ \delta\not=0, \sigma_1\neq0,\sigma_2\neq0\}$, then
           the corresponding operators are
           \begin{eqnarray*}
           (Ax)(t) &=&  \int\limits_{\alpha_1}^{\beta_1}  I_{[\alpha,\beta]}(t)\left(\frac{1}{\delta\sigma_2}\cos(\omega t)\cos(\omega s)+\theta_{A,3}\sin(\omega t)\sin(\omega s)\right)x(s)d\mu_s, \\
           (Bx)(t) &=&  \int\limits_{\alpha_1}^{\beta_1}  I_{[\alpha,\beta]}(t) \theta_{B,2}\cos(\omega t)\cos(\omega s)x(s)d\mu_s,
           \end{eqnarray*}
           for almost every $t$. These operators satisfy $AB=\delta BA^2=BA$.

           \noindent\ref{solSystemEqWaveletExFourierGivenBfindAa1=0:item15}. If
          $\{\theta_{B,1}=\theta_{B,3}=\theta_{B,4}=0, \theta_{B,2}\in\mathbb{R}\setminus
           \{0\},\, \theta_{A,1}= 0,\, \theta_{A,2}=\frac{1}{\delta\sigma_2}, \theta_{A,3}=-\frac{1}{\delta \sigma_1}, \delta\not=0,\\ \theta_{A,4}\in\mathbb{R}, \sigma_1\neq0,\sigma_2\neq0\}$, then
           the corresponding operators are
           \begin{eqnarray*}
           (Ax)(t) &=&  \int\limits_{\alpha_1}^{\beta_1}  I_{[\alpha,\beta]}(t)\left(\frac{1}{\delta\sigma_2}\cos(\omega t)\cos(\omega s)-\frac{1}{\delta\sigma_1}\sin(\omega t)\sin(\omega s)\right.
           \\
           &&\left. +\theta_{A,4}\cos(\omega t)\sin(\omega s)\right)x(s)d\mu_s, \\
           (Bx)(t) &=&  \int\limits_{\alpha_1}^{\beta_1}  I_{[\alpha,\beta]}(t) \theta_{B,2}\cos(\omega t)\cos(\omega s)x(s)d\mu_s,
           \end{eqnarray*}
           for almost every $t$. These operators satisfy $AB=\delta BA^2$. Moreover, for all $x\in L_p(\mathbb{R},\mu)$, $1< p< \infty$ we have
       \begin{equation*}
         (AB-BA)x(t)=\theta_{A,4}\sigma_2 \theta_{B,2}\int\limits_{\alpha_1}^{\beta_1} I_{[\alpha,\beta]}(t)\cos(\omega t)\sin(\omega s)x(s)d\mu_s,
       \end{equation*}
       for almost every $t$.  By applying Proposition \ref{PropositionCommutatorZeroFourSequenceCond}, since $\alpha_1<\beta_1$, $[\alpha,\beta]\supseteq [\alpha_1,\beta_1]$ and $\sigma_2\neq0$, operators $A$ and $B$ commute  if and only if either $\theta_{A,4}=0$ or $\theta_{B,2}=0$ ($B=0$).
      }
   \end{example}

\begin{example}
{\rm
Let $(\mathbb{R}, \Sigma,\mu)$ be the standard Lebesgue measure space.  %Consider  integral operators acting on $L_p(\mathbb{R},\mu)$ for $1< p < \infty$.
Let operators $$A:L_p(\mathbb{R},\mu)\to L_p(\mathbb{R},\mu),\quad B:L_p(\mathbb{R},\mu)\to L_p(\mathbb{R},\mu),\ 1< p<\infty$$ be defined as follows
\begin{eqnarray*}
 \textstyle (Ax)(t)&=& \int\limits_{\alpha_1}^{\beta_1} \left(\sum_{i=1}^{l_{A,1}}\sum_{j=1}^{l_{A,2}} \theta_{i,j}^A I_{[\alpha,\beta]}(t)\sin(\frac{2\pi m_i }{\lambda_t}t)\cos(\frac{2\pi k_j }{\lambda_s}s)
 \right.\\
 & &+\sum_{i=1}^{l_{A,1}}\sum_{j=1}^{l_{A,2}} \pi_{i,j}^A I_{[\alpha,\beta]}(t)\cos(\frac{2\pi m_i }{\lambda_t}t)\cos(\frac{2\pi k_j }{\lambda_s}s)\\
 & &+\sum_{i=1}^{l_{A,1}}\sum_{j=1}^{l_{A,2}} \nu_{i,j}^A I_{[\alpha,\beta]}(t)\sin(\frac{2\pi m_i }{\lambda_t}t)\sin(\frac{2\pi k_j }{\lambda_s}s)
 \\
  & & \left.+\sum_{i=1}^{l_{A,1}}\sum_{j=1}^{l_{A,2}} \varsigma_{i,j}^A I_{[\alpha,\beta]}(t)\cos(\frac{2\pi m_i }{\lambda_t}t)\sin(\frac{2\pi k_j }{\lambda_s}s)
 \right)x(s)d\mu_s,
\\
  (Bx)(t)&=& \int\limits_{\alpha_1}^{\beta_1} \left(\sum_{i=1}^{l_{A,1}}\sum_{j=1}^{l_{A,2}} \theta_{i,j}^B I_{[\alpha,\beta]}(t)\sin(\frac{2\pi m_i }{\lambda_t}t)\cos(\frac{2\pi k_j }{\lambda_s}s)
 \right.\\
 & &+\sum_{i=1}^{l_{A,1}}\sum_{j=1}^{l_{A,2}} \pi_{i,j}^B I_{[\alpha,\beta]}(t)\cos(\frac{2\pi m_i }{\lambda_t}t)\cos(\frac{2\pi k_j }{\lambda_s}s)\\
 & &+\sum_{i=1}^{l_{A,1}}\sum_{j=1}^{l_{A,2}} \nu_{i,j}^B I_{[\alpha,\beta]}(t)\sin(\frac{2\pi m_i }{\lambda_t}t)\sin(\frac{2\pi k_j }{\lambda_s}s)
 \\
  & & \left.+\sum_{i=1}^{l_{A,1}}\sum_{j=1}^{l_{A,2}} \varsigma_{i,j}^B I_{[\alpha,\beta]}(t)\cos(\frac{2\pi m_i }{\lambda_t}t)\sin(\frac{2\pi k_j }{\lambda_s}s)
 \right)x(s)d\mu_s,
\end{eqnarray*}
for almost every $t$, where the index in  $\mu$ indicates the variable of integration,
 $\alpha,\, \alpha_1,\, \beta,\, \beta_1$ are real constants such that $\alpha_1<\beta_1$, $\alpha\le \alpha_1,$ $\beta\ge \beta_1$ and $I_{E}(t)$ is the indicator function of the set $E$, $\lambda_t,\lambda_s\in \mathbb{R}\setminus \{0\}$, $l_{A,1}$, $l_{A,2}$ are positive integers, $m_i, k_j,\, \theta_{i,j}^A,\, \theta_{i,j}^B,\, \pi_{i,j}^A,\, \pi_{i,j}^B,\, \nu_{i,j}^A,\, \nu_{i,j}^B,\, \varsigma_{i,j}^A,\, \varsigma_{i,j}^B \in \mathbb{R}$, $1\le i\le l_{A,1}$, $1\le j\le l_{A,2}$. We put
\begin{gather*}
a_{i,l}(t)=b_{i,l}(t)=
\left\{\begin{array}{ll}
I_{[\alpha,\beta]}(t)\sin\left( \frac{2\pi m_i t}{\lambda_t} \right), \ & l=1 \, \mbox{ or }l=3 \\
I_{[\alpha,\beta]}(t)\cos\left( \frac{2\pi m_i t}{\lambda_t} \right), \ & l=2 \, \mbox{ or }l=4
\end{array}\right.\\
\hspace{-1.4cm}
c_{i,l}(s)=e_{i,l}(s)=
\left\{\begin{array}{ll}
\cos\left( \frac{2\pi k_i s}{\lambda_s} \right), \ & l=1 \, \mbox{ or }l=2 \\
\sin\left( \frac{2\pi k_i s}{\lambda_s} \right), \ & l=3 \, \mbox{ or }l=4
\end{array}\right.
\end{gather*}
$a_{i,l},\, b_{i,l}\in L_p(\mathbb{R},\mu)$, $1<p<\infty$, $1\le i\le l_{A,1}$, $1\le l\le 4$,  $c_{i,l},e_{i,l}\in L_q[\alpha_1,\beta_1]$, where $1<q<\infty$ such that $\frac{1}{p}+\frac{1}{q}=1$, $1\le i\le l_{A,1}$, $1\le l\le 4$.

 These operators are well defined with similar argument as in Example \ref{ExampleRepintopgensepkernelfouriersequence}  and by the fact $L_p(\mathbb{R},\mu)$ is a vector space. According to Theorem \ref{thmBothIntOPSupGenSeptedKernels} $AB=\delta BA^2$, for some $\delta \in \mathbb{R}$, $G_A=G_B=[\alpha_1,\beta_1]$, if and only if for almost every $(t,s)\in\mathbb{R}\times[\alpha_1,\beta_1]$
 \begin{gather*}
   \sum_{(i,j,l=1)}^{(l_{A,1}, l_{A,2},l_{A,3}) } \sum_{(u,w,r=1)}^{(l_{A,1}, l_{A,2},l_{A,3}) } \gamma_A^{i,j,l}\gamma_B^{u,w,r} a_{i,l}(t)Q_{G_A}(c_{j,l};b_{u,r})e_{w,r}(s)\\
   =\sum_{(u,w,r=1)}^{(l_{A,1}, l_{A,2},l_{A,3}) } \sum_{(i_1,j_1,l_1=1)}^{(l_{A,1}, l_{A,2},l_{A,3}) } \sum_{(i_2,j_2,l_2=1)}^{(l_{A,1}, l_{A,2},l_{A,3}) } \delta Q_{G_A}(a_{i_2,l_2};c_{j_1,l_1})\gamma_B^{u,w,r}
   \gamma_A^{i_1,j_1,l_1}\gamma_A^{i_2,j_2,l_2} \\
   \cdot Q_{G_A}(e_{w,r};a_{i_1,l_1})b_{u,r}(t)c_{i_2,l_2}(s),
 \end{gather*}
 where $l_{A,3}=4$, $Q_\Lambda(\cdot , \cdot)$ is defined in \eqref{QGpairingDefinition},
 \begin{gather*}
   \gamma_A^{i,j,l}=\left\{\begin{array}{cc}
                             \theta_{i,j}^A, & \,\mbox{ if } l=1 \\
                             \pi_{i,j}^A, & \, \mbox{ if } l=2 \\
                             \nu_{i,j}^A, & \,\mbox{ if } l=3\\
                             \varsigma_{i,j}^A, & \,\mbox{ if } l=4
                           \end{array} \right.
                   \         \mbox{ and } \
   \gamma_B^{i,j,l}=\left\{\begin{array}{cc}
                             \theta_{i,j}^B, & \,\mbox{ if } l=1 \\
                             \pi_{i,j}^B, & \,\mbox{ if } l=2 \\
                             \nu_{i,j}^B, & \,\mbox{ if } l=3\\
                             \varsigma_{i,j}^B, & \,\mbox{ if } l=4.
                           \end{array} \right.
 \end{gather*}
}
\end{example}

\begin{example}\label{ConstructExampleTheoremIntOpRepGenKernLp}
{\rm
Let $(\mathbb{R}, \Sigma,\mu)$ be the standard Lebesgue measure space. % Consider  integral operators acting on $L_p(\mathbb{R},\mu)$ for $1< p<\infty $.
Let $$A:L_p(\mathbb{R},\mu)\to L_p(\mathbb{R},\mu),\quad  B:L_p(\mathbb{R},\mu)\to L_p(\mathbb{R},\mu),\ 1< p<\infty$$ be defined as follows
\[
 (Ax)(t)= \int\limits_0^1 \sum\limits_{m=1}^2 a_m(t)c_m(s)x(s)d\mu_s,\quad (Bx)(t)= \int\limits_0^1 \sum\limits_{k=1}^2 b_k(t)e_k(s)x(s)d\mu_s,
\]
for almost every $t$, where the index in  $\mu$ indicates the variable of integration,
$a_m,b_k\in L_p(\mathbb{R},\mu)$, $e_m,c_k\in L_q([0,1],\mu)$, $k,m=1,2$, such that
\begin{gather*}
Q_{[0,1]}(e_i,a_j)=0, \ i\not=j, \ Q_{[0,1]}(c_i,a_j)=0,\ i\not=j,\, Q_{[0,1]}(b_i,c_j)=0, \ i\not=j,\\
a_m(t)=\alpha_m b_m(t),\ c_k(s)=\pi_k e_k(s),\ k,m=1,2.
\end{gather*}
where $ \delta\in \mathbb{R}\setminus\{0\} $, $\alpha_m, \, \pi_m\in\mathbb{R}$, $m=1,2$. If either $\alpha_m=0$ or  $\pi_m= 0$ or  $Q_{[0,1]}(b_m,e_m)=0$ or $\delta \alpha_m \pi_m Q_{[0,1]}(b_m,e_m)=1,\, m=1,2,$
then operators $A$ and $B$ satisfy $AB=\delta BA^2$. In fact, by applying Corollary  \ref{corDiedroRelBothIntOPGenSeptedKernels} we have by applying bilinearity of $Q_{\Lambda}(\cdot,\cdot)$
\begin{eqnarray*}
&&\hspace{-1cm} \sum\limits_{k=1}^2 \sum\limits_{m=1}^2 a_m(t)Q_{[0,1]}(b_k,c_m)e_k(s) \\
&=&  \sum\limits_{m=1}^2 [a_m(t)Q_{[0,1]}(b_1,c_m)e_1(s) + a_m(t)Q_{[0,1]}(b_2,c_m)e_2(s)] \\
&=& Q_{[0,1]}(b_1,c_1) a_1(t)e_1(s)+Q_{[0,1]}(b_2,c_2)a_2(t)e_2(s)\\&=& Q_{[0,1]}(b_1,\pi_1 e_1) a_1(t)e_1(s)+Q_{[0,1]}(b_2,\pi_2 e_2)a_2(t)e_2(s)\\
&=& {\pi_1}Q_{[0,1]}(b_1, e_1)  a_1(t)e_1(s)+{\pi_2}Q_{[0,1]}(b_2, e_2) a_2(t)e_2(s)\\
&=& {\pi_1}\alpha_1 Q_{[0,1]}(b_1, e_1)  b_1(t)e_1(s)+{\pi_2}\alpha_2 Q_{[0,1]}(b_2, e_2) b_2(t)e_2(s)
\end{eqnarray*}
for almost every $(t,s)\in \mathbb{R}\times[0,1]$.
On the other hand,
\begin{align*}
 & \textstyle \sum\limits_{k=1}^2 \sum\limits_{i_1,i_2=1}^2 \delta b_k(t)Q_{[0,1]}(e_1,a_{i_1})\gamma_{i_1,i_2}c_{i_2}(s) \\
& \ \textstyle =\sum\limits_{i_1=1}^2 \sum\limits_{i_2=1}^2  \delta[b_1(t)Q_{[0,1]}(e_1,a_{i_1})\gamma_{i_1,i_2}c_{i_2}(s)+ b_2(t)Q_{[0,1]}(e_2,a_{i_1})\gamma_{i_1,i_2}c_{i_2}(s)]\\
 & \ \textstyle  = \sum\limits_{i_2=1}^2 \delta[b_1(t)Q_{[0,1]}(e_1,a_1)\gamma_{1,i_2}c_{i_2}(s)+b_2(t)Q_{[0,1]}(e_2,a_1)\gamma_{1,i_2}c_{i_2}(s)]\\
 &\quad \textstyle +\sum\limits_{i_2=1}^2 \delta[b_1(t)Q_{[0,1]}(e_1,a_2)\gamma_{2,i_2}c_{i_2}(s)+b_2(t)Q_{[0,1]}(e_2,a_2)\gamma_{2,i_2}c_{i_2}(s)]\\
 & \ \textstyle = \delta b_1(t)Q_{[0,1]}(e_1,a_1)Q_{[0,1]}(a_1,c_1)c_1(s)\\
 &\hspace{4cm} + \delta b_2(t)Q_{[0,1]}(e_2,a_2)Q_{[0,1]}(a_2,c_2)c_2(s)\\
 & \ \textstyle = \delta \alpha_1^2 \pi_1^2 b_1(t)Q_{[0,1]}(e_1, b_1)^2e_1(s)+
 \delta \alpha_2^2\pi_2^2 b_2(t)Q_{[0,1]}(e_2, b_2)^2 e_2(s)
 \end{align*}
for almost every $(t,s)\in \mathbb{R}\times [0,1]$.  If either $\alpha_m=0$ or  $\pi_m= 0 $ or  $Q_{[0,1]}(b_m,e_m)=0$ or $ \delta \alpha_m \pi_m Q_{[0,1]}(b_m,e_m)=1,\, m=1,2,$  then we have
\begin{gather}\nonumber
{\pi_1}\alpha_1 Q_{[0,1]}(b_1, e_1)  b_1(t)e_1(s)+{\pi_2}\alpha_2 Q_{[0,1]}(b_2, e_2) b_2(t)e_2(s)
\\ \label{ConstructExamplesGenSeparatedKernelsEqConclu}
=\delta \alpha_1^2 \pi_1^2 b_1(t)Q_{[0,1]}(e_1, b_1)^2e_1(s)+
 \delta \alpha_2^2\pi_2^2 b_2(t)Q_{[0,1]}(e_2, b_2)^2 e_2(s).
\end{gather}
In fact, if either $\alpha_1=0$ or $\pi_1=0$ or $Q_{[0,1]}(e_1, b_1)=0$ then equation \eqref{ConstructExamplesGenSeparatedKernelsEqConclu} becomes
\begin{equation*}
{\pi_2}\alpha_2 Q_{[0,1]}(b_2, e_2) b_2(t)e_2(s)
= \delta \alpha_2^2\pi_2^2 b_2(t)Q_{[0,1]}(e_2, b_2)^2 e_2(s).
\end{equation*}
which is equivalent to either  $\alpha_2=0$ or $\pi_2=0$ or $Q_{[0,1]}(e_2, b_2)=0$ or  \\
$ \delta \alpha_2 \pi_2 Q_{[0,1]}(b_2,e_2)=1$. Similarly if either $\alpha_2=0$ or $\pi_2=0$ or $Q_{[0,1]}(e_2, b_2)=0$ then equation \eqref{ConstructExamplesGenSeparatedKernelsEqConclu} is equivalent to either  $\alpha_1=0$ or $\pi_1=0$ or $Q_{[0,1]}(e_1, b_1)=0$ or
$ \delta \alpha_1 \pi_1 Q_{[0,1]}(b_1,e_1)=1$.
If $ \delta \alpha_m \pi_m Q_{[0,1]}(b_m,e_m)=1$ for $m=1,2$ then by direct computation we conclude that equation \eqref{ConstructExamplesGenSeparatedKernelsEqConclu} is fulfilled.
Therefore, by applying Corollary \ref{corDiedroRelBothIntOPGenSeptedKernels} we have $AB=\delta BA^2$.
In particular, consider the following operators
$A:L_p(\mathbb{R},\mu)\to L_p(\mathbb{R},\mu)$,  $B:L_p(\mathbb{R},\mu)\to L_p(\mathbb{R},\mu)$, $1< p<\infty$  defined as follows
\[
 \textstyle (Ax)(t)= \int\limits_0^1 \sum\limits_{m=1}^2 a_m(t)c_m(s)x(s)d\mu_s\ ,\quad (Bx)(t)= \int\limits_0^1 \sum\limits_{k=1}^2 b_k(t)e_k(s)x(s)d\mu_s,
\]
for almost every $t$, where the index in  $\mu$ indicates the variable of integration,
\begin{eqnarray*}
\sum\limits_{m=1}^2 a_m(t)c_m(s)&=& I_{[\alpha,\beta]}(t)(-6 t(4s-3)+12 t^2(3s-2)), \\[2mm]
\sum\limits_{k=1}^2 b_k(t)e_k(s)&=& I_{[\alpha,\beta]}(t)( 6t(4s-3)-12t^2(3s-2)),
\end{eqnarray*}
for almost every  $ (t,s)\in \mathbb{R}\times[0,1]$, $\alpha,\,\beta$ are real constants such that $\alpha\le 0,$ $\beta\ge 1$ and $I_{E}(t)$ is the indicator function of the set $E$.
So we have
\begin{gather*}
\begin{array}{ll}
a_1(t)=I_{[\alpha,\beta]}(t)(-6t) \ , \quad & a_2(t)=I_{[\alpha,\beta]}(t) 12t^2  \\
c_1(s)= 4s-3 \ , & c_2(s)= 3s-2 \\
b_1(t)=I_{[\alpha,\beta]}(t) 6t  \ , & b_2(t)=I_{[\alpha,\beta]}(t) (-12)t^2    \\
e_1(s)= 4s-3  \ , & e_2(s)= 3s-2.
\end{array}
\end{gather*}
 These operators are well defined and satisfy $AB=BA^2$. In fact,by applying Corollary \ref{corDiedroRelBothIntOPGenSeptedKernels} when  $\delta=1$, $n=2$, $G_A=G_B=[0,1]$  we have
\begin{equation*}
\resizebox{1\hsize}{!}{$\displaystyle
\begin{array}{ll}
Q_{[0,1]}(e_1,a_1)=\int\limits_0^1 -6s(4s-3) d\mu=1\ , \quad & Q_{[0,1]}(e_1,a_2)=\int\limits_0^1 (4s-3)12s^2 d\mu=0,\\
Q_{[0,1]}(e_2,a_1)=\int\limits_0^1 -6s(3s-2)d\mu=0\ , & Q_{[0,1]}(e_2,a_2)=\int\limits_0^1 12s^2(3s-2) d\mu=1,\\
\gamma_{1,1}=Q_{[0,1]}(a_1,c_1)=\int\limits_0^1 -6s(4s-3)d\mu=1\ , & \gamma_{2,1}=Q_{[0,1]}(a_1,c_2)=\int\limits_0^1 -6s(3s-2) d\mu=0,\\
\gamma_{1,2}=Q_{[0,1]}(a_2,c_1)=\int\limits_0^1 12s^2(4s-3) d\mu=0\ , & \gamma_{2,2}=Q_{[0,1]}(a_2,c_2)=\int\limits_0^1 -12s^2(3s-2) d\mu=1.
\end{array}$}
\end{equation*}
Therefore, we have
\begin{align*}
&  \sum_{k=1}^2 \sum_{m=1}^2 a_m(t)Q_{G_A}(b_k,c_m)e_k(s) \\
& =  \sum_{m=1}^2 [a_m(t)Q_{G_A}(b_1,c_m)e_1(s) + a_m(t)Q_{G_A}(b_2,c_m)e_2(s)] \\
& =  a_1(t)e_1(s)+a_2(t)e_2(s) =\ I_{[\alpha,\beta]}(t)(6t(4s-3)-12t^2(3s-2))
  \end{align*}
   for almost every $(t,s)\in \mathbb{R}\times [0,1]$. On the other hand,
\begin{align*}
 & \textstyle \sum\limits_{k=1}^2 \sum\limits_{i_1,i_2=1}^2 b_k(t)Q_{G_B}(e_1,a_{i_1})\gamma_{i_1,i_2}c_{i_2}(s) \\
& \quad \textstyle =
 b_1(t)Q_{G_B}(e_1,a_1)\gamma_{1,1}c_1(s)+b_2(t)Q_{G_B}(e_2,a_2)\gamma_{2,2}c_2(s)\\
 & \quad \textstyle
  =I_{[\alpha,\beta]}(t)(6t(4s-3)-12t^2(3s-2))
 \end{align*}
for almost every $(t,s)\in \mathbb{R}\times [0,1]$.
Therefore, conditions of
Corollary \ref{corDiedroRelBothIntOPGenSeptedKernels} are fulfilled and so $AB=BA^2$. Moreover,
$AB=BA^2=-A$ and $A^2=A$, that is, $A$ and $B$ commute.
}\end{example}

\begin{example}[Laurent polynomials]\label{ExampleTheoremIntOpRepGenKernLaurenPolynomial}
{\rm
Let $(\mathbb{R}, \Sigma,\mu)$ be the standard Lebesgue measure space.  %Consider  integral operators acting on $L_p(\mathbb{R},\mu)$ for $1< p < \infty$.
Let $A:L_p(\mathbb{R},\mu)\to L_p(\mathbb{R},\mu)$,  $B:L_p(\mathbb{R},\mu)\to L_p(\mathbb{R},\mu)$, $1< p<\infty$  defined as follows
\[
 \textstyle (Ax)(t)= \int\limits_1^2 \sum\limits_{m=1}^4 a_m(t)c_m(s)x(s)d\mu_s \ ,\quad (Bx)(t)= \int\limits_1^2 \sum\limits_{k=1}^2 b_k(t)e_k(s)x(s)d\mu_s,
\]
for almost every $t$, where the index in  $\mu$ indicates the variable of integration,
\[\begin{array}{lll}
\sum\limits_{m=1}^4 a_m(t)c_m(s)&=& I_{[\alpha,\beta]}(t)(\gamma_{A,0}+\gamma_{A,1}\frac{1}{s} +\gamma_{A,2}\frac{1}{t}+\gamma_{A,3}\frac{1}{ts}), \\[2mm]
\sum\limits_{k=1}^2 b_k(t)e_k(s)&=& I_{[\alpha,\beta]}(t)(\gamma_{B,1}\frac{1}{s}+\gamma_{B,2}\frac{1}{t}),
\end{array}
\]
for almost every  $ (t,s)\in \mathbb{R}\times[1,2]$, $\alpha,\,\beta$ are real constants such that $ 0<\alpha\le 1,$ $\beta\ge 2$, $I_{E}(t)$ is the indicator function of the set $E$, $\gamma_{A,m}, \gamma_{B,k}\in \mathbb{R}$, $m=1,2,3,4$, $k=1,2$. So we have
\begin{gather*}
\begin{array}{ll}
a_1(t)=I_{[\alpha,\beta]}(t)\gamma_{A,0} \ , \quad & a_2(t)=I_{[\alpha,\beta]}(t) \gamma_{A,1}\\ a_3(t)=I_{[\alpha,\beta]}(t)\gamma_{A,2}\frac{1}{t} \ , & a_4(t)=I_{[\alpha,\beta]}(t)\gamma_{A,3}\frac{1}{t}\\
c_1(s)=c_3(s)= 1 & c_2(s)=c_4(s)=\frac{1}{s} \\
b_1(t)=I_{[\alpha,\beta]}(t) \gamma_{B,1} \ , & b_2(t)=I_{[\alpha,\beta]}(t) \gamma_{B,2}\frac{1}{t}     \\
e_1(s)= \frac{1}{s} \ ,   & e_2(s)= 1.
\end{array}
\end{gather*}
These operators are well defined, indeed  for any fixed $p$, $1<p<\infty$, functions $a_i,b_j\in L_p(\mathbb{R},\mu)$, $c_i,e_j\in L_q([1,2],\mu)$, $i=1,2,3,4$, $j=1,2$, $q>1$, $\frac{1}{p}+\frac{1}{q}=1$. In fact,
 \begin{equation*}
  \int\limits_{\mathbb{R}} |a_i (s)|^p d\mu=\int\limits_{\alpha}^\beta |{a}_i (s)|^p d\mu<\infty,\ i=1,2,3,4.
\end{equation*}
 Similarly,
$
 \displaystyle{\int\limits_{\mathbb{R}}} |b_j (s)|^p d\mu=\displaystyle{\int\limits_{\alpha}^\beta} |\tilde{b}_j (s)|^p d\mu<\infty,\ j=1,2.
$
Analogously, we have
 \begin{equation*}
  \int\limits_{1}^2 |{c}_i (s)|^q d\mu<\infty,\ i=1,2,3, \mbox{ and } \int\limits_{1}^2 |{e}_j (s)|^q d\mu<\infty,\ j=1,2.
\end{equation*}
So we consider a monomial $F(z)=\delta z^2$, $\delta\in\mathbb{R}$, $G_A=G_B=[1,2]$ and by Corollary \ref{corDiedroRelBothIntOPGenSeptedKernels} we have $AB=\delta BA^2$  if and only if
\begin{equation*}
\sum_{k=1}^2\sum_{m=1}^4 a_m(t)Q_{G_A}(b_k,c_m)e_k(s)=\sum_{k=1}^2 \sum_{i_1,i_2=1}^4 \delta b_k(t)Q_{G_B}(e_k,a_{i_1})c_{i_2}(s)\gamma_{i_1,i_2}.
\end{equation*}
Therefore we get the following system of equations:
\begin{equation*}
\resizebox{1\hsize}{!}{$\displaystyle
\left\{\begin{array}{ll}
\gamma_{A,0}\gamma_{B,2}\ln 2+\gamma_{A,1}\gamma_{B,2}\frac{1}{2}&=\delta \gamma_{B,1}\gamma^2_{A,0}\ln 2+\delta \gamma_{B,1}\gamma_{A,0}\gamma_{A,2}(\ln 2)^2\\
& +\delta\gamma_{B,1}\gamma_{A,1}\gamma_{A,0}(\ln 2)^2
 + \delta \gamma_{B,1}\gamma_{A,1}\gamma_{A,2}(\ln 2)^2\\
 &+\delta \gamma_{B,1}\gamma_{A,2}\gamma_{A,0}\frac{1}{2}+\delta \gamma_{B,1}\gamma^2_{A,2}\frac{\ln 2}{2}\\
  & + \delta \gamma_{B,1}\gamma_{A,3}\gamma_{A,0}\frac{\ln 2}{2} +\delta \gamma_{B,1}\gamma_{A,3}\gamma_{A,2}\frac{1}{4}  \\
  \gamma_{A,0}+\gamma_{A,1}\gamma_{B,1}\ln 2 &= \delta \gamma_{B,1}\gamma_{A,0}\gamma_{A,1}\ln 2+\delta \gamma_{B,1}\gamma_{A,0}\gamma^2_{A,2}(\ln 2)^2\\
  &+ \delta \gamma_{B,1}\gamma_{A,1}\gamma_{A,3}\frac{\ln 2}{2}+ \delta \gamma_{B,1}\gamma_{A,1}\gamma_{A,3}\frac{\ln 2}{2}+\delta \gamma_{B,1}\gamma_{A,2}\gamma_{A,1}\frac{1}{2} \\
  & + \delta \gamma_{B,1}\gamma_{A,2}\gamma_{A,3}\frac{\ln 2}{2}+\delta \gamma_{B,1}\gamma_{A,3}\gamma_{A,1}\frac{\ln 2}{2}+\delta \gamma_{B,1}\gamma^2_{A,3}\frac{1}{4} \\
   \gamma_{A,2}\gamma_{B,2}\ln 2+ \gamma_{A,3}\gamma_{A,2}\frac{1}{2}& =\delta \gamma_{B,2}\gamma^2_{A,0}+\delta \gamma_{B,2}\gamma_{A,0}\gamma_{A,2}\ln 2+\delta \gamma_{B,2}\gamma_{A,1}\gamma_{A,0}\ln 2\\
   & + \delta \gamma_{B,2}\gamma_{A,1}\gamma_{A,2}\ln 2+\delta \gamma_{B,2}\gamma^2_{A,2}(\ln 2)^2 \\
   & +\delta \gamma_{B,2}\gamma_{A,3}\gamma_{A,0}(\ln 2)^2+\delta \gamma_{B,2}\gamma_{A,3}\gamma_{A,2}\frac{\ln 2}{2} \\
  \gamma_{A,2}\gamma_{B,1}\ln 2+ \gamma_{A,3}\gamma_{B,1}\ln 2 &= \delta \gamma_{B,1}\gamma_{A,0}\gamma_{A,1}\ln 2+\delta \gamma_{B,2}\gamma_{A,0}\gamma_{A,3}\ln 2\\
  & +\delta \gamma_{B,2}\gamma_{A,2}\gamma_{A,1}\ln 2+\delta \gamma_{B,2}\gamma_{A,2}\gamma_{A,3}(\ln 2)^2\\
  & + \delta \gamma_{B,2}\gamma_{A,3}\gamma_{A,1}(\ln 2)^2+\delta \gamma_{B,2}\gamma^2_{A,3}\frac{\ln 2}{2}
\end{array}
\right. $}
\end{equation*}
This system of equations has many solutions, for instance $\gamma_{A,1}=\gamma_{B,1}=0$, \\ $\gamma_{A,3}=-2\ln 2 \gamma_{A,2}$, $\gamma_{A,2},\ \gamma_{B,2}\in\mathbb{R}$. This corresponds to the following operators
\begin{eqnarray*}
  (Ax)(t)&=&\int\limits_{1}^{2} I_{[\alpha,\beta]}(t)\left(\frac{\gamma_{A,2}}{t}-\frac{2\ln 2\gamma_{A,2}}{ts} \right)x(s)d\mu_s,\\
   (Bx)(t)&=&\int\limits_{1}^{2} I_{[\alpha,\beta]}(t)\frac{\gamma_{B,2}}{t} x(s)d\mu_s,
\end{eqnarray*}
for almost every $t$.  These operators satisfy $AB=\delta BA^2=0$. Moreover, for all $x\in L_p(\mathbb{R},\mu)$, $1< p<\infty$ we have
       \begin{equation*}
         (AB-BA)x(t)=-(BAx)(t)=-\gamma_{A,2}\gamma_{B,2}\ln2 \int\limits_{1}^{2} I_{[\alpha,\beta]}(t)\frac{1}{t}\left(1-\frac{2\ln 2}{s}\right)x(s)d\mu_s,
       \end{equation*}
       for almost every $t$.  Since $\alpha <\beta $, by applying Lemma \ref{LemmaAllowInfSetsEqLp} we have that $A$ and $B$ commute if and only if $\gamma_{A,2}=0$ ($A=0$) or $\gamma_{B,2}=0$ ($B=0$).
}
\end{example}

\section*{Acknowledgments}
This work was supported by the Swedish International Development Cooperation Agency (Sida), bilateral program for capacity development in research and higher education in Mathematics in Mozambique. Domingos Djinja is grateful to
Mathematics and Applied Mathematics research environment MAM, Division of Mathematics and Physics, School of Education, Culture and Communication, M\"alardalen University for excellent environment for research in Mathematics.

%\printindex
%\addcontentsline{toc}{chapter}{Index}
%\printindex[aut]
%\addcontentsline{toc}{chapter}{Author Index}

\end{document}